\documentclass[12pt,a4paper,leqno]{amsart}
\usepackage[latin1]{inputenc}
\usepackage[T1]{fontenc}
\usepackage{amssymb,amsthm,amsmath,amsfonts,mathrsfs}
\numberwithin{equation}{section}
\usepackage{centernot}
\usepackage{graphicx}
\usepackage{empheq}
\usepackage{eqnarray}
\usepackage{geometry}
\usepackage{pdfpages}
\usepackage[colorlinks=true]{hyperref}
\usepackage{cleveref}

\DeclareMathOperator\dig{div}

\DeclareMathOperator\rota{rot}

\newcommand{\ee}{\varepsilon}
\newcommand{\ro}{\rho}
\newcommand{\n}{\nabla}
\newcommand{\pt}{\partial_t}
\newcommand{\pal}{\partial_\alpha}
\newcommand{\pbe}{\partial_\beta}
\newcommand{\sumk}{\sum_{k=1}^{n-1} }
\newcommand{\sumpa}{\sum_{\alpha=1}^d\pal }

\newcommand{\sumab}{\sum_{\alpha,\beta=1}^d}
\newcommand{\suma}{\sum_{\alpha=1}^d}
\newcommand{\Rd}{\mathbb{R}^d}
\newcommand{\C}{\mathbb{C}}
\newcommand{\cd}{\frac{d}{2}}

\newcommand{\esp}[1]{\quad \text{#1} \quad}
\newcommand{\defn}{\overset{\rm def}{=}}
\newcommand{\wt}{\widetilde}

\newcommand{\U}{\mathcal{U}}
\newcommand{\cO}{\mathcal{O}}
\newcommand{\cC}{\mathcal{C}}
\newcommand{\cM}{\mathcal{M}}

\newcommand{\R}{\dot{R}}

\newcommand{\Dj}{\Delta_j}
\newcommand{\DDj}{\Dot{\Delta}_j}

\newcommand{\intd}{\int_{\mathbb{R}^d}}
\newcommand{\dm}[1]{2^{-j#1}}

\newcommand{\normede}[1]{\left\Vert #1\right\Vert_{L^2}}
\newcommand{\normep}[1]{\left\Vert #1\right\Vert_{L^p}}

\newcommand{\normeinf}[1]{\left\Vert #1\right\Vert_{L^\infty}}
\newcommand{\BH}[3]{\dot{B}^{#1}_{#2,#3}}
\newcommand{\NBH}[4]{\left\Vert {#4} \right\Vert_{\dot{B}^{#1}_{#2,#3}}}
\newcommand{\NBHc}[1]{\left\Vert {#1} \right\Vert_{\dot{B}^{\frac{d}{2}}_{2,1}}}

\newcommand{\NB}[4]{\left\Vert {#4} \right\Vert_{B^{#1}_{#2,#3}}}
\newcommand{\LpNB}[5]{\left\Vert {#4} \right\Vert_{L^{#5}_T(B^{#1}_{#2,#3})}}

\newcommand{\LpNBH}[5]{\left\Vert {#4} \right\Vert_{L^{#5}_T(\Dot{B}^{#1}_{#2,#3})}}
\newcommand{\TLpNB}[5]{\left\Vert {#4} \right\Vert_{\widetilde{L}^{#5}_T(B^{#1}_{#2,#3})}}

\newcommand{\B}[3]{B^{#1}_{#2,#3}}

\newtheorem{thm}{Theorem}[section]
\newtheorem{dfn}{Definition}[section]
\newtheorem{propo}{Proposition}[section]

\newtheorem{rmq}{Remark}[section]

\newtheorem{lem}{Lemma}[section]
\definecolor{mycolor}{HTML}{D35400}

\definecolor{hide}{HTML}{A0AC81}
\definecolor{myred}{HTML}{8A1538}

\begin{document}
\title[Global existence for partially diffusive systems]{Global existence  for multi-dimensional partially diffusive systems 
}
\author{Jean-Paul Adogbo \& Raph\"{a}el Danchin }

\begin{abstract}
          In this work, we explore the global existence of strong solutions for a class of partially diffusive hyperbolic systems within the framework of critical homogeneous Besov spaces. Our objective is twofold: first, to extend our recent findings on the local existence presented in \cite{DanADOloc}, and second, to refine and enhance the analysis of Kawashima \cite{Kawashima83}.

To address the distinct behaviors of low and high frequency regimes, we employ a hybrid Besov norm approach that incorporates different regularity exponents for each regime. This allows us to meticulously analyze the interactions between these regimes, which exhibit fundamentally different dynamics.

A significant part of our methodology is based on the study of a Lyapunov functional, inspired by the work of Beauchard and Zuazua \cite{Beau11} and recent contributions \cite{BaratDan23, BaratDan22D1, BaratDan22M}. To effectively handle the high-frequency components, we introduce a parabolic mode with 
better smoothing properties, which plays a central role in our analysis.

Our results are particularly relevant for important physical systems, such as the magnetohydrodynamics (MHD) system and the Navier-Stokes-Fourier equations.
\end{abstract}
\maketitle
\section*{Introduction}
This work concerns  global-in-time strong solutions and the study of the large time asymptotics for the following class of $n\times n$ systems of PDEs: 
 \begin{equation}
	\label{system_diff}
	 S^0(U)\pt U +\suma{S^\alpha (U) \pal U} - \sumab{\pal (Y^{\alpha\beta}(U) \pbe U)}=f(U,\n U),
	\end{equation}
    with $\pt\defn \frac{\partial}{\partial_t}$ and for all $\alpha=1,\cdots, d$,  $\pal\defn \frac{\partial}{\partial x_\alpha} $.
 These systems  come into play  in the description of various  physical phenomena such as gas dynamics,  fluid models with chemical
reactions \cite{GioMat13}, fluids in the presence of a magnetic field \dots~ 
In System \eqref{system_diff} the smooth matrix-valued functions $S^\alpha$ ($\alpha=0,\cdots,d$), $Y^{\alpha\beta}$ ($\alpha,\beta=1,\cdots, d$) and vector-valued function $f$ are defined on some open subset $ \mathcal{U} $ of $\mathbb{R}^n$ ($n\ge 1$) and the unknown $U\in \mathbb{R}^n$  depends on the time variable $t\in \mathbb{{R}}_+$ and on the space variable $x\in \Rd$ with $d\geq2.$ Additional structure assumptions will be specified in the next section.
\medbreak
It is well known that classical systems of conservation laws (that is, \eqref{system_diff} with $f\equiv 0$ and $Y^{\alpha\beta}\equiv 0 $) supplemented with smooth data admit local-in-time strong solutions that may develop singularities (shock waves) in finite time even if the initial data are small perturbations of a constant solution (see, for instance, the works by Majda in \cite{Madja84}  and D. Serre in \cite{Serre97}). The picture changes when the second order operator in \eqref{system_diff} is strongly elliptic. Then, the global existence for
small and sufficiently regular perturbations of a constant solution $\overline{U} \in \mathcal{U}$
holds, as well as the convergence of the solution $U$ to $\overline{U}$ with the same decay rate as for the heat equation (see \cite{Kawashima83, SK85}).
\medbreak

In many physical systems however, that condition is not verified, and it is more reasonable to assume  that
up to change of coordinates, the symbol $Y(\omega, U)\defn \sum_{\alpha,\beta=1}^dY^{\alpha\beta}(U)\omega_\alpha\omega_\beta$ is a block-diagonal and nonnegative  matrix:
\begin{equation}    \label{structure:Y:alp:bet}
    Y(\omega, U)=
    \begin{pmatrix}
        0& 0\\
        0& Z(\omega, U)
    \end{pmatrix},\qquad\omega\in {\mathbb S}^{d-1},\quad U\in{\mathcal U},
\end{equation}
where $Z(\omega, U)$ is a $n_2\times n_2$ positive definite matrix, and $ 1\le n_2\le n$.
In other words, $Y(\omega, U)$ acts on some components of the unknown, while the other components are unaffected since they satisfy conservation laws. An informative example is the MHD system, where the mass is conserved (see Section \ref{sec:MHD}).  As remarked by A. Matsumura and T. Nishida in their seminal paper \cite{MatsmuraNishida80} dedicated to the equations of motion of viscous and heat conductive
gases, this indirect effect may be sufficient to ensure the global existence for small regular initial data. The resulting class of systems is named, depending on the authors and on the context, hyperbolic-parabolic or hyperbolic partially diffusive. 
\smallbreak
Hyperbolic  partially diffusive systems have been extensively studied since the pioneering work by S. Kawashima in 1983 in his PhD thesis \cite{Kawashima83} which is  the cornerstone of all theory. There, he pointed out sufficient conditions for local well-posedness for general (smooth enough) data, global existence for small data, and large time asymptotics. 
He also pointed out a condition on the structure of the matrices and functions involved in System \eqref{system_diff} allowing one to obtain the \textit{normal} form of the System \eqref{system_diff}. This condition enabled the author to prove the local well-posedness of the system \eqref{system_diff} for initial data belonging to $ H^s(\Rd) $, with $s>\cd+2$. 

A few years later, S. Kawashima and S. Shizuta in \cite{KawaSui88} exhibited two sufficient conditions to derive the normal form of system \eqref{system_diff}: the first one is the existence of a \textit{dissipative  entropy} which provides a symmetrization of the system that is compatible with the second order terms, the second condition stipulates that the null space of $ Y(\omega, U)$ is independent of $U\in \mathcal{U}$ and $\omega\in \mathbb{S}^{d-1}$ (condition N in \cite{KawaSui88}). Later, D. Serre in \cite{Serre09} pointed out that this second condition is equivalent to the range of $Y(\omega, U)$ being independent of $U \in \mathcal{U}$ and $\omega \in \mathbb{S}^{d-1}$, which is simpler to verify.  
 In \cite{Serr10}, he further improved Kawashima's local result from \cite{Kawashima83} by enlarging  the class of initial data to include those from inviscid conservation laws, specifically $ H^s(\Rd) $, with $s>\cd+1$. 
 \medbreak
 In \cite{DanADOloc}, we established the local well-posedness for System \eqref{system_diff} supplemented with initial data in nonhomogeneous Besov space of type
$B^{\sigma}_{2,1}.$ More precisely, the component of the initial data affected by the diffusion belongs to $B^s_{2,1}$ while the unaffected component is taken in $B^{s+1}_{2,1}.$ 
The regularity index can be any real number $s\geq\cd$
(and even $s\geq \cd-1$ for a particular class of systems) so that we can
consider initial data which are not Lipschitz. 
Let us also emphasize that our structural assumptions are weaker than those made in \cite{Serr10}. 
\medbreak
 In  \cite{Kawashima83}, the author exhibited  a rather simple sufficient condition  for global well-posedness
 of \eqref{system_diff} supplemented with initial data  which are  
 in the neighborhood of linearly stable constant solutions.
   It is nowadays called the (SK) (meaning Shizuta-Kawashima) condition. In the case where $S^0$ is the  Identity matrix,  it  exactly says that for the linearized system, the intersection between the kernel of the second order term (the symbol $Y(\cdot,\overline{U})$) and the set of all eigenvectors of the symmetric first order term is reduced to $\{0\}$ (see Section \ref{sec:2} for more details). 
 In \cite{Kawashima83,KawaSui88}, the equivalence between Condition (SK) and the existence of a \emph{compensating function} (that allows
 to work out a Lyapunov functional equivalent to a suitable Sobolev
 norm of the solution)  enabled the authors  to exhibit the global-in-time $L^2$ integrability properties of all the components of the solution and to get global-in-time solutions for initial data belonging to $H^s(\Rd)$ with $s>3+\cd\cdotp$ 
\medbreak
In \cite{SK85}, S. Shizuta and S. Kawashima pointed out that the existence of a compensating function is equivalent to the strict
dissipativity of the system. This means that, in Fourier space, the real parts of all eigenvalues of the matrix of the linearized system of \eqref{system_diff} around the reference solution are strictly positive. 
 In the same paper, the authors proved that if, in addition of being in a Sobolev space $H^s(\Rd)$ with large enough $s$, we assume the initial data $U_0$
to satisfy $(U_0-\overline U)\in L^p$ for some $1\leq p<2,$ then
  the decay rate of  $U(t)-\overline{U}$ in $L^2$ for $t$ going to $\infty$ is the same as for the heat equation, namely $(1+t)^{-\frac{d}{2}(\frac{1}{p}-\frac{1}{2})}.$ 
Since
then, more decay estimates have been proved under various assumptions; see, for example \cite{BianHanou07,XinXU21,XuKawa15}.
\medbreak
The condition (SK) is not sharp: several authors observed that it is not necessary for the existence of global strong solutions. For instance, in \cite{QuWang18}, P. Qu and Y. Wang established a global existence result in the case where exactly one eigenvector violates the condition (SK). In this respect, one can also mention the paper by C. Burtea, T. Crin-Barat, and J. Tan \cite{BurCrinTan23} dedicated to the relaxation limit of the Baer-Nunziato (BN) system, which does not verify the (SK) condition, and the recent work \cite{DanMucha24} by the second author and P. B. Mucha dedicated to the mathematical study of compressible Euler system with nonlocal pressure.

 Although Condition (SK) provides a compensating function that ensures global existence, along with estimates on the entire solution and its decay, it is difficult to verify for concrete systems. Another limitation is that it does not provide more specific information on the part of the solution experiencing diffusion, even though it is expected to have better properties than the whole solution (see \cite{Dan01glob} for the case of the compressible Navier-Stokes equations).

For partially dissipative systems,  K. Beauchard and E. Zuazua \cite{Beau11} pointed out 
a link between the (SK) condition and the Kalman maximal rank condition 
from control theory of ODEs (see  \cite{coron2007control, Zua05}) for the linearized system.
 They proved that for a symmetric hyperbolic partially dissipative system, the (SK) condition is equivalent to the Kalman maximal rank condition for the linearized system about a constant solution. They also introduced a Lyapunov functional that encodes sufficient information to recover most of the dissipative properties of the linearized system.
 This equivalence was generalized to any linear system in \cite{Danc22}. In \cite{BaratDan22M,BaratDan22D1,BaratDan23}, 
  T.~Crin-Barat and the second author leveraged the method of Beauchard and Zuazua to establish global existence and relaxation results in a \textit{critical regularity setting} for symmetric hyperbolic partially dissipative systems with data having different regularities in low and high frequencies.  An important new ingredient for achieving these improvements is to examine the time evolution of a \textit{damped mode} which corresponds to the part of the solution that experiences maximal dissipation at low frequencies.
\medbreak
The primary aim of the present work is to extend the results of Kawashima by broadening the class of initial data and weakening the smallness condition. 
Following our recent work \cite{DanADOloc} dedicated to the local well-posedness
and adapting the approach developed in  \cite{BaratDan22M} to the hyperbolic-parabolic setting, 
we shall  propose  a   `critical' functional framework that offers better control over the solution and, in particular, more accurate convergence rates to the equilibrium state as time goes to infinity. 
  Specifically, we will prove the global existence for small initial data in functional spaces with \emph{different} regularity exponents at low and high frequencies. 
  As for partially dissipative systems, Beauchard and Zuazua's approach will give us the information that the low frequencies (resp. high frequencies) of the solution of the linearized system behave like the heat flow (resp. are exponentially damped). 
  To achieve maximal regularity for the component affected by diffusion, we will introduce the \textit{parabolic mode}, analogous to the damped mode in the hyperbolic partially dissipative case \cite{BaratDan22M}.
\medbreak
This work is arranged as follows. The first section is dedicated to presenting the structure of the class of symmetric hyperbolic partially diffusive systems we shall consider, 
constructing a low and high frequencies Lyapunov functional that will play a key role in our global results and exhibiting a \textit{parabolic mode} that
will allow us to recover the full regularizing effect for  
the components of the solution experiencing the diffusion. In Section \ref{subs:main:resu:NHHB}, we state the main results of the paper: global existence and decay estimates.  Proving a first global existence result and time decay estimates for general partially diffusive systems satisfying the (SK) condition will be the goal of Section \ref{sec:proof:global:exis}. In Section \ref{sec:proof:global:crit}, under additional assumptions of structure (which are satisfied by the compressible Navier-Stokes system), we obtain a global existence result
at the critical regularity level. In the last section, we apply the results to the MHD system. Some technical results are proved or recalled in Appendix. 

\medbreak\noindent{\bf Notation.} As usual, $C$ designates a generic harmless constant, the value of which depends on the context. 
We sometimes use the short notation $A\lesssim B$ to mean that $A\leq CB.$

For any normed space $X,$ real number $\ro \in [1,\infty]$ and $T\in[0,\infty],$
we denote by $L^\rho(0,T;X)$ (or sometimes just $L^\rho_T(X)$) the set
of measurable functions $z:(0,T)\to X$ such that $\Vert z\Vert_{L^\ro_T(X)}= \Vert \Vert z(t,\cdot)\Vert_{X} \Vert_{L^\ro(0,T)}$ is finite. The notation $\cC_b(\mathbb R_+;X)$ stands for the set of continuous bounded functions from $\mathbb R_+$ to $X.$


\section{Assumptions and approaches}\label{sec:2}
In this section, we specify the class of partially diffusive systems we aim
to study, and present the tools allowing to get the main results.

\subsection{Normal form}

Let $\overline{U} \in \mathcal{U}$ be a constant solution of \eqref{system_diff}. 
Then, $V\defn U-\overline{U}$ satisfies 
  \begin{equation}	\label{Eq_b}
	 S^0(U)\pt V +\sum_{\alpha}{S^\alpha (U) \pal V} -\sum_{\alpha,\beta}{ \pal (Y^{\alpha\beta}(U) \pbe V)}=f(U,\n U).\end{equation}
In what follows, we assume that the matrices of System
\eqref{system_diff} and the external force $f$ satisfy:
 \medbreak\paragraph{\bf Assumption D}
 \begin{itemize}
    \label{cond:D}
    \item[\textbf{(D1)}] The matrix $S^0(U)$ is symmetric, positive definite  for every $U\in \mathcal{U}$, and has the form
$$ S^0(U)=
      \begin{pmatrix}
S^{0}_{11}(U) & 0\\
0 &  S^0_{22}(U)
\end{pmatrix},$$
where $ S^{0}_{11}(U) $ (resp. $S^{0}_{22}(U)$) is a $n_1\times n_1$ (resp. $n_2\times n_2$) 
matrix, with $n_1+n_2=n$.\smallbreak
      \item[\textbf{(D2)}] \label{cond:DD:2} The matrices $S^\alpha(U)$ are symmetric for every $U\in \mathcal{U}$ and $\alpha=1,\cdots d$.\smallbreak
      \item[\textbf{(D3)}] For every $U\in \mathcal{U}$ and $\alpha,\beta=1,\cdots d$, the matrices $Y^{\alpha\beta}(U)  $ have the form 
      \begin{equation}\label{def:Y:al:be}
       Y^{\alpha\beta}(U)=\begin{pmatrix}
0_{n_1} & 0\\
0 &  Z^{\alpha\beta}(U)
\end{pmatrix},
  \end{equation}
  where $  Z^{\alpha\beta}(U)$ is a $n_2\times n_2$ matrix.
     Furthermore,  the operator  $\sumab Z^{\alpha\beta}(U)\pal\pbe$ is strongly elliptic in the sense:
	      \begin{equation}
	    \label{strong_elli}
	    \sum_{\alpha,\beta=1}^{d} \sum_{i,j\ge n_1+1}{\xi_\alpha\lambda_i\xi_\beta\lambda_jZ_{ij}^{\alpha\beta}  (U) }  \ge c_1(U) |\xi|^2|\lambda|^2, \quad\forall \xi \in \mathbb{R}^d, \lambda\in \mathbb{R}^{n_2},
	    \end{equation}
 for some positive continuous function $c_1$ defined on ${\mathcal U}$.\smallbreak
      \item[\textbf{(D4)}] $f$ has components 
      \begin{equation}
\label{globa:f=f(f1,f2)}
      f(U,\n U)\defn \begin{pmatrix}
f^1(U,\n U)\\
f^2(U,\n U)
\end{pmatrix}
      \defn
\begin{pmatrix}
0_{n_1}\\
Q(U,\n U,\n U)
\end{pmatrix}
  \end{equation}
  with  $Q$  quadratic in terms of $\n U$ (i.e.   $Q$  is a combination of terms  $\upsilon(U)\n U\!\bigotimes\!\n U$).
  \end{itemize}
System \eqref{system_diff} under assumptions (i), (ii), (iii)
and (iv) is said to be of the normal form (\cite{KawaSui88,Serre09}).
\begin{rmq}
    \label{rmq:f:gene}
    The assumptions on $f$ depend on the regularity framework: in \cite{Kawashima83,SK85}, since the solution is more regular, $f^1$ (resp. $f^2$) can be any function of 
    $U,\nabla U^2$ (resp. $U,\nabla U$).
    
Back to our setting, Theorems \ref{thm:glob:cri} and \ref{thm:decay:sucri} below are in fact valid whenever
    \begin{enumerate}
        \item $f^1= f^1(U,\n U^2)$ is a combination of terms 
        $L^1V^1 + \upsilon_1(U)V\bigotimes V +\upsilon_2(U) V\bigotimes \n V^2$, with
     $L^1$ a nonpositive matrix,    
                and        $\upsilon_1, \upsilon_2$ some smooth functions,
        \item  $f^2$ is a combination of terms of type 
        $L^2V^2 + \upsilon_3(U,\n V^1)\n V\bigotimes\n V+ Q_2(V,\n V^1) $, where  
        $L^2$ is a nonpositive  matrix,   
        $Q_2$ is quadratic  and $\upsilon_3$ is a smooth function.  
    \end{enumerate}
\end{rmq}
 
     
  \subsection{The Shizuta-Kawashima and Kalman rank conditions}
 \label{subs:SK:Kalman}
Let us consider the  following  linearization  of \eqref{Eq_b}:
\begin{equation}	\label{eq:lindiff0}
\overline{S}^0\pt V+  \sum_{\alpha=1}^d\overline{S}^\alpha \pal V-\sum_{\alpha,\beta=1}^d\overline{Y}^{\alpha\beta} \partial_{\alpha}\partial_{\beta}V=F,
	\end{equation}
where, from now on, we use  for any function $S$ the notation  $\overline{S}\defn S(\overline{U}).$
Note that \eqref{Eq_b} corresponds to \eqref{eq:lindiff0}
with  $F\defn F^1+F^2+F^t+f,$ where 
\begin{equation}\begin{aligned}
    \label{F_T:F_1:F_2}
            &F^t\defn(\overline{S}^0-S(U))\pt V,\  &&F^1\defn(\overline{S}^\alpha-S^\alpha(U)) \sumpa  V,\\
        &F^2\defn\sumab \pal(r^{\alpha\beta}\pbe V) \esp{ and} &&r^{\alpha\beta}\defn(\overline{Y}^{\alpha\beta}-Y^{\alpha\beta}(U)).
    \end{aligned}\end{equation}
     Let $\xi\in\Rd$ stand for the Fourier variable
corresponding the space variable $x\in\Rd$. 
In the Fourier space, System \eqref{eq:lindiff0} recasts into
\begin{equation}
	\label{eq:lindiff}
\overline{S}^0\pt \widehat{V}+  i\sum_{\alpha=1}^d\overline{S}^\alpha \xi_\alpha \widehat{V}+ \sum_{\alpha,\beta=1}^d\overline{Y}^{\alpha\beta}\xi_\alpha\xi_\beta\widehat{V}=\widehat{ F}.
	\end{equation}
 Set $\xi=\rho\omega$ where $\ro\ge 0$ and $\omega\in \mathbb{S}^{d-1}$. We define
 \begin{equation}
     \label{A:B_ome:def}
A_\omega\defn i\sum_{\alpha=1}^d\overline{S}^\alpha \omega_\alpha
 \esp{ and }
B_\omega\defn\sum_{\alpha,\beta=1}^d\overline{Y}^{\alpha\beta}\omega_\alpha\omega_\beta= \sum_{\alpha,\beta=1}^d
\begin{pmatrix}
0_{n_1\times n_1} & 0_{n_1\times n_2}\\
0_{n_2\times n_1} &  \overline{Z}^{\alpha\beta}\omega_\alpha\omega_\beta
\end{pmatrix}\cdotp
 \end{equation}
 \smallbreak
  With the above notation, taking $F=0$ to simplify the presentation, System \eqref{eq:lindiff0} becomes
 \begin{equation}
	\label{eqs:lindiff}
\overline{S}^0\pt \widehat{V}+  (\rho A_\omega+\rho^2B_\omega) \widehat{V} =0.
	\end{equation}
Under Assumption \textbf{D}, we observe that for all $\omega \in \mathbb{S}^{d-1}$, the matrix $ A_\omega$ is skew Hermitian, while the matrix $ B_\omega$ is nonnegative.
\medbreak
System \eqref{eq:lindiff0} is a particular case of the more general 
class of  systems:
 \begin{align}
     \label{lin:eq:gene}
     S\pt V+\mathcal{A}(D)V+ \mathcal{B}(D)V=0, \esp{where}
 \end{align}
 \begin{itemize}
 \item $  S$ is a Hermitian positive definite matrix.
     \item $A(D)$ is a homogeneous (matrix-valued) Fourier multiplier of degree $a$ that satisfies
     \begin{align}
         \label{null:A:sky}
         \mathcal{R}e(\mathcal{A}(\omega)\eta\cdot \eta)=0, \; \forall \omega \in \mathbb{S}^{d-1} \esp{and} \eta\in \mathbb{C}^n
     \end{align}
where $\cdot$ designates the Hermitian scalar product in $ \mathbb{C}^n $.
\item  $B(D)$ is an homogeneous (matrix-valued) Fourier multiplier of degree $b$, such that,
for some positive real number $\kappa$,
\begin{align}
         \label{posi:ope:B:}
         \mathcal{R}e(\mathcal{B}(\omega)\eta\cdot \eta)\ge \kappa |\mathcal{B}(\omega)\eta|^2, \; \forall \omega \in \mathbb{S}^{d-1} \esp{and} \eta\in \mathbb{C}^n.
     \end{align}
 \end{itemize}
 Observe that System \eqref{eqs:lindiff} corresponds to the case $a=1$ and $b=2$.
\medbreak
 Taking the real part of the inner product of \eqref{lin:eq:gene} with $\widehat{V}$, 
 then using \eqref{null:A:sky} yields
 \begin{equation}
     \label{eq:classic_ener}
  \frac12    \frac{d}{dt}\left( S \widehat{V}(\xi)\cdot\widehat{V}(\xi)\right) 
  + \mathcal{R}e\Bigl(\mathcal{B}(\xi)\widehat{V}(\xi)\cdot\widehat{V}(\xi)\Bigr) =0
  \esp{for all}\xi \in \Rd.
 \end{equation}
On the one hand, since the matrix $S$ is Hermitian positive definite, we have
\begin{equation}
    \label{equi_nor2_V}
     S \widehat{V}\cdot\widehat{V} \simeq |\widehat{V}|^2.
\end{equation}
On the other hand, the positivity \eqref{posi:ope:B:} of $ \mathcal{B}$   guarantees that 
\begin{equation}
    \label{elliall_V_b}
    \mathcal{R}e\bigl(  \mathcal{B}(\xi)\widehat{V}(\xi)\cdot\widehat{V}(\xi)\bigr) \ge \kappa |\xi|^{-b} |\mathcal{B(\xi)}\widehat{V}|^2=\rho^b|\mathcal{B}_\omega\widehat{V}|^2,
\end{equation}
where, for  $\omega\in {\mathbb S}^{d-1},$ we denoted
 \begin{align}     \label{def:A:B:N:M_omega}
     \mathcal{A}_\omega\defn  \mathcal{A}(\omega)\esp{and}
 \mathcal{B}_\omega\defn  \kappa\mathcal{B}(\omega).
 \end{align}
Hence \eqref{eq:classic_ener} finally gives us some constant 
$C_S$ depending only on $S$ such that for  all $t\in{\mathbb R}_+,$
\begin{equation}
\label{classical:ener}
     \normede{\widehat V(t)}^2+\int^t_0 \intd |\xi|^{b}|\mathcal{B}_\omega \widehat{V}|^2d\xi\le C_S\normede{\widehat V_0}^2.
\end{equation}

 If, for some $\omega\in \mathbb{S}^{d-1},$  the matrix $\mathcal{B}_\omega$ has a strictly lower rank than $n$, then the last inequality does not ensure dissipation on all components of $V$. We are going to make some assumptions that allow us to recover the decay of the missing components of the solution. Following Beauchard and Zuazua in \cite{Beau11}; Crin-Barat and the second author in \cite{BaratDan22M,Danc22}, we introduce a Lyapunov functional to track the diffusion of the solution to \eqref{eqs:lindiff}.
We have the following result (we postpone the proof in the appendix):
\begin{lem}
\label{lem:derI}
Let  $\mathcal{N}_\omega= S^{-1}\mathcal{A}_\omega$ and $\mathcal{M}_\omega=S^{-1} \mathcal{B}_\omega$ for all $\omega\in\mathbb S^{d-1}.$
There exist positive parameters $\varepsilon_0,\cdots\varepsilon_{n-1}$ (that are defined inductively and can be taken arbitrarily small) and a Lyapunov functional 
\begin{multline}
    \label{def:Lya:r:ome}
    L_{\ro,\omega}(\widehat{V})\defn  S \widehat{V}\cdot\widehat{V}
   \! +\!\min\Big( \frac{1}{\kappa\ro^{a-b}},\kappa\rho^{a-b}\Big) \mathcal{I}_\omega\\\esp{with}\!\! \mathcal{I}_\omega\defn
    \sumk \ee_k   \mathcal{R}e (\mathcal{M}_\omega \mathcal{N}_\omega^{k-1} \widehat{V} \cdot \mathcal{M}_\omega \mathcal{N}_\omega^{k} \widehat{V})
\end{multline}
such that the following inequalities  hold for some positive $c$ and $C,$
and all $\ro>0$ and $\omega\in{\mathbb S}^{d-1}$:
\begin{align}
    \label{dt:I}
     &\frac{d}{dt} \mathcal{I}_\omega+ \frac{1}{2}\ro^a \sum_{k=1}^{n-1}\varepsilon_k |\mathcal{M}_\omega \mathcal{N}_\omega^{k} \widehat{V}|^2\le C\ee_0 \max\Big(\ro^a, \frac{\ro^{2b-a}}{\kappa^2}\Big)  |\mathcal{M}_\omega\widehat{V}|^2,\\
    \label{dt:L_xi}
   & \frac{d}{dt} L_{\ro,\omega}(\widehat{V})+c\min(\frac{\ro^{b}}{\kappa}, \kappa \ro^{2a-b}) 
   \biggl(|\mathcal{M}_\omega \widehat V|^2+
   \sum_{k=1}^{n-1}\varepsilon_k |\mathcal{M}_\omega \mathcal{N}_\omega^{k} \widehat{V}|^2\biggr)\le 0,
\end{align}
with, additionally, 
\begin{align}    \label{eqi:lya:ro:omega}
    C^{-1} | \widehat{V}|^2\le  L_{\ro,\omega}(\widehat{V})\le C| \widehat{V}|^2.
\end{align}
\end{lem}
  
 The question now is whether the rate of dissipation in \eqref{dt:L_xi} can be compared to $|\widehat{V}|^2$. We get a positive answer if we work under the following (SK) (for Shizuta and Kawashima) condition.  Let us recall this condition stated in \cite{SK85}.
\begin{dfn}\label{dfn:SK}
Let $n\in \mathbb{N}$, $N, M$  be $n\times n$ matrices with complex coefficients.
The pair $(N, M)$ is said to satisfy  the (\text{SK}) condition  if we have at the same time $M \phi =0 $ and $\lambda \phi+N \phi =0$ for some $\lambda\in \mathbb{R}$, if and only if $\phi=0_{\mathbb{C}^n}$.
\end{dfn}

  In \cite{Beau11},  K. Beauchard and E. Zuazua have highlighted an interesting connection between Condition (SK) and the Kalman criterion for observability in the theory of linear ODEs, that is recalled in the following lemma.
\begin{lem}
    \label{lem:SK_diss}
    Let $N$ and $M$ be two $n\times n$ complex valued matrices. 
        Then, the following properties are equivalent:
        \begin{enumerate}
            \item For all positive $\ee_0,\dots \ee_{n-1}$, we have $\sum_{l=0}^{n-1}\ee_l|MN^l\eta|^2>0$ for all $\eta\in  \mathbb{S}^{n-1}$.
            \item $(N, M)$ satisfies the Kalman rank condition: the rank of the $n^2\times n$ matrix $
            \begin{pmatrix}
                &M\\
                &MN\\
                &\dots\\
                &MN^{n-1}
            \end{pmatrix}
            $
            is equal to $n$.
            \item $(N, M)$ satisfies Condition (SK).
        \end{enumerate}
        \end{lem}

      It turns out that this connection had been noticed before by D. Serre in Chapter 6 of \cite{Serre08cours}. Here, Condition (SK) becomes apparent when investigating the question of convergence toward the equilibrium state of the system \eqref{lin:eq:gene}. More precisely, relying on iterated commutators (\`a la H\"{o}rmander in his theory of the hypoellipticity) between the hyperbolic and dissipative parts, D. Serre justifies, 
      in
      the one-dimensional case, that $\mathfrak{N}_\omega=\{0\}$ (with  $\mathfrak{N}_\omega$ being the intersection of the kernels of $\mathcal{M}_\omega, \mathcal{M}_\omega \mathcal{N}_\omega,\cdots, \mathcal{M}_\omega \mathcal{N}_\omega^{n-1}$) is a necessary condition for decay, which is indeed equivalent to Condition (SK) and the Kalman rank condition. 
 
    Still considering  System \eqref{lin:eq:gene}, assuming that the (SK) condition
    is satisfied by $(\mathcal{N}_\omega, \mathcal{M}_\omega)$  at every 
    point of the unit sphere $ \mathbb{S}^{d-1}$, we gather
from the compactness of $ \mathbb{S}^{d-1},$   \eqref{dt:L_xi} 
and \eqref{eqi:lya:ro:omega} that there exists a constant $ c_*>0 $ such that
        \begin{equation}
    \label{1:dt:L_ro_om}
    \frac{d}{dt}L_{r,\omega}(\widehat{V})+c_* \min\Bigl(\frac{\ro^{b}}{\kappa}, \kappa\ro^{a-b}\Bigr) L_{r,\omega}(\widehat{V}) \le 0.
\end{equation}
This implies, owing to \eqref{eqi:lya:ro:omega}, that 
 \begin{align}     \label{V(xi)=eV_0}
     \widehat{V}(t,\xi)\le C e^{-ct\min(\frac{|\xi|^{b}}{\kappa}, \kappa|\xi|^{a-b}) }  \widehat{V_0}(\xi),\quad t\in\R_+,\quad \xi\in\R^d \end{align}
which reveals a  different behavior of the solution in high and low frequencies.
In particular, as it corresponds to the case $a=1$ and $b=2,$ 
 the linear system \eqref{eq:lindiff0} behaves as a parabolic system for low frequencies, and has a damped behavior (but no gain  of  derivative)  for high frequencies.

Note also that  integrating \eqref{1:dt:L_ro_om} on $\Rd,$ omitting the second term and
 using Fourier-Plancherel theorem ensures that 
$${\mathcal L}(V)\defn (2\pi)^{-d}\int_0^\infty\!\!\int_{\mathbb S^{d-1}} L_{r,\omega}(\widehat{V})
\, r^{d-1} \, d\sigma(\omega)\,dr$$
 is  a Lyapunov functional for System \eqref{lin:eq:gene}, that is equivalent 
 to $\|V\|_{L^2}^2.$

\subsection{Parabolic mode}\label{para:mode:subsec}

Observe that under Assumption \textbf{D}, System \eqref{Eq_b} can be rewritten: 
 \begin{equation}
     \label{Eq_b:V1:V2}
     \begin{cases}
      S^0_{11}(U) \pt V^1 +\sum_{\alpha=1}^d\left(S^\alpha_{11} (U) \pal V^1+S^\alpha_{12} (U) \pal V^2\right) =0 \\[0,2cm]
     S^0_{22}(U) \pt V^2 +\sum_{\alpha=1}^d\left(S^\alpha_{21} (U) \pal V^1+S^\alpha_{22} (U) \pal V^2\right)-\sum_{ \substack{\alpha,\beta=1}}^d{ \pal (Z^{\alpha\beta}(U) \pbe V^2)}\\
     \hspace{12cm}=  f^2(U,\n U)
     \end{cases}
 \end{equation}
  with $V^1:\mathbb{R}_+\times \mathbb{R}^d\to \mathbb{R}^{n_1}$ and 
  $V^2:\mathbb{R}_+\times\mathbb{R}^d\to \mathbb{R}^{n_2}$
  \begin{align}
      \label{def:matrice:bloc}
      S^0(U)=
      \begin{pmatrix}
S^{0}_{11}(U) & 0\\
0 &  S^0_{22}(U)
\end{pmatrix}
\esp{and}
S^\alpha(U) =
\begin{pmatrix}
S^\alpha_{11}(U) & S^\alpha_{12}(U)\\
S^\alpha_{21}(U) &  S^\alpha_{22}(U)
\end{pmatrix}\cdotp
  \end{align}
Owing to Assumption \eqref{strong_elli}, for fixed $V^1,$ the second equation of \eqref{Eq_b:V1:V2} is of parabolic type. 
Like for the heat equation, one can thus expect a gain of two space derivatives 
for $V^2,$ compared to e.g. the source term $f^2.$
From Beauchard-Zuazua's analysis however, we cannot track this smoothing
property: we only get exponential decay for high frequencies of $V^2$.
 
 Introducing a  suitable ``parabolic mode'' will help us to recover this information. 
 To do so, let us  isolate the linear part of the second equation of \eqref{Eq_b:V1:V2}, and rewrite it as:
\begin{align}
    \label{lin:part:V2}
    \begin{split}
         \overline{S}^0_{22}\pt V^2+ \sum_{\alpha}\overline{S}^\alpha_{21} \pal V^1+ \sum_{\alpha}\overline{S}^\alpha_{22} \pal V^2  - \sumab\overline{Z}^{\alpha \beta}\pal\pbe V^2=h
    \end{split}
\end{align}
with $h=h^t+h^{22}+h^{21}+h^2+f^2$ and 
\begin{equation}
    \label{def:h}
    \begin{aligned}
        & h^t\defn (\overline{S}^0_{22}-S^0_{22}(U))\pt V^2;\; &&\quad h^{21}\defn \sum_{\alpha}(\overline{S}^\alpha_{21}-S^\alpha_{21}(U) )\pal V^1;\; \\
    &h^{22}\defn\sum_{\alpha}(\overline{S}^\alpha_{22}- S^\alpha_{22}(U))\pal V^2; \;
    &&\quad h^2\defn \sum_{\alpha,\beta}\pal(r^{\alpha\beta}\pbe V^2).
    \end{aligned}
\end{equation}
In the Fourier space, if we set $\xi=\rho\omega$ with $\rho\ge 0$ and $\omega\in \mathbb{S}^{d-1}$,  we can rewrite \eqref{lin:part:V2} as follows:
\begin{align}
    \label{l:pa:Fou:V2}
    \begin{split}
         \overline{S}^0_{22}\pt \widehat{V}^2+ i\rho( S_{21}(\omega) \widehat{V}^1+  S_{22}(\omega) \widehat{V}^2)  + \rho^2 Z(\omega) \widehat{V}^2=\widehat{h}\end{split}
\end{align}
where 
\begin{align}
    \label{def:S:21:22:Z}
     S_{21}(\omega)\defn  \sum_{\alpha}\overline{S}^\alpha_{21}\omega_\alpha,\;\; S_{22}(\omega)\defn  \sum_{\alpha}\overline{S}^\alpha_{22} \omega_\alpha \esp{and} Z(\omega)\defn \sumab\overline{Z}^{\alpha \beta}\omega_\alpha\omega_\beta. 
\end{align}
After extending the above functions on $\Rd\setminus\{0\}$ by homogeneity of degree $0$, 
we set:
\begin{align}
    \label{def:S:21:22:ZD}
    S_{21}(D) z \defn \mathcal{F}^{-1}(S_{21} \widehat{z}),\;\;  S_{22}(D) z \defn \mathcal{F}^{-1}(S_{22} \widehat{z}) \esp{and} Z(D) z \defn \mathcal{F}^{-1}(Z\widehat{z}).
\end{align}
 For all $\omega\in \mathbb{S}^{d-1},$  $Z(\omega)$ is  an invertible matrix (in fact from \eqref{strong_elli} the Kernel of $ Z(\omega)\in \mathcal{M}_{n_2}({\mathbb C}) $ is ${0_{{\mathbb C}^{n_2}}}$), we define $W$ from its Fourier transform as follows: 
\begin{align*}
    \widehat{W}(\xi)&\defn i\rho^{-1}(Z(\omega)^{-1}( S_{21}(\omega) \widehat{V}^1+  S_{22}(\omega)\widehat{V}^2)  +\widehat{V}^2\\&= \rho^{-2} (Z(\omega))^{-1} 
    \bigl(\widehat{h}-    \overline{S}^0_{22}\pt \widehat{V}^2\bigr)\cdotp
\end{align*}
In other words,  we have
\begin{align}
    \label{W:def}
    \begin{split}
    W &\defn i|D|^{-1}(Z(D))^{-1}\left( S_{21}(D) V^1+  S_{22}(D)V^2\right)  +V^2\\
    &= (-\Delta)^{-1} (Z(D))^{-1} \bigl(h-\overline{S}^0_{22}\pt V^2\bigr)\cdotp
    \end{split}
\end{align}
Then, we discover that $W$ satisfies the following parabolic equation
\begin{equation}
    \label{widehatW:eq}
    \overline{S}^0_{22}\pt W -\Delta Z(D) W= i |D|^{-1}\overline{S}^0_{22}(Z(D))^{-1}\pt\left( S_{21}(D) V^1+  S_{22}(D) V^2 \right)+h,
\end{equation}
and thus  expect $W$  to have better regularity properties in high frequencies than the whole solution.

Although a similar idea has been used before in the context of the compressible Navier-Stokes equations (see \cite{Haspot11} and the references therein), or for
particular cases of symmetric hyperbolic partially dissipative system (see \cite{BaratDan22M, lemarie2023parabolic})  it seems to be new for the general class that is considered here.


\subsection{Transfer of integrability  from  the parabolic to the  hyperbolic mode}\label{subsec:V1:hf}
Inserting $V^2= W-i|D|^{-1}(Z(D))^{-1}\left( S_{21}(D) V^1+  S_{22}(D)V^2\right) $ in the 
first equation of \eqref{Eq_b:V1:V2}, we observe that $V^1$ is solution of the following symmetric hyperbolic system:
\begin{align}
    \label{eq:V1=f(W)}
     S^0_{11}(U) \pt V^1 +\sum_{\alpha=1}^dS^\alpha_{11} (U) \pal V^1+ S_{12}(D)Z(D)^{-1}S_{21}(D)V^1 =G^{11}+G^{12}+G^{13}
\end{align}
with
\begin{align}
    \label{def:G11:12:13}
    \begin{split}
    G^{11}&\defn \suma \bigl( \overline{S}^\alpha_{12}- S_{12}^\alpha(U)\bigr)\pal V^2,\\
     G^{12}&\defn - S_{12}(D)Z(D)^{-1}S_{22}(D) V^2 \esp{and} G^{13}\defn -i|D|S_{12}(D)  W .
       \end{split}
\end{align}
Putting all the linear terms in the left-hand side,  \eqref{eq:V1=f(W)} reads
\begin{multline}
    \label{eq:V1:lin=f(W)}
     \overline{S}^0_{11} \pt V^1 +\sum_{\alpha=1}^d \overline{S}^\alpha_{11} \pal  V^1+ S_{12}(D)Z(D)^{-1}S_{21}(D)V^1 =G^{1}\defn \sum_{k=1}^5 G^{1k}\\
\esp{with}   
    G^{14}\defn  \sum_{\alpha=1}^d \bigl( \overline{S}^\alpha_{11}-S^\alpha_{11}(U)\bigr) \pal V^1 \esp{and} G^{15}\defn  \bigl(\overline{S}^0_{11}-S^0_{11}(U)\bigr) \pt V^1.
\end{multline}
Given that $ Z(\omega)$ is invertible, we have
$$\text{rank}\bigl(S_{12}(\omega)Z(\omega)^{-1}S_{21}(\omega)\bigr) = \text{rank}\bigl(S_{21}(\omega)\bigr)\esp{for all}\omega\in\mathbb S^{d-1}.$$
Since the $0-$order operator $S_{12}(D)Z(D)^{-1}S_{21}(D) $ is responsible of the long time behavior of $V^1$,  the fact that $\text{rank}\bigl(S_{12}(\omega)Z(\omega)^{-1}S_{21}(\omega)\bigr)=n_1$
 for all $\omega\in \mathbb{S}^{d-1}$ guarantees decay  of $V^1.$
  This is actually equivalent to asserting that $\text{rank}\bigl(S_{21}(\omega)\bigr)=n_1$ for all $\omega\in \mathbb{S}^{d-1}$.
However, this assumption is not strictly necessary. In fact, under the assumption \eqref{strong_elli} and denoting 
\begin{equation}\begin{aligned}
    \label{def:matbf:N:M}
    & M_\omega\defn ( \overline{S^0})^{-1}B_\omega, \;\; N_\omega= ( \overline{S^0})^{-1}A_\omega, \esp{with} B_\omega, A_\omega \text{  defined in } \eqref{A:B_ome:def}\\
    &\mathbf{N}_\omega\defn  i(\overline{S}^0_{11})^{-1} S_{11}(\omega)=  i(\overline{S}^0_{11})^{-1}\suma \overline{S}^\alpha_{11}\omega_\alpha \!\esp{and}\!  \mathbf{M}_\omega\defn ( \overline{S}^0_{11})^{-1} S_{12}(\omega)Z(\omega)^{-1}S_{21}(\omega)
\end{aligned}
\end{equation}
we observe that the pair $(M_\omega, N_\omega)$  satisfies the (SK) condition if and only the pair $(\mathbf{M}_\omega , \mathbf{N}_\omega)$ does (see Lemma \ref{lem:SK:SK'} for more details). Moreover, the operator $\sum_{\alpha=1}^d \overline{S}^\alpha_{11} \pal$  is skew-Hermitian in the sense of \eqref{null:A:sky}, while $S_{12}(D)Z(D)^{-1}S_{21}(D)$ is elliptic (not necessarily strictly) in the sense of \eqref{posi:ope:B:}.  As a result, the Inequality \eqref{V(xi)=eV_0} (with $a=1$ and $b=0$) ensures that 
\begin{align}
    \label{est:V1(xi)}
    \widehat{V^1}(t,\xi)\le C e^{-ct\min( 1, |\xi|^2) }  \widehat{V^1_0}(\xi)\cdotp
\end{align}
where $c$ and $C$ are positive constants. In particular, the linear equation \eqref{eq:V1=f(W)} (with $G^1=0$) exhibits damped behavior in high frequencies.


\section{Main results}\label{subs:main:resu:NHHB}
  
Our first goal is to prove a global existence result for System \eqref{Eq_b}
supplemented with small initial data having different regularity in low and high frequencies. 
In order to find out a suitable functional framework for solving \eqref{Eq_b}, 
we have to remember that a part of the system is first order hyperbolic. 
Hence conservation of regularity is closely related to 
the control of the norm of $\nabla V$ in $L^1_{loc}(\mathbb R_+;L^\infty)$
(see e.g. \cite[Chap. 4]{HajDanChe11}).  
Consequently, we have to assume that $V_0$ belongs to a Banach space $X$
such that the corresponding solution to the linearized system \eqref{eq:lindiff0} 
belongs to the set $L^1_{loc}({\mathbb R}_+;C^{0,1})$. 
Following our recent  work \cite{DanADOloc}, we shall assume in a first time that the high frequencies of $(V_0^1,V^2_0)$ belong to\footnote{The reader is referred to Appendix \ref{appendix:LP}
for the definition of Besov spaces, and of the Littlewood-Paley decomposition.}  $\BH{\cd+1}{2}{1}\times\BH{\cd}{2}{1}.$ There is some freedom in the choice of the regularity index $s$ for the low frequencies part $V_0^l$.  A rather natural choice is $s=\cd-1$ because, as seen in subsection \ref{subs:SK:Kalman}, we have parabolic properties of the system in low frequencies. Hence starting from $V_0^l\in \BH{\cd-1}{2}{1}$ we expect that $\n V^l\in L^1_{loc}(\mathbb R_+;\BH{\cd}{2}{1})$, and thus, by embedding, $\n V^l\in L^1_{loc}(\mathbb R_+;L^\infty)$.   Combining these properties and the dissipative properties
of the system in high frequencies we get a control of the gradient of the solution, which is the key to preventing  finite time blow-up (see \cite{DanADOloc} ). 

The following general global existence result for System \eqref{Eq_b} in this framework
may be seen as an improvement of those in \cite{Kawashima83,SK85,GioMat13}.
\begin{thm}
    \label{thm:glob:cri}
  Let Assumption \textbf{D} be in force.  Assume that $d\ge 2$ 
  and that the pair $(N_\omega,M_\omega)$ defined in \eqref{def:matbf:N:M} satisfies the (SK) condition, for all $\omega\in \mathbb{S}^{d-1}$. 
    Then, there exists a positive constant $\alpha$ such that for all 
     $U_0$ with range in $\U$
    and such that  $V_0=(V_0^1, V^2_0)\defn U_0- \overline{U}$
    with    
    $V_0^1\in \BH{\cd-1}{2}{1}\cap\BH{\cd+1}{2}{1}$ and $ V_0^2\in \BH{\cd-1}{2}{1}\cap\BH{\cd}{2}{1} $ satisfies
\begin{align}
    \label{small:cnd:whole:space}
  \Vert V_0^1 \Vert_{ \BH{\cd-1}{2}{1}\cap\BH{\cd+1}{2}{1} }+  \Vert V_0^2 \Vert_{ \BH{\cd-1}{2}{1}\cap\BH{\cd}{2}{1} }\le \alpha,
\end{align}
System \eqref{Eq_b} supplemented with initial data $ (V_0^1,V^2_0)$ admits a unique global solution~$V$ in 
\begin{multline*}
    E:=\left\{V=( V^1,V^2): V^l\in \mathcal{C}_b(\mathbb{R}_+,\BH{\cd-1}{2}{1} )\cap L^1(\mathbb{R}_+, \BH{\cd+1}{2}{1}),\;\; \right.\\
    \left.\;\; V^{1,h}\in  \mathcal{C}_b(\mathbb{R}_+,\BH{\cd+1}{2}{1} )\cap L^1(\mathbb{R}_+, \BH{\cd+1}{2}{1}),\;\;  V^{2,h}\in  \mathcal{C}_b(\mathbb{R}_+,\BH{\cd}{2}{1} )\cap L^1(\mathbb{R}_+, \BH{\cd+2}{2}{1})\right\},
\end{multline*}
where, for any tempered distribution $Z,$ $Z^l$ (resp. $Z^h$)
denotes the low (resp. high) frequency part of $Z$ (see the definition in \eqref{def:LF:HF:}).
\medbreak
Moreover, there exists an explicit Lyapunov functional equivalent to 
$$\Vert V^1 \Vert_{ \BH{\cd-1}{2}{1}\cap\BH{\cd+1}{2}{1} }+  \Vert V^2 \Vert_{ \BH{\cd-1}{2}{1}\cap\BH{\cd}{2}{1} }$$ and a  constant $C$ depending only on the matrices $S^0, S^\alpha$ and  $Y^{\alpha\beta}$ such that
\begin{align}
    \label{ine:lya:V:LF:HF:BH}
     \mathbb{V}(t)\le C \mathbb{V}_0,  \esp{  } t>0
\end{align}
where (see \eqref{def:NBH:LF:HF} for the definition of $\|\cdot\|^l_{\dot B^{s}_{2,1}}$
and $\|\cdot\|^h_{\dot B^{s}_{2,1}}$):
\begin{equation}
     \label{ mathbb{V}:def}
    \mathbb{V}(t)\defn \Vert V \Vert^{l}_{L^\infty_t(\BH{\frac{d}{2}-1}{2}{1})}+  \Vert V^1\Vert^{h}_{L^\infty_t(\BH{\frac{d}{2}+1}{2}{1})}+ \Vert V^2\Vert^{h}_{L^\infty_t(\BH{\frac{d}{2}}{2}{1})}+ \Vert V\Vert_{L^1_t(\BH{\frac{d}{2}+1}{2}{1})}+\Vert V^2\Vert_{L^1_t(\BH{\frac{d}{2}+2}{2}{1}) }
    .\end{equation}
\end{thm}
The following statement specifies the decay rates of the above solutions. 
\begin{thm}
    \label{thm:decay:sucri}
    Under the hypotheses of Theorem \ref{thm:glob:cri} and if, additionally, $V_0\in \BH{-\sigma_1}{2}{\infty}$ for some  $ 1-d/2<\sigma_1\le d/2 $ then, there exists a constant $C$ depending only on $\sigma_1$ and such that
    \begin{align}
        \label{est:prog:NBH(V):r=inf}
        \NBH{-\sigma_1}{2}{\infty}{V(t)}\le C \NBH{-\sigma_1}{2}{\infty}{V_0}, \esp{for all } t\ge 0.
    \end{align}
    Furthermore, denoting 
    \begin{align*}
        {\langle t \rangle}\defn \sqrt{1+t^2}\esp{and} C_0\defn  \NBH{-\sigma_1}{2}{\infty}{V_0}^l + \lVert V^1_0 \rVert_{ \BH{\cd+1}{2}{1}}+ \lVert V^2_0 \rVert_{\BH{\cd}{2}{1}}^h
    \end{align*}
    the following decay estimates hold true (where $W$ has been defined according to \eqref{W:def}):
    \begin{align}\label{decay1}
     & \sup_{t\ge 0}  \NBH{\sigma}{2}{1}{  {\langle t \rangle}^\frac{\sigma_1+\sigma}{2} V(t) }^l \le C C_0 \esp{if} -\sigma_1< \sigma\le \cd-1,\\ \label{decay2}
     & \sup_{t\ge 0}  \NBH{\cd}{2}{1}{  {\langle t \rangle}^\frac{\sigma_1+d/2-1}{2} \bigl(W(t), V^2(t)\bigr) }^h \le C C_0,\\\label{decay3}
     &\sup_{t\ge 0}  \NBH{\cd+1}{2}{1}{  {\langle t \rangle}^\frac{\sigma_1+d/2-1}{2} V^1(t) }^h \le C C_0.
    \end{align}
\end{thm}
\begin{rmq}
\label{rmq:nosmall:BH:deca}
    Note  that $\NBH{-\sigma_1}{2}{\infty}{V}$ can be arbitrarily
large: only $ \Vert V_0^1 \Vert_{ \BH{\cd-1}{2}{1}\cap\BH{\cd+1}{2}{1} }+  \Vert V_0^2 \Vert_{ \BH{\cd-1}{2}{1}\cap\BH{\cd}{2}{1} }$ has to be small.
\end{rmq}
It is known since the second author in \cite{Dan01glob}  
on the compressible Navier-Stokes system that global well-posedness and time decay estimates hold for initial data less regular than those in Theorem \ref{thm:glob:cri}:
regularity $\BH{\cd-1}{2}{1}$ and $\BH{\cd}{2}{1}$ for the high frequencies of the velocity and
of the density. However, this alternative framework requires stronger structure assumptions on the system. We here assume that $\U=\U^1\times\mathbb R^{n_2}$ for some open subset $\U^1$ of
$\mathbb R^{n_1}$ and:
 \paragraph{\bf Assumption E} \label{assumption:cri}
 \medbreak
\begin{itemize}
  \item[(\textbf{E1})] $S^{0}_{11}={\rm Id}$ and  
  $S^{0}_{22}$ depends only (and smoothly) on $U^1\in\U^1$ 
 and has range in the set of $n_2\times n_2$ positive definite matrices,
     \item[(\textbf{E2})]for all $   \alpha=1,\cdots, d,$ the matrix $S^\alpha(U)$ is symmetric, and
   \begin{itemize}
       \item  all the maps  $S^{\alpha}_{11}$ are affine with respect to $U^2$
       and independent of $U^1$,
       \item   all the maps $S^{\alpha}_{21}$ are independent of $U^2$ and 
       depend smoothly on $U^1,$
       \item the smooth maps $ S^{\alpha}_{11},S^{\alpha}_{22}: (U^1,U^2)\in \mathcal{U}\subset \mathbb{R}^{n} \to \mathcal{M}_{n_2}(\mathbb R)$ are affine in~$U^2,$
   \end{itemize}
    
     \item[(\textbf{E3})]  all the  maps $ Y^{\alpha\beta}$ have the form \eqref{def:Y:al:be},
    are independent of $U^2,$ smooth with respect to $U^1$ 
     and the ellipticity property \eqref{strong_elli} is satisfied,
     \item[(\textbf{E4})]  $f\equiv0$. 
    \end{itemize}
    Let us emphasize that, as for our local existence result (see \cite[Thm 1.6]{DanADOloc}), Assumption \textbf{E}  is essentially needed to get estimates for the high frequencies of the solutions of \eqref{Eq_b}. Note that these assumptions are fulfilled by the barotropic compressible Navier-Stokes equations. 
\begin{thm}
    \label{thm:glo:L2:cri}
    Let  Assumption {\bf E}  be in force and assume  that the pair $(N_\omega,M_\omega)$ defined in \eqref{def:matbf:N:M} satisfies the (SK) condition, for all $\omega\in \mathbb{S}^{d-1}$.
    Then, there exists a positive constant $\alpha$ such that for all $U_0=(U_0^1,U_0^2)$ with range in $\U^1\times \mathbb R^{n_2}$
    and such that  $V_0=(V_0^1, V^2_0)\defn U_0- \overline{U}$
is in $(\BH{\cd-1}{2}{1}\cap \BH{\cd}{2}{1})\times\BH{\cd-1}{2}{1}$ and satisfies
\begin{align}
    \label{small:data:cri:thm}
    \lVert V^1_0\rVert_{\BH{\cd-1}{2}{1}\cap \BH{\cd}{2}{1}} +\NBH{\cd-1}{2}{1}{V^2_0}\le \alpha,
\end{align}
System \eqref{Eq_b} supplemented with initial data $V_0$ admits a unique global-in-time solution in 
the subspace $\mathcal{E}$ of functions $V=(V^1,V^2)$ of $\cC_b(\mathbb{R}_+; \BH{\cd-1}{2}{1}\cap \BH{\cd}{2}{1} )\times \cC_b(\mathbb{R}_+; \BH{\cd-1}{2}{1})$ such that
\begin{align*}
    V^{1,l}\in  L^1(\mathbb{R}_+, \BH{\cd+1}{2}{1}),\!\quad V^{1,h}\in  L^1(\mathbb{R}_+, \BH{\cd}{2}{1}), \!\quad  V^{2}\in L^1(\mathbb{R}_+, \BH{\cd+1}{2}{1})\!\esp{and}\! \partial_tV^2 \in L^1(\mathbb{R}_+, \BH{\cd-1}{2}{1}).
\end{align*}
Moreover, there exists an explicit Lyapunov functional equivalent to $  \lVert V^1\rVert_{\BH{\cd-1}{2}{1}\cap \BH{\cd}{2}{1}} +\NBH{\cd-1}{2}{1}{V^2} $ and 
we have:
\begin{multline}
    \label{lya:cri:thm}
    \widetilde{\mathbb{V}}(t) \le C  \widetilde{\mathbb{V}}(0) \esp{where}
    \widetilde{\mathbb{V}}(t)\defn \Vert V \Vert^{l}_{L^\infty_t(\BH{\frac{d}{2}-1}{2}{1})}+  \Vert V^1\Vert^{h}_{L^\infty_t(\BH{\frac{d}{2}}{2}{1})}+ \Vert V^2\Vert^{h}_{L^\infty_t(\BH{\frac{d}{2}-1}{2}{1})}\\
        + \Vert \partial_tV^2\Vert_{L^1_t(\BH{\frac{d}{2}-1}{2}{1})}
    +\Vert V^1\Vert_{L^1_t(\BH{\frac{d}{2}+1}{2}{1})}^l+ \Vert V^1\Vert_{L^1_t(\BH{\frac{d}{2}}{2}{1})}^h.
\end{multline}
\end{thm}


\section{Proof of Theorem \ref{thm:glob:cri}}
\label{sec:proof:global:exis}
This section is dedicated to proving the global existence and uniqueness of strong solutions for System \eqref{system_diff}, supplemented with initial data close to the reference solution $\overline{U}$, in the general case where Assumption \textbf{D} and the (SK) condition are satisfied.

The core of the proof involves establishing a priori estimates, with the remaining steps (proving existence and uniqueness) being more classical. As previously explained, our strategy revolves around deriving a Lyapunov functional in the style of Beauchard and Zuazua, which is equivalent to the norm we seek to control. This is then combined with the study of the \textit{parabolic mode} $W$, defined in \eqref{W:def}, to close the estimates.

\subsection{Establishing the a priori estimates}
\label{susec:priori:est:glob}
Throughout this section, we fix some bounded open subset $\cO$  such that 
$\overline{\cO}\subset\U,$ and 
assume that we are given a smooth (and decaying) solution $V$ of \eqref{Eq_b} on $[0,T]\times \Rd$ with $V_0\defn U_0-\overline{U}$ as initial data, satisfying 
\begin{equation}
    \label{smal:cnd:V}
    \sup_{t\in [0,T]}\NBH{\frac{d}{2}}{2}{1}{V(t)} \ll 1\esp{and} U([0,T]\times\Rd)\subset\overline\cO.
\end{equation}
We shall use repeatedly that, owing to the embedding $\BH{\frac{d}{2}}{2}{1} \hookrightarrow L^\infty$, we have also
\begin{equation}
    \label{smal:inf:V}
    \sup_{t\in [0,T]} \normeinf{V(t)} \ll 1.
\end{equation}

From now on,  $(c_j)_{j\in \mathbb{Z} }$ stands for a positive sequence that verifies $\left\Vert (c_j)\right\Vert_{l^1(\mathbb{Z})}\le1$. 
\medbreak

Since the system exhibits distinct behavior in the low- and high-frequency regimes at the linear level (see Subsection \ref{subs:SK:Kalman}), it is natural to analyze it within these two contexts. But first, we need to prove the following bounds of time derivatives.
\medbreak\paragraph{$\bigstar$ \textit{Estimates for $\pt V$}:}
\begin{lem}
    \label{lem:ptV}
    Let V be a smooth and sufficiently decaying solution  of \eqref{Eq_b} on $[0,T]\times\Rd$ satisfying \eqref{smal:cnd:V}, Assumption \textbf{D} and Condition (SK). Then, for all  $\kappa\in ]-\frac{d}{2},\frac{d}{2}]$, we have:
\begin{align}
    \label{pt:V:1:est}
         \NBH{\kappa}{2}{1}{ \pt V^1 } &\lesssim \NBH{\kappa}{2}{1}{\nabla V}\\
    \label{pt:V:2:est}
\esp{and}
         \NBH{\kappa}{2}{1}{ \pt V^2 } &\lesssim \NBH{\kappa}{2}{1}{\nabla V}+
      \NBH{\kappa+2}{2}{1}{V^2}+\NBH{\cd}{2}{1}{\n V}\NBH{\kappa}{2}{1}{\n V}.
\end{align}
\end{lem}
\begin{proof}
 Proving   \eqref{pt:V:1:est} relies on product and composition estimates
(see Inequalities \eqref{product_propo2} and \eqref{comp:uv:propo:r=1}), 
on the smallness condition \eqref{smal:cnd:V} and on the fact that 
$$\partial_tV^1=-\sum_\alpha\Bigl(\wt S_{11}^\alpha(U)\partial_\alpha V^1+ \wt S_{12}^\alpha(U)\partial_\alpha V^2\Bigr)\esp{with}  \wt S_{1k}^\alpha\defn
 (S_{11}^0)^{-1}S_{1k}^\alpha\esp{for} k=1,2. $$
Similarly, for proving \eqref{pt:V:2:est}, we use that denoting $\wt S_{2k}^\alpha\defn (S_{22}^0)^{-1}S_{2k}^\alpha\esp{for} k=1,2,$ we have
\begin{multline*}
	 \pt V^2 = -\sum_{\alpha}{ \Bigl( \wt S^\alpha_{21} (U) \pal V^1+  \wt S^\alpha_{22} (U) \pal V^2}\Bigr)\\
  +\sum_{\alpha,\beta}{ (S^0_{22})^{-1}(U)\pal (Z^{\alpha\beta}(U) \pbe V^2)  }+ (S^0_{22})^{-1}(U)f^2 \cdotp
	\end{multline*}
 So, using \eqref{product_propo2} and  \eqref{comp:uv:propo:r=1} yields 
 \begin{align*}
     \NBH{\kappa}{2}{1}{ \pt V^2 } \lesssim \biggl( 1+\NBH{\frac{d}{2}}{2}{1}{V}\biggr) \biggl( \NBH{\kappa}{2}{1}{\nabla V}+ \NBH{\kappa}{2}{1}{f^2}+
      \sumab \NBH{\kappa}{2}{1}{\pal(Z^{\alpha\beta}(U) \pbe V^2)}\biggr)\cdotp
 \end{align*}
Using Leibniz formula in order to decompose the second order term as
 \[ \pal\Bigl(Z^{\alpha\beta}(U) \pbe V^2\Bigr)=  Z^{\alpha\beta}(U) \pal\pbe V^2+ \pal((Z^{\alpha\beta}(U)) \pbe V^2  ,\]
  and remembering  Assumption \textbf{D} for $f^2$  yields
 \begin{multline*}
     \NBH{\kappa}{2}{1}{ \pt V } \lesssim \left( 1+\NBH{\frac{d}{2}}{2}{1}{V}\right) \biggl(\NBH{\kappa}{2}{1}{\nabla V}+ (1+\NBH{\frac{d}{2}}{2}{1}{V})\NBH{\cd}{2}{1}{\n V}\NBH{\kappa}{2}{1}{\n V}\\+
       ( 1+\NBH{\frac{d}{2}}{2}{1}{V}) \NBH{\kappa+2}{2}{1}{V^2} +\NBH{\cd+1}{2}{1}{V} \NBH{\kappa+1}{2}{1}{V^2}\biggr)\cdotp
 \end{multline*}
 We get then \eqref{pt:V:2:est} by using the smallness condition \eqref{smal:cnd:V}. 
 \end{proof}

\paragraph{$\bigstar$ \textit{Low frequencies analysis}:}
Using from now on the notation $Z_j\defn \DDj Z$, $F_j\defn \DDj F$ and applying to \eqref{eq:lindiff0}  the operator $\DDj$, we get
\begin{align}
\label{eq:lh:V}
     \overline{S}^0\pt V_j +\sum_{\alpha}{\overline S^\alpha\pal V_j}-\sum_{\alpha,\beta=1}^d\overline{Y}^{\alpha\beta} \partial_{\alpha}\partial_\beta V_j = F_{j} .
 \end{align}
 Taking the $L^2(\mathbb{R}^d,\mathbb{R}^n)$ scalar product with $V_j$, integrating by parts in the second term, and using the symmetry of matrices $\overline S^\alpha$ leads to:
 \begin{align*}
     \begin{split}
           \frac{d}{dt}\intd\overline{S}^0 V_j\cdot V_j -\sumab\intd \overline{Z}^{\alpha\beta}\pal\pbe V^2_j\cdot V^2_j=\intd F_{j} \cdot V_j.
     \end{split}
 \end{align*}
 Next, the strong ellipticity  property \eqref{strong_elli} and Fourier-Plancherel formula ensure that 
 \begin{equation}
 \label{1:es:hf:}
     \frac{d}{dt}\intd\overline{S}^0 \widehat{V_j}\cdot \widehat{V_j} +\overline{c_1} \intd |\xi|^2|\widehat{V^2_j}|^2d\xi\le 
     \normede{\widehat F_{j}}\normede{\widehat V_j}
     \hbox{ with }\  \overline{c_1} \defn c_1(\overline{U}).
 \end{equation}
 Inequality \eqref{1:es:hf:} provides  a $L^2$-in-time control only for $V^2$. In order to 
 track  diffusion effects for the whole solution, 
 we proceed as explained in Subsection \ref{subs:SK:Kalman} 
 introducing the functional $L_{\ro,\omega}$ defined in \eqref{def:Lya:r:ome}. 
 After integration on $\Rd,$ this amounts to considering 
\begin{equation}    \label{df:I:j:low}
    I_j \defn  \sumk \ee_k {\rm  Re} \intd (M_\omega N_\omega^{k-1} \widehat{V_j}\cdot M_\omega N_\omega^{k} \widehat{V_j})
\end{equation}
(where $N_\omega$ and $M_\omega$ have been defined in \eqref{def:matbf:N:M}) and 
$$ L_{\ro,\omega} (\widehat{V}_j)\defn  \overline{S}^0 \widehat{V}_j\cdot \widehat{V}_j + \min( \rho,\rho^{-1} )  I_j .$$
Leveraging Lemma \ref{lem:derI} with $a=1$ and $b=2$  and  remembering that ${\rm Supp}\, \widehat{V}_j\subset \{ \frac{3}{4}2^{j}\le |\xi|\le \frac{8}{3}2^j\} $ (hence $ |\xi|=\rho \simeq 2^j$ on ${\rm Supp} \,\widehat{V}_j$), we get (note that the only change lies in the harmless additional source term $F_j$):
\begin{equation}
\label{low:lya:eqj}
    \frac{d}{dt}I_j +2^{j} \sumk \ee_k \intd|M_\omega N_\omega^{k} \widehat{V}_j|^2 \le C\ee_0\max(2^{j},2^{3j})\intd| M_\omega  \widehat{V}_j|^2+C \normede{(\overline{S}^0)^{-1}\widehat F_j}\normede{\widehat V_j}
\end{equation}
provided we choose $\ee_1,\cdots,\ee_{n-1}$ small enough
and, from \eqref{eqi:lya:ro:omega}, we have
\begin{equation}
    \label{lya:fun:def:l}
   \mathcal{L}_j^l\defn (2\pi)^{-d}\intd  L_{\ro,\omega} (\widehat{V}_j) \simeq  \normede{V_j}^2, \;\; \forall\;  j \in \mathbb{Z} .
\end{equation}
Putting together Inequalities  \eqref{1:es:hf:}, \eqref{low:lya:eqj}, and remembering that \eqref{lya:fun:def:l} holds true,  we get
\begin{multline}
    \label{est:LV:lf:1}
      \frac{d}{dt}  \mathcal{L}_j^l + 2^{2j}c \intd|\widehat{V^2_j}|^2 + \min(1, 2^{2j}) \sumk \ee_k \intd |M_\omega N_\omega^{k} \widehat{V}_j|^2 \le C\ee_02^{2j} \intd | M_\omega  \widehat{V}_j|^2\\
      +C \normede{F_j } \sqrt{\mathcal{L}_j^l}.
\end{multline}
Choosing $\varepsilon_0,\cdots, \varepsilon_{n-1}$ according to Lemma \ref{lem:derI} yields the following inequality:
\begin{align*}    
    \mathcal{N}_V& \defn 2^{2j}c\intd|\widehat{V^2_j}|^2  +\min(1, 2^{2j}) \sumk \ee_k \intd |M_\omega N_\omega^{k} \widehat{V}_j|^2 - C\ee_02^{2j}\intd | M_\omega  \widehat{V}_j|^2\\
   &\ge  \min(1, 2^{2j}) \sum_{k=0}^{n-1} \ee_k \intd |M_\omega N_\omega^{k} \widehat{V}_j|^2 .
\end{align*}
Lemma \ref{lem:SK_diss} ensures that, since condition (SK) is satisfied, the quantity on the right-hand side of the above inequality is positive for
any choice of positive parameters $ \varepsilon_0, \cdots, \varepsilon_{n-1} $.  Consequently, if we use Fourier-Plancherel identity and the equivalence \eqref{lya:fun:def:l}, we see that, for all $ j\in \mathbb{Z}$,
\begin{align}    \label{get:diffu:low}
     \mathcal{N}_V \ge c\min(1, 2^{2j})  \mathcal{L}_j^l.\end{align}
\begin{propo}
\label{pro:low:est:Vj:ge}
Let V be a smooth and sufficiently decaying solution of \eqref{Eq_b} on $[0,T]\times\Rd$ satisfying \eqref{smal:cnd:V}, Assumption \textbf{D} and Condition (SK). Then, for all  $s\in ]-\frac{d}{2},\frac{d}{2}]$ and $j\in \mathbb{Z}$, we have:
\begin{multline}
    \label{lo:fr:est:s}
              \frac{d}{dt}  \mathcal{L}_j^l  \!+\! \min(1, 2^{2j})  \mathcal{L}_j^l \\\le 
               c_j \dm{s}\sqrt{\mathcal{L}_j^l}\biggl(
\NBH{\cd+1}{2}{1}{ V}\NBH{s+1}{2}{1}{V}   \!+\!\NBHc{V}\left( \NBH{s+1}{2}{1}{ V}+\NBH{s+2}{2}{1}{V^2}\right)\biggr)\cdotp
    \end{multline}
\end{propo}
\begin{proof}
 For proving \eqref{lo:fr:est:s}, we first remark that Inequalities  \eqref{est:LV:lf:1} and \eqref{get:diffu:low} ensure that
 \begin{align}
     \label{0:lya:lf:est}
     \begin{split}
          \frac{d}{dt}  \mathcal{L}_j^l  + c\min(1, 2^{2j})  \mathcal{L}_j^l \le 
      C \normede{F_j } \sqrt{\mathcal{L}_j^l}.
     \end{split}
 \end{align}
The rest of the proof consists in bounding  the r.h.s of \eqref{0:lya:lf:est}, which requires  product estimates in homogeneous Besov spaces. Thanks to \eqref{product_propo2} and Proposition \ref{propo_compo_BH}, remembering the definition 
of $F^1$ in \eqref{F_T:F_1:F_2}, we discover that for all $s\in ]-\frac{d}{2},\frac{d}{2}]$ :
\begin{align}
\label{F_1:est:st}
 \|F^{1}\|_{\dot B^s_{2,1}} &\lesssim \suma\NBHc{ \overline{S}^\alpha-S^\alpha(U)}\NBH{s}{2}{1}{\n V}\nonumber\\
    & \lesssim \NBHc{ V}\NBH{s}{2}{1}{\n V}.
\end{align}
To handle $F^t$, we use \eqref{pt:V:1:est} and \eqref{pt:V:2:est} with $\kappa=s$ and  Inequality \eqref{product_propo2}.
We get:
\begin{align}      \|F^{t}\|_{\dot B^s_{2,1}}
     &\lesssim\NBHc{V} \NBH{s}{2}{1}{\pt V}\nonumber\\
    \label{F_t:est:st}
     &\lesssim \NBHc{V}\Bigl(\NBH{s}{2}{1}{\nabla V}+  \NBH{s+2}{2}{1}{V^2}+
     \NBH{s+1}{2}{1}{V} \NBH{\frac{d}{2}+1}{2}{1}{V}\Bigr)\cdotp
\end{align}
 To handle $F^2$, we use Leibniz identity:
 \begin{align}
    \label{decom:F2<=Leib}
    \pal\Bigl( r^{\alpha\beta}(U)\pbe V^2\Bigr) =  \pal\Bigl( r^{\alpha\beta}(U)\Bigr)\pbe V^2 +  r^{\alpha\beta}(U)  \pal\pbe V^2,
\end{align}
 then  Inequalities   \eqref{product_propo2}  and   \eqref{compo:propo:base}.
 We get
\begin{align}    \label{F_2:est:st}
     \|F^{2}\|_{\dot B^s_{2,1}} \lesssim \Bigl( \NBH{\cd+1}{2}{1}{V} \NBH{s+1}{2}{1}{V^2}+ \NBH{\frac{d}{2}}{2}{1}{V} \NBH{s+2}{2}{1}{V^2}\Bigr)\cdotp
\end{align}
The  term $f$ can be bounded from product estimate \eqref{product_propo2}, and the form of $f$ in \eqref{globa:f=f(f1,f2)}:
\begin{align}    \label{f:est:st}
     \|f\|_{\dot B^s_{2,1}} \lesssim  (1+\NBHc{V}) \NBH{\cd}{2}{1}{\n V}\NBH{s}{2}{1}{\n V}\lesssim  \NBH{\cd}{2}{1}{\n V}\NBH{s}{2}{1}{\n V}.
\end{align}
Putting together \eqref{0:lya:lf:est},  \eqref{F_1:est:st}, \eqref{F_t:est:st} and \eqref{F_2:est:st} yields the desired estimate. 
\end{proof}

\paragraph{$\bigstar$ \textit{High frequencies analysis}:} 
As pointed out in Section \ref{subs:main:resu:NHHB}, we need to control (at least the high frequencies part of)  $\Vert\n V\Vert_{L^1_t(L^\infty)}.$
In our Besov spaces setting, this amounts to bounding 
$ V$ in $L^1_T(\BH{\cd+1}{2}{1})$. To achieve it, 
we have to exhibit the optimal smoothing properties of the system
for $V^2.$ It turns out to be more convenient to rewrite 
the whole system as a coupling between the  \textit{parabolic unknown} $W$ defined in \eqref{W:def}
and the \textit{hyperbolic unknown}~$V^1.$

Before proceeding, we introduce for any $q\times q$ symmetric positive definite
matrix $A$ and $L^2(\Rd;\mathbb R^q)$ function $X$ the notation
 \begin{equation}\label{eq:weight}
     \lVert X\rVert_{L^2_{A}}\defn \biggl(\intd A X\cdot X\biggr)^{1/2}\cdotp
    \end{equation}

$ \bullet$ \textbf{\textit{Parabolic mode.} }
Our aim here is to establish  the following result of 
 parabolic maximal regularity for  $W.$ 
\begin{propo}
    \label{pro:es:W}
    Let $V$ be a smooth (decaying) solution of \eqref{Eq_b} on $[0,T]\times \Rd$ satisfying \eqref{smal:cnd:V}. Assume that $\sigma\in ]-\frac{d}{2},\frac{d}{2}]$. 
    There exists a constant $\bar c_1$  depending only on the coefficients
    of the system at $\overline U$ such that, denoting 
    $ W_j\defn \DDj W$   
 we have for all $t\in[0,T]$ and  $j\in \mathbb{Z}$,  
    \begin{multline}    \label{W:es:-1}
         2^{j\sigma}\lVert W_j(t) \rVert_{L^2_{\overline{S}^0_{22}}}+ \bar c_12^{(\sigma+2)j}\int^t_0 \normede{W_j}\le 2^{j\sigma}\lVert W_j(t) \rVert_{L^2_{\overline{S}^0_{22}}}\\
         +C\int^t_0 c_j \biggl(2^{-j}\Bigl(\NBH{\sigma+1}{2}{1}{  V}+\NBH{\sigma+2}{2}{1}{ V^2}\Bigr)
 + \NBHc{V} \Bigl(\NBH{\sigma+1}{2}{1}{V}+\NBH{\sigma+2}{2}{1}{ V^2}+\NBH{\sigma+1}{2}{1}{ V}\NBH{\frac{d}{2}+1}{2}{1}{ V}\Bigr)\biggr)\cdotp
    \end{multline}
\end{propo}
\begin{proof}
    Applying $\DDj$ to \eqref{widehatW:eq} and taking the scalar product with $W_j$  yields (see \eqref{def:h} for the definition of $h$):
    \begin{multline}
    \label{W:es:0}
     \dfrac{d}{dt}\intd \overline{S}^0_{22}W_j\cdot W_j- \intd \Delta Z(D) W_j\cdot W_j \\
        \leq\left( \normede{\DDj|D|^{-1}\overline{S}^0_{22}(Z(D))^{-1}\pt\left( S_{21}(D) V^1+  S_{22}(D) V^2 \right)}+\normede{h_j}\right)\normede{W_j}.
    \end{multline}
    Since Operator $ |D|^{-1}\overline{S}^0_{22}(Z(D))^{-1} $ is 
    homogeneous of degree $-1$,  Bernstein inequality yields
    \begin{align*}        
        \normede{\DDj|D|^{-1}\overline{S}^0_{22}(Z(D))^{-1}\pt\left( S_{21}^\omega(D) V^1+  S_{22}(D) V^2 \right)} \lesssim  2^{-j}\normede{\DDj\pt\left( S_{21}(D) V^1+  S_{22}(D) V^2 \right)}.
    \end{align*}
    Also,  thanks to Fourier-Plancherel theorem and \eqref{strong_elli}, we have:
    \begin{align*}
        -\intd\Delta Z(D) W_j\cdot W_j&= (2\pi)^{-d}\intd |\xi|^2 Z(\xi) \widehat{W}_j(\xi)\cdot \widehat{W}_j(\xi) \ge 2^{2j}\overline{c}_1 \normede{W_j}^2.
    \end{align*}
Hence,  we deduce that 
    \begin{multline}   \label{W:es:1}
        \dfrac{d}{dt}\intd \overline{S}^0_{22}W_j\cdot W_j+ 2^{2j}\overline{c}_1 \normede{W_j}^2\\\lesssim \Bigl(2^{-j}\normede{\DDj\pt\left( S_{21}(D) V^1+  S_{22}(D) V^2 \right)}
        +\normede{h_j}\Bigr)\normede{W_j}.
    \end{multline}
    There remains to look at the right-hand side of \eqref{W:es:1}.
As \eqref{smal:cnd:V} is satisfied, the inequalities \eqref{product_propo2}, \eqref{pt:V:1:est} and \eqref{pt:V:2:est} give that for all $\sigma \in ]-\frac{d}{2},\frac{d}{2}]$
\begin{align}
    \label{est:h^t}
    \normede{\DDj h^t} &\lesssim \dm{\sigma}c_j \NBHc{V}\NBH{\sigma}{2}{1}{\pt V^2}\nonumber\\
    &\lesssim  \dm{\sigma} c_j\NBHc{V} (\NBH{\sigma}{2}{1}{\nabla  V}+\NBH{\cd}{2}{1}{\n V}\NBH{\sigma}{2}{1}{\n V}+\NBH{\sigma+2}{2}{1}{ V^2}).
\end{align}
For bounding $ h^{21}$ and $h^{22}$ we use \eqref{product_propo2}. We have for all $\sigma \in ]-\frac{d}{2},\frac{d}{2}]$:
\begin{align}
    \label{h^t:est}
    \normede{\DDj h^{21}}+ \normede{\DDj h^{22}} 
    &\lesssim \dm{\sigma} c_j\NBHc{V}\NBH{\sigma+1}{2}{1}{ V}.
\end{align}
The control of $h^2$  relies on Inequality \eqref{product_propo2}. Then, for all $\sigma \in ]-\frac{d}{2},\frac{d}{2}],$ we have
\begin{align}
    \label{h2:est}
    \normede{\DDj h^{2}} &\lesssim \dm{\sigma} c_j\NBHc{V}(\NBH{\sigma+2}{2}{1}{ V^2}+\NBH{\sigma+1}{2}{1}{ V}\NBH{\frac{d}{2}+1}{2}{1}{ V^2}).
\end{align}
Finally recalling that the operators $ S_{21}(D)$ and $S_{22}(D)$  are $0$-order, the term $ \DDj \pt( S_{21}(D) V^1+  S_{22}(D) V^2 )$ may be bounded as follows: 
\begin{align}
\label{est:ptS(D)V=ptV}
    \normede{\DDj \pt( S_{21}(D) V^1+  S_{22}(D) V^2 )} \lesssim \normede{\pt \DDj V}
\end{align}
which becomes using \eqref{pt:V:1:est}, \eqref{pt:V:2:est} and the smallness condition \eqref{smal:cnd:V}
\begin{align*}
   \normede{\DDj \pt( S_{21}(D) V^1+  S_{22}(D) V^2 )} &\lesssim \dm{\sigma} c_j  (\NBH{\sigma}{2}{1}{\nabla  V}+ \NBH{\cd}{2}{1}{\n V}\NBH{\sigma}{2}{1}{\n V}+\NBH{\sigma+2}{2}{1}{ V^2}).
\end{align*}
The term ${f}$ has been bounded in the proof of Proposition \ref{pro:low:est:Vj:ge}. In fact, from \eqref{f:est:st}, we have,
\begin{align}
    \label{f:est:st:hf}
     \normede{f_{j}} \lesssim c_j\dm{\sigma}  \NBH{\cd}{2}{1}{\n V}\NBH{\sigma}{2}{1}{\n V}.
\end{align}
Putting all this information together 
we deduce that for all $\sigma \in ]-\frac{d}{2},\frac{d}{2}]$ and $j\in \mathbb{Z}$
\begin{multline*}
            \dfrac{d}{dt}\intd \overline{S}^0_{22}W_j\cdot W_j+ 2^{2j}\overline{c}_1 \normede{W_j}^2\\\lesssim 2^{-j\sigma}c_j\left( 2^{-j}\Bigl(  \NBH{\sigma+1}{2}{1}{  V}+\NBH{\sigma+2}{2}{1}{ V^2}+\NBH{\sigma+1}{2}{1}{ V}\NBH{\frac{d}{2}+1}{2}{1}{ V}\Bigr) \right. \\
\left. +\NBHc{V} \Bigl(\NBH{\sigma+1}{2}{1}{V}+\NBH{\sigma+2}{2}{1}{ V^2}+\NBH{\sigma+1}{2}{1}{ V}\NBH{\frac{d}{2}+1}{2}{1}{ V}\Bigr)\right)\normede{W_j}.
    \end{multline*}
    Using Lemma \ref{lem_der_int} with $X=\overline{S}^0_{22}W_j\cdot W_j$ 
    gives us the desired estimate.
\end{proof}

$ \bullet$ \textbf{\textit{Estimate for $V^1$.} }
Here we will look at $V^1$ as the solution of the symmetric hyperbolic equation \eqref{eq:V1=f(W)} whose linearized part is given by the left-hand side of \eqref{eq:V1:lin=f(W)}.
 Let $V^1_j\defn \DDj V^1.$
Applying Operator $S^0_{11}(U)\DDj(S^0_{11}(U))^{-1}$ to   \eqref{eq:V1=f(W)}  yields
\begin{equation}
  \label{j:full_dif:V1:2}
	S^0_{11}(U)\pt V^1_j + S^\alpha_{11}(U) \pal V^1_j + S_{12}(D)Z(D)^{-1}S_{21}(D)V^1_j = R^{11}_j+\sum_{k=1}^2 H_j^{k}+\sum_{k=1}^3 G^{1k}_j,
	 \end{equation}
  with  
  \begin{equation}
     \label{df:R^12_j}
     \begin{split}
 &  G^{1k}_j\defn S^0_{11}(U)\DDj \bigl( (S^0_{11}(U))^{-1}G^{1k}\bigr), \;\;  R^{11}_j\defn S^0_{11}(U) \sum_{\alpha=1}^d\left[ \DDj,\wt S^\alpha_{11}(U)\right]\pal V^1,
      \end{split}
 \end{equation}
 where  $G^{11}, G^{12}$ and $G^{13}$ are defined in \eqref{def:G11:12:13}, 
and 
 \begin{equation}
     \label{df:H^12_j}
     \begin{split}
 &  H^{1}_j\defn \biggl( \overline{S^0_{11}}- S^0_{11}(U)  \biggr)\DDj \biggl( (S^0_{11}(U))^{-1} S_{12}(D)Z(D)^{-1}S_{21}(D)V^1 \biggr),\\
 &H^2_j\defn \overline{S^0_{11}} \DDj \biggl( \biggl( (\overline{S^0_{11}})^{-1}- (S^0_{11}(U))^{-1}\biggr) S_{12}(D)Z(D)^{-1}S_{21}(D)V^1 \biggr)\cdotp
      \end{split}
 \end{equation}
 
Taking the $L^2$ inner product of \eqref{j:full_dif:V1:2}  with $V^1_j$, applying Leibniz' formula in the first term, integrating by parts in the second term, using the symmetry of the matrices $S^0_{11}$ and $S^\alpha_{11}$ and the fact that ${}^T\!S^\alpha_{12}=S^\alpha_{21}$  yields
\begin{multline*}
   \frac{1}{2}  \frac{d}{dt}\intd S^0_{11}(U) V_j^1\cdot V^1_j+ \intd Z(D)^{-1}S_{21}(D)V^1_j\cdot S_{21}(D)V^1_j\\
     =  \frac{1}{2} \intd \biggl( \partial_t S^0_{11}(U)+ \sum_{\alpha}\pal (S^\alpha_{11} (U))\biggr)  V_j^1\cdot V_j^1
     +\intd (R^{11}_j+H^1_j+H^2_j+G^{11}_j+ G^{12}_j+G^{13}_j)\cdot V^1_j.
 \end{multline*}
 Next, the strong ellipticity of operator $Z(D)$  (see \eqref{strong_elli}) ensures that
 \begin{multline}
 \label{2:es:hf:j}
    \frac{1}{2} \frac{d}{dt}\intd S^0_{11}(U) V_j^1\cdot V^1_j+ \overline{c}_1\normede{S_{21}(D)V^1_j}^2 \\
     \le  \frac{1}{2} \intd \biggl( \partial_t S^0_{11}(U)+ \sum_{\alpha}\pal (S^\alpha_{11} (U))\biggr)  V_j^1\cdot V_j^1
     +\intd (R^{11}_j+H^1_j+H^2_j+G^{11}_j+ G^{12}_j+G^{13}_j)\cdot V^1_j.
 \end{multline}
As $S^0(U)$ is a uniformly  positive definite matrix on $\cO,$ using \eqref{smal:inf:V} gives
\begin{equation}
\label{def:mathcal(V)}
    \mathcal{V}_j^1  \defn \intd S^0_{11}(U) V_j^1\cdot V^1_j \simeq \normede{V^1}^2,
\end{equation}
and thus
\begin{multline}
 \label{V12:es:hf:j}
 \frac{1}{2}\frac{d}{dt} \mathcal{V}_j^1+ \overline{c}_1 \normede{S_{21}(D)V^1_j}^2 \le \frac{1}{2}  \biggl( \normeinf{\partial_t S^0_{11}(U)} \\
 +\sum_{\alpha}\normeinf{\pal (S^\alpha_{11} (U))}\biggl) \mathcal{V}_j^1
+\normede{(R^{11}_j,H^1_j+H^2_j,G^{11}_j, G^{12}_j,G^{13}_j)}\sqrt{\mathcal{V}_j^1} .
 \end{multline}
 If the rank of $S_{21}(\omega)$ is equal to $n_1$	
  for all $\omega\in \mathbb{S}^{d-1}$, then  Inequality \eqref{V12:es:hf:j} guarantees decay for all components of $V^1$. Since that rank assumption has not been made, we take advantage of Beauchard-Zuazua's method developed in subsection \ref{subs:SK:Kalman}, this time for the equation obtained when applying the operator $\DDj$ to \eqref{eq:V1:lin=f(W)}. The only change lies in the (harmless) additional source terms $G^1$  defined in \eqref{def:G11:12:13} and in \eqref{eq:V1:lin=f(W)}.
  
  Now, using  the notation introduced   in  \eqref{def:matbf:N:M}, it holds:
\begin{equation}
    \label{sys:j_mor_reg}
     \pt \widehat{V^1_j} + (\rho \mathbf{N}_\omega+  \mathbf{M}_\omega) \widehat{V_j^1}= (\overline{S^0_{11}})^{-1}\widehat{G^1_j}.
\end{equation}
Let us further set:
\begin{equation}
    \label{mathbf{I}:def  }
    \mathbf{I}_j\defn \sumk \ee_k {\rm  Re} \intd(\mathbf{M}_\omega \mathbf{N}_\omega^{k-1} \widehat{V^1_j}\cdot \mathbf{M}_\omega \mathbf{N}_\omega^{k} \widehat{V^1_j})
    \esp{and}\mathcal{L}_j^{1,h}  \defn \mathcal{V}_j^1+  \min{(2^j,2^{-j})}\mathbf{I}_j.
\end{equation}
 Lemma \ref{lem:derI} (with $a=1$ and $b=0$) gives  for suitable
  $\ee_1,\cdots,\ee_{n-1},$
\begin{equation}
    \label{lya:fun:def:h}
    \mathcal{L}_j^{1,h}  \simeq \normede{V^1_j}^2
\end{equation}
and, 
using Fourier-Plancherel theorem, for all $j\in \mathbb{Z}$:
\begin{multline}
    \label{hi:lya:ful}
 \frac{d}{dt}\mathcal{L}_j^{1,h}+ \overline{c}_1 2^{2j}\normede{S_{21}(D)V^1_j}^2
 \\+ \min{(1,2^{2j})} \sumk \ee_k \int_{{\mathbb R}^d} |\mathbf{M}_\omega \mathbf{N}_\omega^{k} \widehat{V^1_j}|^2
    \le C  \biggl( \normeinf{\partial_t S^0_{11}(U)}+ \sum_{\alpha}\normeinf{\pal (S^\alpha_{11} (U))}\\
    +\normede{\left(R^{11}_j,G^{11}_j, G^{12}_j,G^{13}_j\right)}\biggr) \mathcal{L}_j^{1,h}
     +\left( \min(2^j,2^{-j})\normede{ G^1_j}\right)\sqrt{\mathcal{L}_j^{1,h}}.
\end{multline}
At this stage,  we have to remember (see Lemma \ref{lem:SK:SK'})
that  if the pair $(N_\omega,M_\omega)$ satisfies the (SK) condition then 
so does  $(\mathbf{N}_\omega,\mathbf{M}_\omega),$ which  guarantees  that
\begin{align} \label{sk:appl:HF}
    \mathcal{N}_{V^1}&\defn\overline{c}_1 2^{2j}\normede{S_{21}(D)V^1_j}^2 
 + \min{(1,2^{2j})} \sumk \ee_k \int_{{\mathbb R}^d} |\mathbf{M}_\omega \mathbf{N}_\omega^{k} \widehat{V^1_j}|^2
 \gtrsim \min{(1,2^{2j})} \mathcal{L}_j^{1,h}.
\end{align}
 After suitably bounding the r.h.s. of \eqref{hi:lya:ful}, these considerations will lead to
\begin{propo}\label{pro:est:HF:BFV}
Assume that $V$ satisfies the assumptions of Proposition \ref{pro:low:est:Vj:ge}. Then the following inequality holds true  for all $j\in \mathbb{Z}$:
\begin{align*}
    \begin{split}
        \frac{d}{dt}\mathcal{L}_j^{1,h}+ \min(1,&2^{2j})\mathcal{L}_j^{1,h}\lesssim   \biggl( \NBH{\cd+1}{2}{1}{V}+\NBH{\cd+2}{2}{1}{V^2}+ \NBH{\cd+1}{2}{1}{V}\NBH{\cd+1}{2}{1}{V^2}\biggr)\mathcal{L}_j^{1,h}\\
        &+\normede{ (V^2_j, 2^j W_j) }\sqrt{\mathcal{L}_j^{1,h}}+ 2^{-j(\cd+1)} c_j \biggl( \NBH{\cd}{2}{1}{V}\NBH{\cd+2}{2}{1}{ V^2}+ \NBH{\cd+1}{2}{1}{V}^2\biggr)\sqrt{\mathcal{L}_j^{1,h}}\\
  &+ c_j 2^{-j(\cd+1)}\biggl(\NBH{\cd}{2}{1}{V}\NBH{\cd+1}{2}{1}{V^1}+ \NBH{\cd+1}{2}{1}{V}\NBH{\cd}{2}{1}{V^1}\biggr)\sqrt{\mathcal{L}_j^{1,h}}\\
&  + \min(2^j,2^{-j})\biggl( \normede{ (V^2_j, 2^j W_j) }+ c_j \dm{\cd}  \NBH{\cd+1}{2}{1}{V}\NBH{\cd}{2}{1}{ V}\biggr)\sqrt{\mathcal{L}_j^{1,h}}\cdotp
    \end{split}
\end{align*}
\end{propo}

\begin{proof}
Combining \eqref{hi:lya:ful} and \eqref{sk:appl:HF}, we deduce that
\begin{multline}
    \label{hf:lya:redc:ful}
     \frac{d}{dt}\mathcal{L}_j^{1,h} +\min(1,2^{2j}) \mathcal{L}_j^{1,h}
    \le C    \biggl( \normeinf{\partial_t S^0_{11}(U)}+ \sum_{\alpha}\normeinf{\pal (S^\alpha_{11} (U))}\biggr) \mathcal{L}_j^{1,h}\\
    +\normede{\left(R^{11}_j,H^1_j,H^2_j,G^{11}_j, G^{12}_j,G^{13}_j\right)}\sqrt{\mathcal{L}_j^{1,h}}
     + \min(2^j,2^{-j})\normede{ G^1_j}\sqrt{\mathcal{L}_j^{1,h}}.
    \end{multline}
Let us now look at the right-hand side of \eqref{hf:lya:redc:ful}.
For the first term, the embedding $\BH{\frac{d}{2}}{2}{1} \hookrightarrow L^\infty$ and \eqref{pt:V:1:est} and \eqref{pt:V:2:est} (with $\kappa=\cd$) allow us to get
\begin{align}
\label{inf:es:hf:p_t}
 \normeinf{\partial_t S^0_{11}(U)}
 + \sum_{\alpha}\normeinf{\pal (S^\alpha_{11} (U))}&\lesssim   \normeinf{\pt V}+ \normeinf{\nabla V}
 \nonumber\\
 &\lesssim \NBHc{\nabla V} + \NBH{\cd+2}{2}{1}{V^2}+\NBH{\frac{d}{2}+1}{2}{1}{V}\NBH{\frac{d}{2}+1}{2}{1}{V^2}.\end{align}
The other terms may be treated by using commutator and product estimates in homogeneous Besov spaces. First, combining \eqref{compo:propo:base}  and Proposition \ref{propo_commutator-BH}, one has 
\begin{align}
\label{es:G11:j}
    \normede{ R^{11}_j } &\lesssim 2^{-(\cd+1)} c_j \NBHc{\nabla V}\NBH{\cd}{2}{1}{\nabla V^1}.
\end{align}
Next, using \eqref{product_propo4} combined with \eqref{compo:propo:base} and the smallness condition \eqref{smal:cnd:V}, one gets
\begin{equation}\label{es:G12:j}
   \normede{ G^{11}_j } \lesssim 2^{-j\sigma} c_j \biggl( \NBH{\cd}{2}{1}{V}\NBH{\sigma}{2}{1}{\n V^2}+ \NBH{\sigma}{2}{1}{V}\NBH{\cd}{2}{1}{\n V^2}\biggr)\esp{for all} \sigma>0.
\end{equation}
Moreover, from Bernstein's Inequality, we have
\begin{equation}    \label{G_2:est:sigma}
         \normede{ (G^{12}_j, G^{13}_j) } \lesssim   \normede{ (V^2_j, 2^j W_j) }.
     \end{equation}
Bounding $G^{14}_j$ relies on product estimate \eqref{product_propo4} and \eqref{compo:propo:base}. We get 
\begin{align}
\label{G:14:est:st:1}
     \normede{G^{14}_j} &\lesssim c_j \dm{\cd} \Vert V \Vert_{\BH{\frac{d}{2}}{2}{1} }
     \Vert \n V^1 \Vert_{\BH{\frac{d}{2}}{2}{1}}.
\end{align}
Next, we write that, by virtue of  Proposition \ref{propo_produc_BH} and Inequality  \eqref{compo:propo:base}
\begin{align*}
     \normede{G^5_j} \lesssim c_j \dm{\cd} \NBHc{V} \NBH{\cd}{2}{1}{\pt V^1}
\end{align*}
which, combined with \eqref{pt:V:1:est} with $\kappa=\cd $  and  \eqref{smal:cnd:V} leads to 
\begin{align}    \label{G:15:est:st:1}
         \normede{G^{15}_j} \lesssim c_j \dm{\cd} \Vert V \Vert_{\BH{\frac{d}{2}}{2}{1} }
     \Vert V \Vert_{\BH{\frac{d}{2}+1}{2}{1}}.\end{align}
Taking $\sigma=\cd$ in \eqref{es:G12:j}  and combining with  
\eqref{G_2:est:sigma}, \eqref{G:14:est:st:1}, \eqref{G:15:est:st:1}, we get that
\begin{align}
    \label{es=G1:cd}
         \normede{G^{1}_j}  \lesssim \normede{ (V^2_j, 2^j W_j) }+ c_j \dm{\cd}  \NBH{\cd+1}{2}{1}{V}\NBH{\cd}{2}{1}{ V}\cdotp
\end{align}
     
Taking advantage of Bernstein's inequality, we can bound the terms $H_j^1$ and $H^2_j$ in $L^2$. On the one hand, using H\"{o}lder inequality combined with the embedding  $\BH{\frac{d}{2}}{2}{1} \hookrightarrow L^\infty$ one may write
\begin{align*}
    \normede{H^1_j}\lesssim \NBH{\cd}{2}{1}{V}\normede{\DDj \biggl( (S^0_{11}(U))^{-1} S_{12}(D)Z(D)^{-1}S_{21}(D)V^1 \biggr)}\cdotp
\end{align*}
Now, the product estimate \eqref{product_propo4}, the composition estimate \eqref{comp:uv:propo:r=1} and the fact that the operator $S_{12}(D)Z(D)^{-1}S_{21}(D) $ is homogeneous of degree $0$ ensure that: for all $\sigma>0$
\begin{align*}
    \normede{H^1_j}\lesssim c_j 2^{-j\sigma}\NBH{\cd}{2}{1}{V}\biggl((1+\NBH{\cd}{2}{1}{V})\NBH{\sigma}{2}{1}{V^1}+ \NBH{\sigma}{2}{1}{V}\NBH{\cd}{2}{1}{V^1}\biggr)\cdotp
\end{align*}
which combined with the smallness condition \eqref{smal:cnd:V} implies that
\begin{align}
    \label{est:H1:j}
    \normede{H^1_j}\lesssim c_j 2^{-j\sigma}\NBH{\cd}{2}{1}{V}
    \biggl(\NBH{\sigma}{2}{1}{V^1} + \NBH{\sigma}{2}{1}{V}\NBH{\cd}{2}{1}{V^1}\biggr)\cdotp
\end{align}
On the other hand, the product estimate \eqref{product_propo4}, the composition estimate \eqref{comp:uv:propo:r=1} and the fact that the operator $S_{12}(D)Z(D)^{-1}S_{21}(D) $ is homogeneous of degree $0$ provide, for all $\sigma>0$ 
\begin{align}
    \label{est:H2:j}
    \normede{H^2_j}\lesssim c_j 2^{-j\sigma}(\NBH{\cd}{2}{1}{V}\NBH{\sigma}{2}{1}{V^1}+ \NBH{\sigma}{2}{1}{V}\NBH{\cd}{2}{1}{V^1})\cdotp
\end{align}

 Finally,  plugging \eqref{inf:es:hf:p_t}, \eqref{es:G11:j}, \eqref{G_2:est:sigma}, \eqref{es:G12:j}, \eqref{est:H1:j},\eqref{est:H2:j} (with $\sigma=\cd+1$), \eqref{es=G1:cd} into \eqref{hf:lya:redc:ful}, yields the desired inequality.
 
\end{proof}

$\bigstar$ \textbf{\textit{Global a priori estimates.}} 
Propositions \ref{pro:low:est:Vj:ge}, \ref{pro:es:W} and \ref{pro:est:HF:BFV} will enable us to derive bounds on $V$ in the space $E$ (defined in Theorem \ref{thm:glob:cri})  in terms of the initial data.

Before proceeding, introducing the following notation (with $N_0\in\mathbb N$) is in order:
\begin{align}
  \label{lya:LF:HF:et:gain}
    \begin{split}
         & \mathcal{L}\defn   \sum_{j\le N_0}2^{j(\cd-1)} \sqrt{\mathcal{L}^l_j}+  \sum_{j> N_0}2^{j(\cd+1)} \sqrt{\mathcal{L}_j^{1,h}}+ \sum_{j> N_0}2^{j\cd}\lVert W_j\rVert_{L^2_{\overline{S}^0_{22}}},\\
          & \mathcal{H}\defn   \NBH{\cd+1}{2}{1}{V}^{l,N_0} +\NBH{\frac{d}{2}+1}{2}{1}{V^1}^{h,N_0}+\NBH{\frac{d}{2}+2}{2}{1}{W}^{h,N_0}
    \end{split}
\end{align}
  with
    \begin{align}        \label{def:NBH:HF:N0}
        \NBH{s}{2}{1}{z}^{l,N_0}\defn \sum_{j\le N_0} 2^{js}\|\DDj z\|_{L^2}
        \esp{and} \NBH{s}{2}{1}{z}^{h,N_0}\defn \sum_{j> N_0} 2^{js}\|\DDj z\|_{L^2}. 
    \end{align}
    Note that \eqref{lya:fun:def:l}, \eqref{lya:fun:def:h} and  the fact that $ \overline{S}^0_{22}\simeq I_{n_2} $ guarantee that 
    \begin{align}
\label{equi:lya}
\mathcal{L}\simeq \NBH{\cd-1}{2}{1}{V}^{l,N_0}+ \NBH{\frac{d}{2}+1}{2}{1}{V^1}^{h,N_0}+ \NBH{\frac{d}{2}}{2}{1}{W}^{h,N_0}\cdotp
    \end{align}
Furthermore, \eqref{W:def} and Bernstein inequalities give 
\begin{align}
    \label{comp:Wj:Vj}
    \normede{W_j}\lesssim \normede{V_j},\;\; \text{for all}\;\;   j> 0  \esp{and} \normede{W_j}\lesssim 2^{-j}\normede{V_j},\;\; \text{for all}\;\; j\le 0,
\end{align}
and for all $s\in \mathbb{R}$
\begin{align*}
   \NBH{s}{2}{1}{V^2}^{h,N_0}&\le   \NBH{s}{2}{1}{W}^{h,N_0}+  \NBH{s}{2}{1}{|D|^{-1}(Z(D))^{-1}\left( S_{21}(D) V^1+  S_{22}(D)V^2\right)}^{h,N_0}\nonumber\\
    & \le \NBH{s}{2}{1}{W}^{h,N_0}+ C \NBH{s-1}{2}{1}{V^1}^{h,N_0}+C2^{-N_0} \NBH{s}{2}{1}{V^2}^{h,N_0}.
\end{align*}
Therefore for $N_0$ satisfying 
\begin{align}    \label{cnd:N0:1} 2^{-N_0}C \le \frac{1}{2},\end{align} 
one has
\begin{align}    \label{com:int:W:V:22}    \begin{split}
 & \NBH{\frac{d}{2}+2}{2}{1}{V^2}^{h,N_0}
    \lesssim \NBH{\frac{d}{2}+2}{2}{1}{W}^{h,N_0}+  \NBH{\frac{d}{2}+1}{2}{1}{V^1}^{h,N_0}\lesssim \mathcal{H}, \\
    &  \NBH{\frac{d}{2}}{2}{1}{V^2}^{h,N_0}
    \lesssim \NBH{\frac{d}{2}}{2}{1}{W}^{h,N_0}+  \NBH{\frac{d}{2}-1}{2}{1}{V^1}^{h,N_0}\lesssim \mathcal{L} \cdotp     \end{split}\end{align}
Note that in particular, we  have
 \begin{align}
\label{equi:lyabis}
\mathcal{L}\simeq \NBH{\cd-1}{2}{1}{V}^{l}+ \NBH{\frac{d}{2}+1}{2}{1}{V^1}^{h}+ \NBH{\frac{d}{2}}{2}{1}{V^2}^{h}
    \end{align}
    so that is enough to bound  the Lyapunov functional $\mathcal{L} $.
    \medbreak
Let us further recall the following inequalities: 
\begin{align}
    \label{emb:lh:hf:BH}
    \NBH{s}{2}{1}{z}^{l,N_0}\le C 2^{N_0(s-s')} \NBH{s'}{2}{1}{z}^{l,N_0} \esp{and} \NBH{s'}{2}{1}{z}^{h,N_0}\le C 2^{N_0(s'-s)} \NBH{s}{2}{1}{z}^{h,N_0} \esp{for} s'\le s,
\end{align}
and the following interpolation inequalities  for all $\varsigma\in [0,1]$: 
\begin{align}
    \label{int:ine:BH:cd}
   \Vert V^2 \Vert_{\BH{3\varsigma+\cd-1}{2}{1} }\le \NBH{\cd-1}{2}{1}{V^2}^{1-\varsigma}\NBH{\cd+2}{2}{1}{V^2}^\varsigma \esp{and}   \Vert V^1 \Vert_{\BH{\cd}{2}{1}}\le \sqrt{\NBH{\cd-1}{2}{1}{V^1}\NBH{\cd+1}{2}{1}{V^1}}\,.
\end{align}
Our goal is to bound ${\mathcal L}(t)$ and $\int_0^t{\mathcal H}(t)$ for all $t\in\mathbb R_+,$ in terms of the initial data.
Let us start with the low frequencies. As $d\ge 2$, taking $s=\cd-1$ in Inequality \eqref{lo:fr:est:s}, applying lemma \ref{lem_der_int}, multiplying by $2^{j(\frac{d}{2}-1)}$,  summing up on $j\le N_0$ and denoting 
\begin{align*}
    \mathfrak{V}^{l,N_0}\defn \sum_{j\le N_0}2^{j(\cd-1)} \sqrt{\mathcal{L}^l_j},
\end{align*}
we arrive for some positive constants $c_0,C$ and all $t\in\mathbb R_+$ at
\begin{equation*}
      \mathfrak{V}^{l,N_0}(t) +c_0\int^t_0 \NBH{\frac{d}{2}+1}{2}{1}{V(t)}^{l,N_0} \le  \mathfrak{V}^{l,N_0}(0)+ C\int^t_0\NBH{\cd}{2}{1}{V} 
       \|V\|_{\dot B^{\frac d2}_{2,1}\cap \dot B^{\frac d2+1}_{2,1}}\cdotp    \end{equation*}
     Next, after using the interpolation inequalities \eqref{int:ine:BH:cd} combined with \eqref{emb:lh:hf:BH}, \eqref{com:int:W:V:22} and \eqref{equi:lya}, the previous inequality gets simplified into
  \begin{align}    \label{2:lo:fr:est:s}
        \mathfrak{V}^{l,N_0}(t)+c_0\int^t_0 \NBH{\frac{d}{2}+1}{2}{1}{V(t)}^{l,N_0} \leq
         \mathfrak{V}^{l,N_0}(0)+ \int^t_0\mathcal{L}\mathcal{H}.   \end{align}
Let us next focus on the high frequencies. Let
\begin{equation*}
    \mathfrak{V}^{1,h,N_0}\defn  \sum_{j> N_0}2^{j(\cd+1)} \sqrt{\mathcal{L}_j^{1,h}}. 
\end{equation*}
Combining  Proposition \ref{pro:est:HF:BFV}, Lemma \ref{lem_der_int}, multiplying by $2^{j(\cd+1)}$ and summing up on $j>N_0$ gives (for some positive constants $c'_0$ and $C$)
\begin{align*}
    \begin{split}
        \mathfrak{V}^{1,h,N_0}(t)&+ c'_0\int^t_0  \NBH{\cd+1}{2}{1}{V^1}^{h,N_0}\le  \mathfrak{V}^{1,h,N_0}(0)+ C\int^t_0 \NBH{\cd+1}{2}{1}{V}\NBH{\cd}{2}{1}{ V}\\
        &+C\int^t_0  \bigl( \NBH{\cd+1}{2}{1}{V}+\NBH{\cd+2}{2}{1}{V^2}+ \NBH{\cd+1}{2}{1}{V}\NBH{\cd+1}{2}{1}{V^2}\bigr) \NBH{\cd+1}{2}{1}{V^1}^{h,N_0}\\
        &+2^{-N_0}C\int^t_0\NBH{\cd+2}{2}{1}{V^2}^{h,N_0}+C\int^t_0\NBH{\cd+2}{2}{1}{W}^{h,N_0}+C\int^t_0 \bigl( \NBH{\cd}{2}{1}{V}\NBH{\cd+2}{2}{1}{ V^2}+ \NBH{\cd+1}{2}{1}{V}^2\bigr)\\
  &+C\int^t_0\bigl(\NBH{\cd}{2}{1}{V}\NBH{\cd+1}{2}{1}{V^1} + \NBH{\cd+1}{2}{1}{V}\NBH{\cd}{2}{1}{V^1}\bigr)\cdotp
    \end{split}
\end{align*}
Hence, since \eqref{com:int:W:V:22} is satisfied,  using  \eqref{emb:lh:hf:BH}, \eqref{int:ine:BH:cd} and \eqref{equi:lya} enables us to simplify the previous Inequality as  
\begin{multline*}
      \mathfrak{V}^{1,h,N_0}(t)+ c'_0\int^t_0  \NBH{\cd+1}{2}{1}{V^1}^{h,N_0}\le  \mathfrak{V}^{1,h,N_0}(0)+ C\int^t_0\mathcal{L}\mathcal{H}+ C\int^t_0\mathcal{L}^2\mathcal{H}\\
      +2^{-N_0}C\int^t_0\NBH{\cd+2}{2}{1}{V^2}^{h,N_0}+C\int^t_0\NBH{\cd+2}{2}{1}{W}^{h,N_0} \cdotp
\end{multline*}
 Using again \eqref{com:int:W:V:22} for the penultimate term in the right-hand side of the previous Inequality, we end up with,
\begin{multline}
    \label{est:V1:hf:fina}
      \mathfrak{V}^{1,h,N_0}(t)+ c'_0\int^t_0  \NBH{\cd+1}{2}{1}{V^1}^{h,N_0}\le \mathfrak{V}^{1,h,N_0}(0)+ C\int^t_0\mathcal{L}\mathcal{H}+ C\int^t_0\mathcal{L}^2\mathcal{H}\\
      + 2^{-N_0}C \int^t_0\mathcal{H}+ C\int^t_0\NBH{\cd+2}{2}{1}{W}^{h,N_0} \cdotp
\end{multline}
 Next, let us simplify as much as possible the estimate given by Proposition \ref{pro:es:W}. Taking $\sigma=\frac{d}{2}$ in Inequality \eqref{W:es:-1}, summing up on $j> N_0$ 
 and setting
 $$\mathfrak{W}^{h,N_0}\defn \sum_{j>N_0}2^{\cd j} \lVert W_j\rVert_{L^2_{\overline{S}^0_{22}}}
 $$
 provides 
\begin{multline*}
           \mathfrak{W}^{h,N_0}(t)+  \widetilde{c}_0\int^t_0\NBH{\frac{d}{2}+2}{2}{1}{W}^{h,N_0}\le \mathfrak{W}^{h,N_0}(0)+ 2^{-N_0}C\int^t_0\left(   \NBH{\frac{d}{2}+1}{2}{1}{  V}+ \Vert V^2\Vert_{ \BH{\frac{d}{2}+2}{2}{1} } \right)+ C\int_0^t \NBH{\frac{d}{2}+1}{2}{1}{ V}^2\\
  + C\int_0^t\NBHc{V} \left(\NBH{\frac{d}{2}+1}{2}{1}{V}+\Vert V^2\Vert_{\BH{\frac{d}{2}+2}{2}{1} }+\NBH{\frac{d}{2}+1}{2}{1}{ V}^2\right),
    \end{multline*}
   where $\widetilde{c}_0$ and $C$ are  positive constants.
Next, using  \eqref{emb:lh:hf:BH}, \eqref{int:ine:BH:cd} combined with \eqref{com:int:W:V:22} and eventually the smallness condition \eqref{smal:cnd:V}, we can eliminate some redundant terms.
We get
\begin{equation}
    \label{W:es:3}
          \mathfrak{W}^{h,N_0}(t)+ \widetilde{c}_0\int^t_0 \NBH{\frac{d}{2}+2}{2}{1}{W}^{h,N_0}\le  \mathfrak{W}^{h,N_0}(0)
          + C\int^t_0 \bigl(\mathcal{L}+\mathcal{L}^2\bigr)\mathcal{H}
        + 2^{-N_0}C\int^t_0  \mathcal{H}.
      \end{equation}
For some  $0<\varepsilon<1$ that will be fixed hereafter, denote
\begin{align}
    \label{de:lya:fina}
    \begin{split}
    \widetilde{\mathcal{L}}&\defn  \mathfrak{V}^{l,N_0} + \varepsilon  \mathfrak{V}^{1,h,N_0}+  \mathfrak{W}^{h,N_0}\cdotp
     \end{split}
\end{align}
 Then, putting together  \eqref{2:lo:fr:est:s}, \eqref{est:V1:hf:fina} and \eqref{W:es:3} we can find a constant $\kappa_0$ such that  
\begin{multline}
    \label{es:lya:1}
          \widetilde{\mathcal{L}}(t)+  \kappa_0\int^t_0 \NBH{\frac{d}{2}+1}{2}{1}{V(t)}^{l,N_0}+\varepsilon \kappa_0\int^t_0  \NBH{\cd+1}{2}{1}{V^1}^{h,N_0}+ \kappa_0\int^t_0\NBH{\frac{d}{2}+2}{2}{1}{W}^{h,N_0} \le  \widetilde{\mathcal{L}}(0)\\
         + C\int^t_0 \bigl(\mathcal{L}+\mathcal{L}^2\bigr)\mathcal{H}
        +2^{-N_0}C\int^t_0  \mathcal{H} +\varepsilon  C\int^t_0\NBH{\cd+2}{2}{1}{W}^{h,N_0}. 
    \end{multline}
   Choosing   $\varepsilon$ and $N_0$ so that
   \begin{align}
       \label{es:ee:kapa_0}
       \varepsilon C < \kappa_0/2\esp{and} 2^{-N_0}C\le \frac{\kappa_0}{4},
   \end{align}
 the last two terms of \eqref{es:lya:1}  may be absorbed by left-hand side. We deduce that 
   
   \begin{multline}
    \label{W:es:5}
    \widetilde{\mathcal{L}}(t)+  \frac{\kappa_0}{2}\int^t_0 \NBH{\frac{d}{2}+1}{2}{1}{V(t)}^{l,N_0}+\frac{\kappa_0}{4} \int^t_0  \NBH{\cd+1}{2}{1}{V^1}^{h,N_0}+ \frac{\kappa_0}{2}\int^t_0\NBH{\frac{d}{2}+2}{2}{1}{W}^{h,N_0} \\\le  \widetilde{\mathcal{L}}(0)
         +          C\int^t_0 \bigl(\mathcal{L}+\mathcal{L}^2\bigr)\mathcal{H}.
    \end{multline}
  
We claim that there exists $\alpha>0$ such that if $\widetilde{\mathcal{L}}(0)<\alpha$ then, for all $t \in  [0,T]$, the left-hand side of Inequality \eqref{W:es:5} is smaller than $\widetilde{\mathcal{L}}(0)$.
Indeed, let us choose $\alpha< \min(1,\frac{\kappa_0}{16C} )$ so that $ \mathcal{L}(0)<\alpha $ implies that \eqref{smal:cnd:V} is satisfied, and set 
\begin{equation}   \label{def:T0:GS}
    T_0\defn \sup\Big\{ T_1\in [0,T],  \text{ such that } \sup_{0\le t\le T_1}\widetilde{\mathcal{L}}(t)\le \alpha\Big\}\cdotp
\end{equation}
The above set is nonempty (as $0$ is in it) and contains its supremum since $\widetilde{\mathcal{L}}$ is continuous
(remember that we assumed that $V$ is smooth). Hence, starting from \eqref{W:es:5}, we have
\begin{align*}
 \widetilde{\mathcal{L}}(t)+  \frac{\kappa_0}{2}\int^t_0 \NBH{\frac{d}{2}+1}{2}{1}{V(t)}^{l,N_0}+\frac{\kappa_0}{4} \int^t_0  \NBH{\cd+1}{2}{1}{V^1}^{h,N_0}+ \frac{\kappa_0}{2}\int^t_0\NBH{\frac{d}{2}+2}{2}{1}{W}^{h,N_0}
         &\le   \widetilde{\mathcal{L}}(0)+2C \alpha\int^t_0 \mathcal{H}.
\end{align*}
    Using the smallness of $ \widetilde{\mathcal{L}}(0)$, one may conclude that $ \widetilde{\mathcal{L}}<\alpha$ on $[0,T_0]$. As $\widetilde{\mathcal{L}}$ is
continuous, we must have $T_0 = T$ and  thus the following estimate  holds on $[0, T ]$:
\begin{align}
    \label{est:fin:GE}
    \widetilde{\mathcal{L}}(t)+  \frac{\kappa_0}{8}\int^t_0 \NBH{\frac{d}{2}+1}{2}{1}{V(t)}^{l,N_0}+\frac{\kappa_0}{8} \int^t_0  \NBH{\cd+1}{2}{1}{V^1}^{h,N_0}+ \frac{\kappa_0}{8}\int^t_0\NBH{\frac{d}{2}+2}{2}{1}{W}^{h,N_0} &\le  \widetilde{\mathcal{L}}(0).
\end{align}

Clearly, time $t=0$ does not play any particular role, and one can apply the same
argument on any sub-interval of $[0, T]$ which leads  for $0\le t_0\le t\le T$ to
\begin{align}
    \label{est:fin:GE:T0}
    \widetilde{\mathcal{L}}(t)+  \frac{\kappa_0}{8}\int^t_{t_0} \NBH{\frac{d}{2}+1}{2}{1}{V(t)}^{l,N_0}+\frac{\kappa_0}{8} \int^t_{t_0}  \NBH{\cd+1}{2}{1}{V^1}^{h,N_0}+ \frac{\kappa_0}{8}\int^t_{t_0}\NBH{\frac{d}{2}+2}{2}{1}{W}^{h,N_0} &\le  \widetilde{\mathcal{L}}(t_0).
\end{align}
Note that the estimates \eqref{est:fin:GE}, \eqref{com:int:W:V:22} and the fact that $2^{js}\simeq 2^{js'}$ for all $s,s'\in [\cd-1,\cd+2]$ and $ 0\le j\le N_0$ ensure that 
\begin{align}
    \label{est:V:E}
    \lVert V \rVert_{E_t}\le C \biggl(
     \NBH{\cd-1}{2}{1}{V(t_0)}^l+\NBH{\cd+1}{2}{1}{V^2(t_0)}^h+\NBH{\cd+1}{2}{1}{V^1(t_0)}^h \biggr)\esp{for all } 0\le t_0\le t\le T
\end{align}
where $ E_t$ is the space defined in Theorem \ref{thm:glob:cri} pertaining to the interval 
$[0,t]$.\smallbreak
Owing to \eqref{comp:Wj:Vj} and \eqref{est:V:E} the low frequencies of $W$ also satisfy, for all $ 0\le t_0\le t\le T$ 
 \begin{align}
        \label{est:W:LF:V0}
        \NBH{\cd}{2}{1}{W(t)}^l+ \int^t_{t_0} \NBH{\cd+2}{2}{1}{W}^l\le  C \biggl(
     \NBH{\cd-1}{2}{1}{V(t_0)}^l+\NBH{\cd+1}{2}{1}{V^2(t_0)}^h+\NBH{\cd+1}{2}{1}{V^1(t_0)}^h \biggr)\cdotp
    \end{align}


 \subsection{Proving the existence and uniqueness parts   of  Theorem \ref{thm:glob:cri}}
 \label{subsec:exis:glob:uni:BH}
 Having the a priori estimate \eqref{est:fin:GE}  at hand, constructing a global solution obeying Inequality \eqref{ine:lya:V:LF:HF:BH} for any data $V_0$ satisfying \eqref{smal:cnd:V} follows
from rather standard arguments. First, in order to benefit from the classical theory of symmetric partially hyperbolic diffusive systems, 
we remove the very low and very high frequencies of $Z_0$ so as to have initial
data belonging to nonhomogeneous Besov spaces.
More precisely, we set for all $p\in \mathbb{N}$, 
\begin{align}
    \label{def:smo:data:BH}
    V_{0,p}\defn (\Dot{S}_{p}-\Dot{S}_{-p}) V_0 \esp{with} \Dot{S}_q\defn \chi (2^{-q}D),\quad q\in\mathbb Z.
\end{align}
For each $p\in \mathbb{N}$, and $s\ge \cd+2$,
the data $  V_{0,p}=( V_{0,p}^1, V_{0,p}^2)$ belong to the \emph{nonhomogeneous} Besov space
$\B{s+1}{2}{1}\times\B{s}{2}{1}$ (see the definition in \eqref{def:NBH:NB}). Consequently, \cite[Theorem 1.2]{DanADOloc} provides us with a unique maximal
solution $(V^1_p,V^2_p)\in \cC([0,T^p[; \B{s+1}{2}{1}\times\B{s}{2}{1})$. Moreover $(V^1_p,V^2_p)\in \cC^1([0,T^p[; \B{s-1}{2}{1}\times\B{s-2}{2}{1})$. 
Taking advantage of the a priori estimates \eqref{est:W:LF:V0} and \eqref{est:V:E}, and denoting by $\mathbb{V}_p$ the function $\mathbb{V}$ defined in \eqref{ mathbb{V}:def} pertaining to $V_p$, we get
$ \mathbb{V}_p\le C \mathbb{V}_{0,p} $ as long as $V_p$ satisfies the smallness condition \eqref{smal:cnd:V}. Owing to the definition of $V_{0,p}$, we have
$ \mathbb{V}_{0,p}\le C \mathbb{V}_{0} $  and, obviously, $\Vert V_p \Vert_{L^\infty_{T^p}(\B{\cd}{2}{1})}\le \mathbb{V}_{p}(T^p)$. Hence, using a classical bootstrap 
argument, one can conclude that, if $\mathbb{V}_{0}$ is small enough, then
\begin{align}
    \label{est:mathbb:V:V0}
    \mathbb{V}_{p}(t)\le C \mathbb{V}_{0} \text{   for all } t\in [0,T^p[.
\end{align}
In order to show that the solution $V_p$ is global (that is $T^p=\infty$), we take advantage of the continuation criterion in \cite[Rem. 1.3]{DanADOloc}. In fact the interpolation inequality \eqref{int:ine:BH:cd}, embedding Inequality \eqref{emb:lh:hf:BH} and eventually the estimate \eqref{est:mathbb:V:V0} imply that if, for some $ T<\infty$
\begin{align}
    \label{blow:Vn:BH}
    \begin{split}
   & \bullet U_p([0,T)\times \Rd) \text{  is contained  in a compact subset of         } \mathcal{U},\\
&\bullet  \int^{T}_0\Bigl(\normeinf{\n V_p}^2+ \normeinf{\pt V_p}\Bigr)<\infty,
   \end{split}
\end{align}
then the solution $V_p$ may be continued beyond $[0, T^p[$. 
Note that \eqref{blow:Vn:BH} is satisfied, owing to \eqref{est:mathbb:V:V0} and embedding.
Hence $T^p=\infty$ and \eqref{est:mathbb:V:V0} is thus satisfied for all time.
\smallbreak 
To prove the convergence of $(V_p)$ to some $V$, one may for instance show that 
for all $T>0,$ $(V_p)$ is a Cauchy sequence in the space
\begin{align*}
    \mathcal{F}_T\defn\{ (V^1,V^2)\in L^\infty(0,T;\BH{\cd}{2}{1}\times\BH{\cd-1}{2}{1}) \text{  and  } V^2 \in L^1(0,T;\BH{\cd}{2}{1})\}.
\end{align*}
Adapting the proof of \cite[section 2.4]{DanADOloc} to our case where instead of nonhomogeneous Besov space, we use the corresponding homogeneous Besov spaces ensures that $(V_p)$ is a Cauchy sequence in $\mathcal{F}_T$ and thus has a limit $V$ in this space, and then passing to the
limit in \eqref{Eq_b} is straightforward. Compactness arguments can also be used (see \cite[Chap. 4]{HajDanChe11}).

Furthermore, time continuity of the solution and $\mathbb{V}(T)\le C\mathbb{V}_0 $ (for V),  for all $T >0$ may be obtained by  adapting the arguments of \cite[Chap. 4]{HajDanChe11}. This completes the proof of the existence part of Theorem \ref{thm:glob:cri}. As for uniqueness, it suffices to adapt  the proof of \cite[section 2.4]{DanADOloc}.

\subsection{Proof of Theorem \ref{thm:decay:sucri}}
The overall strategy is inspired by the joint work of the second author with T. Crin-Barat in \cite[Theorem 2.2]{BaratDan22M},  and Z. Xin and J. Xu in \cite{XinXU21}.

$\blacksquare$ \textit{First step: Uniform bounds in $\BH{-\sigma_1}{2}{\infty}$.} 
The proof of Inequality \eqref{est:prog:NBH(V):r=inf} starts from estimate \eqref{0:lya:lf:est} on the Lyapunov functional $\mathcal{L}^l_j$ that has been defined
on  \eqref{lya:fun:def:l}. 
After using Lemma \ref{lem_der_int} omitting the second (nonnegative) term 
of the left-hand side, then multiplying by $2^{-j\sigma_1}$ and taking the supremum on $\mathbb{Z}$, we end up for all $t\ge 0$ with 
\begin{align}
    \label{est:B:sig:1:gen}
   \NBH{-\sigma_1}{2}{\infty}{V(t)}\le  \NBH{-\sigma_1}{2}{\infty}{V_0}+\int^t_0  \NBH{-\sigma_1}{2}{\infty}{(f,F^1,F^2,F^t )},
\end{align}
where $F^1,$ $F^2$ and $F^t$ have been defined in \eqref{F_T:F_1:F_2}, and $f$ 
satisfies Assumption \textbf{D}.
\smallbreak
In order to bound the term $F^1$, we use Inequality \eqref{product:propo:uniq}. It holds for $ \cd<\sigma_1\le \cd$,
\begin{align*}
    \NBH{-\sigma_1}{2}{\infty}{F^1}\lesssim \suma  \NBH{-\sigma_1}{2}{\infty}{S^\alpha(U)-\overline{S}^\alpha }  \NBH{\cd}{2}{1}{\n V}.
\end{align*}
In order to bound $S^\alpha(U)-\overline{S}^\alpha  $ in $\BH{-\sigma_1}{2}{\infty}$, one cannot use directly Inequality \eqref{compo:propo:base} as $ -\sigma_1 $ may be
negative. However, applying \eqref{comp:uv:propo:inf} we can still obtain if \eqref{smal:cnd:V} is satisfied,
\begin{align}
\label{comp:est:nega:applic}
    \NBH{-\sigma_1}{2}{\infty}{S^\alpha(U)-\overline{S}^\alpha } \lesssim    \NBH{-\sigma_1}{2}{\infty}{V },
\end{align}
whence
\begin{align*}
    \NBH{-\sigma_1}{2}{\infty}{F^1}\lesssim   \NBH{-\sigma_1}{2}{\infty}{V }\NBH{\cd+1}{2}{1}{ V}.
\end{align*}
Next, remembering the form of $F^2$, then combining Inequality \eqref{product:propo:uniq} (recall that $1-d/2<\sigma_1\le d/2$) and the same composition estimate (applied to $r^{\alpha\beta}(U)$) as in \eqref{comp:est:nega:applic} yields
\begin{align*}
     \NBH{-\sigma_1}{2}{\infty}{F^2}&\lesssim \sum_{\beta=1}^d   \NBH{-\sigma_1+1}{2}{\infty}{r^{\alpha\beta}(U)\pbe V^2 }\\&\lesssim \NBH{1-\sigma_1}{2}{\infty}{V} \NBH{\cd}{2}{1}{\n V^2}\lesssim  (\NBH{-\sigma_1}{2}{\infty}{V}^l+  \NBH{\cd}{2}{1}{V}^h) \NBH{\cd+1}{2}{1}{ V^2}.
\end{align*}
Concerning $F^t$, we have, keeping \eqref{pt:V:1:est} and \eqref{pt:V:2:est} in mind that
\begin{align*}
     \NBH{-\sigma_1}{2}{\infty}{F^t} &\lesssim   \NBH{-\sigma_1}{2}{\infty}{ V } \NBH{\cd}{2}{1}{\pt V}\\
     &\lesssim  \NBH{-\sigma_1}{2}{\infty}{ V } (\NBH{\cd+1}{2}{1}{V}+ \NBH{\cd+2}{2}{1}{V^2}+ \NBH{\cd+1}{2}{1}{V}^2).
\end{align*}
 Similarly, the term $f$ may be bounded as follows: 
\begin{align*}
    \NBH{-\sigma_1}{2}{\infty}{f}&\lesssim (1+\NBH{\cd}{2}{1}{V})  \NBH{-\sigma_1}{2}{\infty}{ \n V\otimes\n V }\\
   & \lesssim \NBH{-\sigma_1}{2}{\infty}{ \n V }  \NBH{\cd}{2}{1}{ \n V}\lesssim (\NBH{-\sigma_1}{2}{\infty}{  V }^l + \NBH{\cd}{2}{1}{  V }^h) \NBH{\cd+1}{2}{1}{ V}.
\end{align*}
Thus, regrouping all those estimates, we obtain
\begin{align*}
  \NBH{-\sigma_1}{2}{\infty}{V(t)}\le  \NBH{-\sigma_1}{2}{\infty}{V_0}+\int^t_0  \mathcal{Y}_1+ \int^t_0 \NBH{-\sigma_1}{2}{\infty}{V} \mathcal{Y}_2
\end{align*}
with
\begin{align*}
    \mathcal{Y}_1  \defn \NBH{\cd}{2}{1}{V}^h \NBH{\cd+1}{2}{1}{V}
    \esp{and}
    \mathcal{Y}_2\defn \NBH{\cd+1}{2}{1}{V}+  \NBH{\cd+1}{2}{1}{V}^2+ \NBH{\cd+2}{2}{1}{V^2}.
\end{align*}
Since \eqref{est:fin:GE:T0} is satisfied, one can prove that  $ \mathcal{Y}_1,  \mathcal{Y}_2 \in L^1(\mathbb{R}_+).$
 Applying Gronwall inequality completes the proof of \eqref{est:prog:NBH(V):r=inf}.

Let us highlight that one has to justify that if $V_0$ is in $\BH{-\sigma_1}{2}{\infty}$ (in addition to \eqref{small:cnd:whole:space}), then
the solution constructed in Theorem \eqref{thm:glob:cri} is in $ \BH{-\sigma_1}{2}{\infty} $ for all time. This may be checked by following the construction scheme of the previous subsection. 

\medbreak
$\blacksquare$ \textit{Second step : proof of generic decay estimates.}
The functional
$$\widetilde{\mathcal{L}}: t\in \mathbb{R}_+ \mapsto  \mathfrak{V}^{l,N_0}(t) + \varepsilon  \mathfrak{V}^{1,h,N_0}(t)+  \mathfrak{W}^{h,N_0}(t),$$
 defined in \eqref{de:lya:fina}, with $\varepsilon$ and $N_0$ satisfying \eqref{es:ee:kapa_0}, is nonincreasing, as Inequality \eqref{est:fin:GE:T0} is satisfied. 
 Consequently $\widetilde{\mathcal{L}}$
is differentiable almost everywhere. It is worth noting that, as, \eqref{equi:lya}, the Inequalities \eqref{lya:fun:def:l}, \eqref{lya:fun:def:h} and  the fact that $ \overline{S}^0_{22}\simeq I_{n_2} $ imply that 
    \begin{align}
\label{equi:lya:tidle}
\widetilde{\mathcal{L}}\simeq \NBH{\cd-1}{2}{1}{V}^{l,N_0}+ \NBH{\frac{d}{2}+1}{2}{1}{V^1}^{h,N_0}+ \NBH{\frac{d}{2}}{2}{1}{W}^{h,N_0}\cdotp
\end{align}
Furthermore, one proved  in \eqref{est:fin:GE:T0} that there exists a positive constant $\kappa_0>0$ such that $\text{ for all } 0\le t_0\le t$. 
\begin{align*}
     \widetilde{\mathcal{L}}(t)+  \frac{\kappa_0}{8}\int^t_{t_0} \NBH{\frac{d}{2}+1}{2}{1}{V}^{l,N_0}+\frac{\kappa_0}{8} \int^t_{t_0}  \NBH{\cd+1}{2}{1}{V^1}^{h,N_0}+ \frac{\kappa_0}{8}\int^t_{t_0}\NBH{\frac{d}{2}+2}{2}{1}{W}^{h,N_0} &\le  \widetilde{\mathcal{L}}(t_0) \cdotp  
\end{align*}
From this, we can deduce that
\begin{align*}
    \frac{d}{dt}  \widetilde{\mathcal{L}} + \frac{\kappa_0}{8}\NBH{\frac{d}{2}+1}{2}{1}{V}^{l,N_0}+\frac{\kappa_0}{8}  \NBH{\cd+1}{2}{1}{V^1}^{h,N_0}+ \frac{\kappa_0}{8} \NBH{\frac{d}{2}+2}{2}{1}{W}^{h,N_0}\le 0\esp{a. e.  on}  \mathbb{R}_+. 
\end{align*}

Granted with this information and \eqref{est:prog:NBH(V):r=inf}, one can prove the decay estimates of Theorem \eqref{thm:decay:sucri} by
following the interpolation argument of \cite{XinXU21}. The starting point is that, provided, $-\sigma_1<d/2-1$, we have the following  interpolation inequality:
\begin{align*}
    \NBH{\cd-1}{2}{1}{V}^{l,N_0}\lesssim \biggl(  \NBH{-\sigma_1}{2}{1}{V}^{l,N_0} \biggr)^{\theta_0} \biggl(  \NBH{\cd+1}{2}{1}{V}^{l,N_0} \biggr)^{1-\theta_0} \text{     with   } \theta_0\defn \frac{2}{d/2+1+\sigma_1} \cdotp
\end{align*}
Inequality \eqref{est:prog:NBH(V):r=inf} thus implies that
\begin{align*}
\NBH{\cd+1}{2}{1}{V}^{l,N_0} \gtrsim   \biggl(  \NBH{\cd-1}{2}{1}{V}^{l,N_0} \biggr)^{\frac{1}{1-\theta_0}}
    \biggl(  \NBH{-\sigma_1}{2}{1}{V_0}^{l,N_0} \biggr)^{-\frac{\theta_0}{1-\theta_0}}.
\end{align*}
Similarly, one has, as $ 1-\sigma< \cd,$
\begin{align*}
\NBH{\cd+2}{2}{1}{W}^{h,N_0} \gtrsim  \biggl(  \NBH{\cd}{2}{1}{W}^{h,N_0} \biggr)^{\frac{1}{1-\theta_0}}
    \biggl(  \lVert W_0 \rVert_{\BH{1-\sigma_1}{2}{1}}^{h,N_0} \biggr)^{-\frac{\theta_0}{1-\theta_0}}\gtrsim  \biggl(  \NBH{\cd}{2}{1}{W}^{h,N_0} \biggr)^{\frac{1}{1-\theta_0}}
    \biggl(  \lVert W_0 \rVert_{\BH{\cd}{2}{1}}^{h,N_0} \biggr)^{-\frac{\theta_0}{1-\theta_0}}.
\end{align*} 
For the high frequencies term $\NBH{\cd+1}{2}{1}{V^1}^{h,N_0}$, using the estimate of Theorem \ref{thm:glob:cri}, one can just write:
\begin{align*}
\NBH{\cd+1}{2}{1}{V^1}^{h,N_0} \gtrsim  \biggl(  \NBH{\cd+1}{2}{1}{V^1}^{h,N_0} \biggr)^{\frac{1}{1-\theta_0}}
    \biggl(  \lVert V^1_0 \rVert_{\BH{\cd-1}{2}{1}\cap \BH{\cd+1}{2}{1}} \biggr)^{-\frac{\theta_0}{1-\theta_0}}\cdotp
\end{align*}
Hence, remembering \eqref{equi:lya:tidle}, we see that there exists a (small) constant $c$ such that
\begin{align*}
    \frac{d}{dt} \widetilde{\mathcal{L}}+c C_0^{-\frac{\theta_0}{1-\theta_0}} \widetilde{\mathcal{L}}^{\frac{1}{1-\theta_0}} \le 0 \ \text{     with   } \ C_0\defn  \NBH{-\sigma_1}{2}{1}{V_0}^{l,N_0} + \lVert V^1_0 \rVert_{\BH{\cd-1}{2}{1}\cap \BH{\cd+1}{2}{1}}+ \lVert W \rVert_{\BH{\cd}{2}{1}}^{h,N_0}.
\end{align*}
Integrating, this gives us
\begin{align*}
     \widetilde{\mathcal{L}(}t)\le \biggl(  1+ c\frac{\theta_0}{1-\theta_0}\left( \frac{\mathcal{L}(0)}{C_0}\right)^{\frac{\theta_0}{1-\theta_0}}t \biggr)^{1-\frac{1}{\theta_0}}\widetilde{\mathcal{L}}(0),
\end{align*}
whence, using again \eqref{equi:lya:tidle}, 
\begin{align}
    \label{deca:gn:lf:hf}
    \NBH{\cd-1}{2}{1}{V(t)}^{l,N_0}+  \NBH{\cd}{2}{1}{W(t)}^{h,N_0}+  \NBH{\cd+1}{2}{1}{V^1(t)}^{h,N_0}\lesssim (1+t)^{-\alpha_1}\mathbb{V}_0\, \text{     with  } \alpha_1\defn \frac{d/2-1+\sigma_1}{2}\cdotp
\end{align}
The previous Inequality combined with \eqref{com:int:W:V:22} provides
\eqref{decay2} and \eqref{decay3}. Inequality \eqref{decay1} follows from  Inequalities \eqref{deca:gn:lf:hf} and \eqref{est:prog:NBH(V):r=inf}, and interpolation.


\section{Proof of Theorem \ref{thm:glo:L2:cri}}\label{sec:proof:global:crit}
This section is dedicated to the proof of a refinement of Theorem \ref{thm:glob:cri} in the case where System \eqref{Eq_b} meets the additional conditions outlined in Assumption \textbf{E}. Let $V=(V^1,V^2)$ be a smooth (and decaying) solution of System \eqref{Eq_b} under Assumption \textbf{E}, satisfying $(\bar U^1+V^1)\in \mathcal{O}^1$ 
for some bounded open subset  $\mathcal{O}^1$ such that $\overline{\mathcal{O}^1}\subset\U^1.$
We assume in addition that
\begin{align}
    \label{smal:cnd:criti:V}
    \sup_{0\le t\le T} \NBH{\cd}{2}{1}{V^1(t)}\ll 1, \esp{for all} T<T^*,
\end{align}
which owing to the Besov embedding $ \BH{\cd}{2}{1} \hookrightarrow L^\infty$ implies that
\begin{align}
    \label{smal:inf:cnd:criti:V}
    \sup_{0\le t\le T} \normeinf{V^1(t)} \ll  1 \esp{for all} T<T^*.
\end{align} 

Let us start with the following result.
\medbreak
\paragraph{$\spadesuit$ \emph{Estimate for $\pt V$:} }
\begin{lem}
    \label{lem:est:pt:cri}
    Let $-\cd<\kappa\le \cd.$
     Under the assumptions of Theorem \ref{thm:glo:L2:cri} and \eqref{smal:cnd:criti:V}, we have for all  $j\in \mathbb{Z},$
   \begin{align}
       \label{est:pt:V:cri}
       \begin{split}
       &\bullet 2^{j\kappa}\normede{ \pt V^1_j }\lesssim 2^{j\kappa}\normede{\n V_j}+c_j\biggl( \NBH{\cd}{2}{1}{V^2}\NBH{\kappa}{2}{1}{\n V^1}+ \NBH{\cd}{2}{1}{V^1}\NBH{\kappa}{2}{1}{\n V^2}\biggr),\\
       &\bullet 2^{j\kappa}\normede{ \pt V^2_j }\lesssim 2^{j\kappa}\normede{\n V_j}  +c_j\biggl(\NBHc{V^1}\NBH{\kappa}{2}{1}{\n V^1} +\NBH{\cd}{2}{1}{V}\NBH{\kappa}{2}{1}{\n V^2}\\
       &\hspace{8cm}+ \NBH{\kappa+2}{2}{1}{ V^2}
       + \NBH{\kappa+1}{2}{1}{V^1}\NBH{\cd+1}{2}{1}{ V^2}\biggr)\cdotp
        \end{split}
   \end{align}
\end{lem}
\begin{proof}
     The second inequality of \eqref{est:pt:V:cri} relies on the following identity:
\begin{multline*}
     \pt V^2 = -(S^0_{22}(U^1))^{-1}\sum_{\alpha}{\left(S^\alpha_{21} (U^1) \pal V^1+ S^\alpha_{22} (U) \pal V^2\right)} \\
     +(S^0_{22}(U^1))^{-1}\sum_{\alpha,\beta}{ \pal (Z^{\alpha\beta}(U^1)) \pbe V^2  }+ \sum_{\alpha,\beta}{  (S^0_{22}(U^1))^{-1} (Z^{\alpha\beta}(U^1) \pal\pbe V^2)  }.
\end{multline*}
 Combining Inequalities \eqref{smal:cnd:criti:V}, \eqref{product_propo2} and  \eqref{compo:propo:base} we discover that for all $\alpha,\beta,$
 \begin{align*}
    & \bullet \NBH{\kappa }{2}{1}{\biggl( (S^0_{22}(U^1))^{-1}S^\alpha_{21} (U^1)- (\overline{S}^0_{22})^{-1}\overline{S}^\alpha_{21}\biggr) \pal V^1}\lesssim \NBH{\cd}{2}{1}{V^1} \NBH{\kappa }{2}{1}{\n  V^1},\\
     & \bullet  \NBH{\kappa }{2}{1}{(S^0_{22}(U^1))^{-1}{ \pal (Z^{\alpha\beta}(U^1)) \pbe V^2  } }\lesssim \NBH{\kappa+1}{2}{1}{V^1}\NBH{\cd+1}{2}{1}{V^2} ,\\
     &\bullet  \NBH{\kappa }{2}{1}{(S^0_{22}(U^1))^{-1}{(Z^{\alpha\beta}(U^1) \pal \pbe V^2)  } }\lesssim \NBH{\kappa+2}{2}{1}{V^2}.
 \end{align*}
 In order to handle the term  $ (S^0_{22}(U^1))^{-1} S^\alpha_{22} (U) \pal V^2$, instead of Inequality \eqref{compo:propo:base} we use \eqref{comp:u_v_1:2:BH:ine:prop:1} since $(S^0_{22}(U^1))^{-1} S^\alpha_{22} (U) $ is at most linear wth respect to $U^2$. It holds, for $\alpha=1,\cdots,d,$
 \begin{multline*}
      \NBH{\kappa }{2}{1}{\biggl((S^0_{22}(U^1))^{-1} S^\alpha_{22} (U)- (\overline{S}^0_{22})^{-1}\overline{S}^\alpha_{22}\biggr) \pal V^2}\\\lesssim \biggl( (1+\NBH{\cd}{2}{1}{V^1})\NBH{\cd}{2}{1}{V^2} + \NBH{\cd}{2}{1}{V^1} \biggr)\NBH{\kappa }{2}{1}{\n  V^2}.
 \end{multline*}
 Combining those estimates and using \eqref{smal:cnd:criti:V} gives the second inequality of \eqref{est:pt:V:cri}. Similar arguments lead to the first one.
\end{proof}

\paragraph{$\spadesuit$ \emph{Low frequencies analysis:}} The refinement of Proposition \ref{pro:low:est:Vj:ge} is the following 
\begin{lem}
\label{lem:lf:cri:L2:V}
   Under the assumptions of Theorem \ref{thm:glo:L2:cri} and \eqref{smal:cnd:criti:V}, we have for all  $s\in ]-\frac{d}{2},\frac{d}{2}]$ and $j\in \mathbb{Z}$, 
\begin{multline}
    \label{lo:fr:est:s:cri}
              \frac{d}{dt}  \mathcal{L}_j^l  + \min(1, 2^{2j})  \mathcal{L}_j^l \le 
               c_j \dm{s}\sqrt{\mathcal{L}_j^l}\biggl(
\NBH{\cd}{2}{1}{ V}\NBH{s+1}{2}{1}{V}   +\NBHc{V^1} \lVert V^2 \rVert_{\BH{s+1}{2}{1}\cap \BH{s+2}{2}{1}}\\
+ \NBH{s+1}{2}{1}{V^1} \lVert V^2 \rVert_{\BH{\cd+1}{2}{1}}\biggr)\cdotp
    \end{multline}
\end{lem}
\begin{proof}
  
Starting from \eqref{0:lya:lf:est}, the proof of \eqref{lo:fr:est:s:cri} is just a matter of re-estimating $F^1_{j}, F^t_{j}$ and $ F^2_{j}$ 
in $L^2$, for all $j\in \mathbb{Z}$, by taking into account Assumption \textbf{E}. 
Now, using \eqref{product_propo2} and  \eqref{compo:propo:base} we get for all $\cd<s\le \cd$ and  $\alpha=1,\cdots, d$
\begin{align*}
     \NBH{\cd}{2}{1}{\overline{S}^\alpha_{21}- S^\alpha_{21}(U^1)}\lesssim  \NBH{\cd}{2}{1}{V^1}.
\end{align*}
Since $ S^\alpha_{21}= {}^T\!S^\alpha_{12} $ we get the same estimate for $  \NBH{\cd}{2}{1}{\overline{S}^\alpha_{12}- S^\alpha_{12}(U^1)} $. Next taking into account the structure of $ S^\alpha_{11}$ and $ S^\alpha_{22}$ and using the composition estimate \eqref{comp:u_v_1:2:BH:ine:prop:1} yields
\begin{align}
\label{est:S:11:22:cri}
     \NBH{\cd}{2}{1}{\biggl(\overline{S}^\alpha_{11}- S^\alpha_{11}(U), \overline{S}^\alpha_{22}- S^\alpha_{22}(U)\biggr)}\lesssim  \NBH{\cd}{2}{1}{V},
\end{align}
whence 
\begin{align*}
\normede{F^1_{j}} &\lesssim c_j \dm{s}\suma\NBHc{ (\overline{S}^\alpha-S^\alpha(U))}\NBH{s}{2}{1}{\n V} \lesssim c_j \dm{s}\NBHc{ V}\NBH{s}{2}{1}{\n V}.
\end{align*}
Keeping inequalities \eqref{est:pt:V:cri} (with $\kappa=s$) and \eqref{smal:cnd:criti:V} in mind, we can write
\begin{align*}
   &  \normede{F^t_{j}} 
     \lesssim c_j \dm{s}\NBHc{V^1} \NBH{s}{2}{1}{\pt V}\nonumber\\
     &\lesssim c_j \dm{s}\NBHc{V^1}\left( \NBH{s}{2}{1}{\nabla V^1}(1+\lVert V^2 \rVert_{\BH{\cd}{2}{1}\cap \BH{\cd+1}{2}{1}})+ \NBH{s}{2}{1}{\n V^2}
       \NBH{\cd}{2}{1}{V}+  \lVert V^2 \rVert_{\BH{s+1}{2}{1}\cap \BH{s+2}{2}{1}}
   \right)\cdotp
\end{align*}
Combining the decomposition \eqref{decom:F2<=Leib} (recall that $  r^{\alpha\beta}(U)$ depend only on $U^1$, with $  r^{\alpha\beta}(\overline{U}^1)=0 $)
and the product estimate \eqref{product_propo2} and the composition law \eqref{compo:propo:base} (if $s+1>0$) or \eqref{comp:uv:propo:r=1} and \eqref{smal:cnd:criti:V} (if $1-d/2<s+1\le 0$), one has
\begin{align}
\label{est:lf:cri:F2:j}
     \normede{F^2_{j}} \lesssim c_j\dm{s} \left( \NBH{s+1}{2}{1}{V^1} \NBH{\cd+1}{2}{1}{V^2}+ \NBH{\frac{d}{2}}{2}{1}{V^1} \NBH{s+2}{2}{1}{V^2}\right)\cdotp
\end{align}
Putting together those estimates gives us \eqref{lo:fr:est:s:cri}.
\end{proof}

\paragraph{$\spadesuit$ \emph{High frequencies analysis: Parabolic mode.}}  
\begin{propo}
    \label{prop:para:mode:crit}
    Under the assumptions of Theorem \ref{thm:glo:L2:cri}
    and using the notation \eqref{eq:weight}, we have the following estimate for the parabolic mode $W$ defined in \eqref{W:def}:  for all $j\ge 0$, 
    \begin{multline}
    \label{ine:para:mode:cri}
      2^{j(\cd-1)}   \|{W_j(t)}\|_{L^2_{\overline S_{22}^0}}+ 2^{j(\cd+1)}\overline{c}_1\int^t_0 \normede{W_j}\le 2^{j(\cd-1)}  \|{W_j(0)}\|_{L^2_{\overline S_{22}^0}}\\
    + Cc_j\int^t_0 \biggl( \NBH{\cd}{2}{1}{ V}^2  +\NBHc{V^1}\lVert V^2 \rVert_{\BH{\cd}{2}{1}\cap \BH{\cd+1}{2}{1}}\biggr)
      +C2^{j(\cd-1)}  \int^t_0 \normede{V_j}.
\end{multline}

\end{propo}
\begin{proof}
Starting from \eqref{W:es:1}, using  Lemma \ref{lem_der_int} with $X= \intd\overline{S}^0_{22} W_j\cdot W_j$ then reverting to \eqref{est:ptS(D)V=ptV}
we get
 \begin{align}
 \label{est:Wj:cri}
       \|{W_j(t)}\|_{L^2_{\overline S_{22}^0}}+ 2^{2j}\overline{c}_1\int^t_0 \normede{W_j}\leq  \|{W_j(0)}\|_{L^2_{\overline S_{22}^0}}+\int^t_0\left(2^{-j}\normede{\pt V_j}+\normede{h_j }\right)
    \end{align}
with $h=h^t+h^{21}+h^{22}+h^2$ defined in \eqref{def:h}. Taking $ \kappa=\cd-1$ in \eqref{est:pt:V:cri} and using the smallness condition \eqref{smal:cnd:criti:V}, we discover that for all $j\in \mathbb{Z}$
\begin{align*}
 2^{j(\cd-1)}  \normede{\pt V_j} &\lesssim  2^{j\cd}\normede{V_j}+ c_j  \bigl( \NBH{\cd}{2}{1}{ V}^2+ \NBHc{V^1}\lVert V^2 \rVert_{ \BH{\cd+1}{2}{1}}\bigr)+\lVert V^2 \rVert_{ \BH{\cd+1}{2}{1}}\\
  &\lesssim   2^{j\cd}\normede{V_j}+ c_j  \bigl( \NBH{\cd}{2}{1}{ V}^2+\lVert V^2 \rVert_{ \BH{\cd+1}{2}{1}}\bigr)\cdotp
\end{align*}
Bearing in mind Assumption \textbf{E} and using Inequality \eqref{product_propo2} combined with \eqref{compo:propo:base} (or \eqref{comp:u_v:BH:ine:prop:1}),
we have the following outcome:
\begin{align*}
    \normede{ h^t }&\lesssim 2^{-j(\cd-1)}c_j\biggl( \NBH{\cd}{2}{1}{ V}^2  +\NBHc{V^1}\lVert V^2 \rVert_{\BH{\cd}{2}{1}\cap \BH{\cd+1}{2}{1}}\biggr) ,\\
    \normede{h^{21} } &\lesssim 2^{-j(\cd-1)}c_j \NBH{\cd}{2}{1}{V^1} \NBH{\cd-1}{2}{1}{\n V^1},\\
     \normede{h^{22} } &\lesssim 2^{-j(\cd-1)}c_j (\NBH{\cd}{2}{1}{V^1} + \NBH{\cd}{2}{1}{V^2} )\NBH{\cd-1}{2}{1}{\n V^2},\\
      \normede{h^{2} } &\lesssim 2^{-j(\cd-1)}c_j \NBH{\cd}{2}{1}{V^1} \NBH{\cd}{2}{1}{\n V^2}.
\end{align*}
Plugging the above inequalities in \eqref{est:Wj:cri},
multiplying by $2^{j(\frac d2-1)}$
and 
    using the fact that $2^{-j}\le 1$ for all $j\ge 0$ completes the proof of \eqref{ine:para:mode:cri}.
\end{proof}

\medbreak\paragraph{$\spadesuit$ \emph{High frequencies analysis: Estimates of $V^1$.}} 
The substitute of Proposition \ref{pro:est:HF:BFV} is the following.

\begin{propo}
\label{propo:hf:V:cri:}
Under the assumptions of Theorem \ref{thm:glo:L2:cri}, the following inequality holds true:  
 \begin{multline*}
 \frac{d}{dt}\mathcal{L}_j^{1,h} +\min(1,2^{2j}) \mathcal{L}_j^{1,h}
    \lesssim \NBH{\cd+1}{2}{1}{V^2}\mathcal{L}_j^{1,h}
  +\biggl( \normede{(V^2_j, 2^j W_j)}\\ + 2^{-j\cd}c_j \NBH{\cd}{2}{1}{V^1}\NBH{\cd+1}{2}{1}{ V^2} 
     + 2^{-j\cd}c_j \min(2^{2j},1) \NBH{\cd}{2}{1}{V^1}\NBH{\cd}{2}{1}{ V^2} \biggr)\sqrt{\mathcal{L}_j^{1,h}},
    \end{multline*}
 with $\mathcal{L}_j^{1,h}$ defined in \eqref{lya:fun:def:l}.
    \end{propo}

\begin{proof}
As $S^0_{11}={\rm Id},$
  performing the method leading to \eqref{hf:lya:redc:ful}, we discover that
    \begin{multline*}
 \frac{d}{dt}\mathcal{L}_j^{1,h} +\min(1,2^{2j}) \mathcal{L}_j^{1,h}
    \le C   \sum_{\alpha}\normeinf{\pal (S^\alpha_{11} (U))} \mathcal{L}_j^{1,h}\\
    +\normede{\left(R^{11}_j,G^{11}_j, G^{12}_j,G^{13}_j\right)}\sqrt{\mathcal{L}_j^{1,h}}
     + \min(2^j,2^{-j})\normede{ G^1_j}\sqrt{\mathcal{L}_j^{1,h}}.   
    \end{multline*}
    Since, from Assumption \textbf{E}, $ S^\alpha_{11} (U) $ is affine with respect to $U^2$, we get
    \begin{align*}
        \normeinf{\pal (S^\alpha_{11} (U))} \lesssim \normeinf{\n V^2}\lesssim \NBH{\cd+1}{2}{1}{V^2}.
    \end{align*}
Next, taking advantage of \eqref{comm:est:a:b} the term $\normede{R^{11}_j}$ can be bounded as follows: 
\begin{align*}
    \normede{R^{11}_j}\lesssim 2^{-j\cd}  c_j\NBH{\cd}{2}{1}{\n V^2}\NBH{\cd}{2}{1}{V^1}.
\end{align*}
The fact that $S^\alpha_{12}(U)$  depends only on $U^1$ combined with the Inequalities \eqref{product_propo2}, \eqref{compo:propo:base} and \eqref{smal:cnd:criti:V} ensures that 
\begin{align*}
    \normede{G_j^{11} }\lesssim  2^{-j\cd}  c_j\NBH{\cd}{2}{1}{V^1}\NBH{\cd}{2}{1}{\n V^2} \esp{and}  \normede{G_j^{11} }\lesssim  2^{-j(\cd-1)}  c_j\NBH{\cd}{2}{1}{V^1}\NBH{\cd-1}{2}{1}{\n V^2} .
\end{align*}
Similarly, keeping in mind that $ S^{\alpha}_{11}(U)$ is at most linear with respect to $U^2$ one has
\begin{align*}
    \normede{G_j^{14} } \lesssim 2^{-j(\cd-1)}  c_j\NBH{\cd}{2}{1}{V^2}\NBH{\cd-1}{2}{1}{\n V^1}.
\end{align*}
Combining the previous inequalities and \eqref{G_2:est:sigma}   completes the proof of the proposition.
 \end{proof}

\paragraph{$\spadesuit$ \emph{Conclusion}} 
Taking  $s=\cd-1$ in \eqref{lo:fr:est:s:cri}, then applying Lemma \ref{lem_der_int}, we get for all $j\le N_0$:
$$\sqrt{{\mathcal L}_j^l(t)} + c2^{2j}  \int^t_0 \sqrt{{\mathcal L}_j^l}  \leq  
    \sqrt{{\mathcal L}_j^l(0)}
              + c_j2^{-j(\cd-1)}\int^t_0\biggl(
\NBH{\cd}{2}{1}{ V}^2   +\NBHc{V^1} \lVert V^2 \rVert_{\BH{\cd}{2}{1}\cap \BH{\cd+1}{2}{1}} \biggr)\cdotp
    $$
    At the same time, applying  Lemma \ref{lem_der_int} to the inequality of Proposition \ref{propo:hf:V:cri:} and taking advantage of the fact that $ {\mathcal{L}}_j^1 \simeq  \normede{V^1_j}^2$  yields for all $j\ge N_0>0$
\begin{multline*}
  \sqrt{{\mathcal L}_j^{1,h}(t)}+ \int^t_0   \sqrt{{\mathcal L}_j^{1,h}}
    \leq \sqrt{{\mathcal L}_j^{1,h}(0)} +\int^t_0  \NBH{\cd+1}{2}{1}{V^2}\normede{V^1_j}\\
   +c_j2^{-j\cd} \int^t_0 \NBHc{V^1}\bigl(\NBH{\cd}{2}{1}{V^2}+\NBH{\cd+1}{2}{1}{V^2}\bigr)
    +  \normede{(V^2_j, 2^j W_j)}\cdotp
    \end{multline*}
    For $0<\varepsilon<1$,  let us set 
     \begin{align}
         \label{de:Lya:cri}
         \begin{split}
        &\widetilde{L}\defn \sum_{j\le N_0 }2^{j(\cd-1)}  \sqrt{{\mathcal L}_j^{l}}
            + \sum_{j> N_0} \biggl(2^{j(\cd-1)} \|{W }\|_{L^2_{\bar S^0_{22}}}+ \varepsilon 2^{j\cd}  \sqrt{{\mathcal L}_j^{1,h}}\biggr) \esp{and}\\
        & \widetilde{H} \defn \sum_{j\le N_0 }2^{j(\cd+1)} \normede{V_j}+ \sum_{j> N_0} \left(2^{j(\cd+1)} \normede{W_j}+ \varepsilon 2^{j\cd}\normede{V_j^1}\right)\cdotp
        \end{split}
     \end{align}
Putting together the previous two inequalities and \eqref{ine:para:mode:cri}
and using the notation $\NBH{s}{2}{1}{\cdot}^{h,N_0}$  defined  in \eqref{def:NBH:HF:N0},
we end up with
\begin{multline}
\label{est:Lya:cri:0}
\widetilde{L}(t)+c\int^t_0 \widetilde{H} \leq \widetilde{L}(0) + \int^t_0\Bigl(
\NBH{\cd}{2}{1}{ V}^2   +\NBHc{V^1} \lVert V^2 \rVert_{\BH{\cd}{2}{1}\cap \BH{\cd+1}{2}{1}} \Bigr)\\
 + 2^{-N_0}\int^t_0 (\NBH{\cd+1}{2}{1}{V^2}^{h,N_0} +\NBH{\cd+1}{2}{1}{V^1}^{h,N_0})+ \varepsilon \int^t_0 \NBH{\cd+1}{2}{1}{W}^{h,N_0} \cdotp
\end{multline}
 In order to close our estimate we need to exhibit a bound of  $\int^t_0 \NBH{\cd+1}{2}{1}{V^2}^{h,N_0}$. This can  be accomplished by utilizing Inequality \eqref{com:int:W:V:22}
 (with index $\cd+1$ instead of $\cd+2$). First by choosing $\varepsilon>0$ sufficiently small, the last term on the r.h.s. of the previous Inequality can be absorbed by the left-hand side. Next, by selecting  $N_0 $ large enough and using \eqref{com:int:W:V:22}, the second integral on the r.h.s. of the previous Inequality can also be absorbed into the left-hand side. Owing to the smallness condition \eqref{smal:cnd:criti:V} and interpolation inequality, the term $\int^t_0\NBHc{V^1} \lVert V^2 \rVert_{\BH{\cd+1}{2}{1}} $ can likewise be bounded by the left hand side of \eqref{est:Lya:cri:0}. Therefore, \eqref{est:Lya:cri:0} transforms into:

 \begin{align}
\label{est:Lya:cri:1}
\widetilde{L}(t)+\frac c2\int^t_0 \widetilde{H} \le \widetilde{L}(0) +C \int^t_0
\NBH{\cd}{2}{1}{ V}^2 \cdotp
 \end{align}
Owing to interpolation inequalities, one can prove that the integral on the right-hand side of
\eqref{est:Lya:cri:1} is  dominated  by $\int^t_0 \widetilde{L}\widetilde{H}$.  Hence one can conclude exactly as in the previous section that if  $ \NBH{\cd-1}{2}{1}{V^2_0}+ \lVert V^1_0 \rVert_{\BH{\cd-1}{2}{1}\cap \BH{\cd}{2}{1}} $ is small enough, then there exist  some (new) positive real numbers $c_0$ and $C$ such that
\begin{align}
\label{lya:est:cri}
      \widetilde{L}(t)+ c_0\int^t_0  \widetilde{H} 
      \le  \widetilde{L}(0)\esp{for all} 0\le t\le T\cdotp
\end{align}
Using \eqref{com:int:W:V:22} and \eqref{lya:est:cri} and the fact that $2^{j\cd}\simeq 2^{j(\cd-1)}\simeq 2^{j(\cd+1)}$ for all $0\le j\le N_0$, we deduce that for all $0\le t\le T,$
\begin{multline}
\label{data:est:cri}
\NBH{\cd-1}{2}{1}{V(t)}^l+\NBH{\cd-1}{2}{1}{(W(t),V^2(t))}^h+ \NBH{\cd}{2}{1}{V^1(t)}^h\\+\int^t_0 \biggl( \NBH{\cd+1}{2}{1}{V}^l+\NBH{\cd+1}{2}{1}{(W,V^2)}
+\NBH{\cd}{2}{1}{V^1}^h\biggr)
     \le C\biggl(   \NBH{\cd-1}{2}{1}{V(0)}^l+\NBH{\cd-1}{2}{1}{V^2_0}^h+ \NBH{\cd}{2}{1}{V^1_0}^h\biggr)\cdotp
     \end{multline}
Since, according to \eqref{W:def}, we have 
\begin{equation}\label{eq:WdtV}
S^0_{22}(U)\partial_tV^2=\Delta Z(D)W+h^{21}+h^{22}+h^2
\end{equation}
with the functions $h^{21},$ $h^{22}$ and $h^2$ defined in \eqref{def:h}, 
taking advantage of the usual product estimates and composition laws
ensures that $\LpNBH{\cd-1}{2}{1}{\partial_tV^2}{1}$ is also 
bounded by the right-hand side of \eqref{data:est:cri}.

From this point, the rest of the proof follows standard methods. We begin by regularizing the initial data as in \eqref{def:smo:data:BH}. 
Then, each regularized initial data
$V_{0,p}$ is in the \emph{nonhomogneous} Besov space $B^\cd_{2,1}\times B^{\cd+1}_{2,1}$ and \cite[Theorem 1.2]{DanADOloc}  provides us with a unique smooth maximal solution $V_p=(V^1_p,V^2_p)$ on $[0,T^p)$ such that for all $T<T^p$, 
\begin{equation}\label{eq:above}V^1_p\in \cC([0,T];B^{\cd+1}_{2,1}),\quad V^2_p\in \cC([0,T];B^{\cd}_{2,1})\cap L^1_T(B^{\cd+2}_{2,1})\esp{and} \partial_t V_p\in L^1([0,T];B^{\cd}_{2,1}).\end{equation}
Since this solution is (relatively) smooth and \eqref{small:data:cri:thm}
holds for all $p\in\mathbb N,$ it satisfies \eqref{data:est:cri}
for all $t<T^p.$  This implies in particular that the smallness condition 
\eqref{smal:cnd:criti:V} holds on $[0,T^p[$ and that
\begin{equation}\label{eq:blowup3}\int_0^{T^p}\|\n V^2_p\|_{\dot B^{\cd}_{2,1}}\,dt<\infty.
\end{equation}
\smallbreak
In order to prove that $T^p=\infty$, we need the following Lemma.
\begin{lem}
    \label{lem:blow:cri}
     Under Assumption \textbf{E}, let $V= (V^1,V^2)\in \cC([0,T^*);\B{\cd+1}{2}{1}\times\B{\cd}{2}{1})$ be a solution of  \eqref{Eq_b}  with the regularity described at \eqref{eq:above}
     for all $T<T^*.$ There exists a (small) positive constant $\eta$
      such that if 
        \begin{align}\label{eq:blowup1}
         &\int_0^{T^*}\|\n V^2\|_{\dot B^{\cd}_{2,1}}+\int_0^{T^*}\NBH{\cd-1}{2}{1}{\pt V^2}<\infty,\\\label{eq:blowup2}
       \esp{and} &\underset{0\le t<T^*}\sup\|V^1(t)\|_{\dot B^{\cd}_{2,1}}\leq\eta,
         \end{align}
    then $V$ may be continued beyond $T^*$ as a solution with the regularity corresponding
    to \eqref{eq:above}.
\end{lem}
\begin{proof}
    The proof is very close to that of \cite[Theorem 1.2]{DanADOloc}. 
    Arguing as for \cite[Proposition 2.1]{DanADOloc}, we can establish that for all $t<T^*$ and 
    $j\in\mathbb Z,$
    we have:
     \begin{multline*}
         \Vert \Dj V^1(t)\Vert_{L^2} \le \normede{\Dj V^1_{0}}+ C\int^t_0 \normeinf{\n V^2} \normede{\Dj V^1}\\+ \int_0^t\Vert ( R^{11}_j ,\suma S_{12}^\alpha(U^1)\pal V^2 )\Vert_{L^2}\ \esp{with}    R_j^{11} \defn \suma[(S_{11}^\alpha(U^2),\Dj](\pal V^1).
     \end{multline*}
 To bound $R^{11}_j$ we use the commutator estimate (A.6) in \cite{DanADOloc} while the last term in the r.h.s. of the previous inequality can be bounded from \eqref{comp:uv:propo:r=1}.
    Owing to \eqref{eq:blowup2}, we finally  get
       \begin{align}
    \label{cri:V1:blow}
        \NBH{\cd+1}{2}{1}{V^1(t)}\le \NBH{\cd+1}{2}{1}{V^1_0}+C\int^t_0\NBH{\cd+1}{2}{1}{V^1} \NBH{\cd}{2}{1}{\n V^2}+ C\int^t_0  \NBH{\cd+2}{2}{1}{V^2}.
    \end{align}
   Next, remembering that $V^2$ satisfies \eqref{lin:part:V2}
   (with  $ h^t,h^{21},h^{22},h^{2}$ defined in \eqref{def:h} and $f^2=0$)
   and following the lines of the proof of \cite[Prop. B.2 or prop. 2.2 ]{DanADOloc},  we discover
   that for all $0<t<T,$
   \begin{multline} \label{cri:V2:blow}
        \NBH{\cd}{2}{1}{V^2(t)}+\int^t_0 \NBH{\cd+2}{2}{1}{V^2}+ \int^t_0 \NBH{\cd}{2}{1}{\pt V^2} \lesssim \NBH{\cd}{2}{1}{V^2_0}\\
        +\suma\int^t_0  \NBH{\cd}{2}{1}{\biggl( \overline{S}_{21}^\alpha\pal V^1, \overline{S}_{22}^\alpha\pal V^2\biggr) }
        + \int^t_0 \NBH{\cd}{2}{1}{(V^2,h^t,h^{21},h^{22},h^{2})}.  \end{multline}
Using the product estimate \eqref{product_propo1} for $h^2$ and estimate (A.15) in \cite{DanADOloc}  for $h^t$, it holds that  
\begin{align*}
    &\NBH{\cd}{2}{1}{h^2}\lesssim \NBH{\cd}{2}{1}{V^1}\NBH{\cd+1}{2}{1}{\nabla V^2}+\NBH{\cd+1}{2}{1}{V^1}\NBH{\cd+1}{2}{1}{V^2} ,\\
    &\NBH{\cd}{2}{1}{h^t}\lesssim\NBH{\cd}{2}{1}{V^1}\NBH{\cd}{2}{1}{\pt V^2}+\NBH{\cd+1}{2}{1}{V^1}\NBH{\cd-1}{2}{1}{\pt V^2}.
\end{align*}
Next, since  $S_{21}^\alpha$ and $S_{21}^\alpha$  are affine with respect to $U^2,$
they are of the form $\mathfrak{S}_1(U^1)V^2+\mathfrak{S}_2(U^1)$ where $ \mathfrak{S}_1  $ and $ \mathfrak{S}_2$ are smooth.
Taking advantage of \eqref{product_propo1}, we thus have for $k=1,2$
\begin{align*}
    \NBH{\cd}{2}{1}{S_{2k}^\alpha(U)- \overline{S_{2k}^\alpha}}    
    \lesssim \NBH{\cd}{2}{1}{V^1}+ (1+\NBH{\cd}{2}{1}{V^1})\NBH{\cd}{2}{1}{V^2},\end{align*}
which implies, owing to \eqref{eq:blowup2}, that
$$\NBH{\cd}{2}{1}{h^{21}}+\NBH{\cd}{2}{1}{h^{22}}\lesssim \NBH\cd{2}{1}V
\NBH\cd{2}{1}{\n V}.$$
Plugging all this information into \eqref{cri:V2:blow} and using \eqref{eq:blowup2} yields
\begin{multline*}
    \NBH{\cd}{2}{1}{V^2(t)}+\int^t_0 \NBH{\cd+2}{2}{1}{V^2}+ \int^t_0 \NB{\cd}{2}{1}{\pt V^2} 
\lesssim \NBH{\cd}{2}{1}{V^2_0}\\
+\int_0^t\Bigl(  \NBH{\cd}{2}{1}{V^2}+\NBH{\cd}{2}{1}{\n V}
+\NBH{\cd+1}{2}{1}{V^1}(\NBH{\cd+1}{2}{1}{V^1}+\NBH{\cd-1}{2}{1}{\pt V^2})+\NBH{\cd}{2}{1}{V}\NBH{\cd}{2}{1}{\n V}\Bigr)\cdotp
\end{multline*}
Using again \eqref{eq:blowup2} and combining with Inequality \eqref{cri:V1:blow}, 
we end up for all $t\in[0,T[$ with 
\begin{multline*}
    \NBH{\cd}{2}{1}{\nabla V^1(t)}+\NBH{\cd}{2}{1}{V^2(t)}+\int^t_0 \NBH{\cd+2}{2}{1}{V^2}
\lesssim \NBH{\cd}{2}{1}{\nabla V^1_0}+\NBH{\cd}{2}{1}{V^2_0}
\\+\int_0^t\Bigl(1+ \NBH{\cd}{2}{1}{\n V^2}+\NBH{\cd-1}{2}{1}{\pt V^2}\Bigr)\Bigl(\NBH{\cd+1}{2}{1}{V^1}+\NBH{\cd}{2}{1}{V^2}\Bigr) 
+\int_0^t\NBH{\cd}{2}{1}{\n V^2}.\end{multline*}
 Remembering \eqref{eq:blowup1} and using  Gronwall Lemma enables us to say that 
 $$\underset{t\in[0,T[}\sup\bigl(\NBH{\cd}{2}{1}{\nabla V^1(t)}+\NBH{\cd}{2}{1}{V^2(t)}\bigr)+\int^T_0 \NBH{\cd}{2}{1}{\nabla^2V^2}<\infty,$$
 and applying \cite[Rem. 1.4]{DanADOloc} and embedding allows to continue the solution beyond $T.$
 \end{proof}
 Combining this Lemma with the global estimates obtained in \eqref{data:est:cri} and the smallness assumption \eqref{small:data:cri:thm}, we conclude that
 the lifespan $T^p$ of $V^p$ is infinite  for all $p.$

To complete the proof of existence, one can argue as in the previous section. For uniqueness, see \cite[Theorem 1.6]{DanADOloc}. In fact, there,  it is shown that uniqueness
holds true (without smallness condition) in the set of functions $(V^1,V^2)$ such that $ V^1\in C([0,T];\BH{\cd}{2}{1})  $ and $V^2\in C([0,T];\BH{\cd-1}{2}{1})\cap L^1_T(\BH{\cd+1}{2}{1})$.

\section{Application to the MHD system}
\label{sec:MHD}


As explained in e.g. \cite[Chapter 6]{Kawashima83} or \cite[Chapter 3]{LiQin12},
the motion of a viscous, compressible, and heat conducting magnetohydrodynamic (abbreviated as MHD)  flow in $\mathbb{R}^3$ can be described by the following full compressible 
MHD equations\footnote{Below, for any differentiable function $S=S(\ro,\theta)$
we use the notation  $S_\ro\defn \frac{\partial S}{\partial \ro}$ and $ S_\theta\defn \frac{\partial S}{\partial \theta}\cdotp$}: 
\begin{align}
    \label{MHD}
     \begin{cases}
         &  \pt \ro+ \dig(\ro u )=0,\\
          &   \ro \pt u + \ro u\cdot \n u- \dig( 2\mu D(u)+ \lambda I_3\dig(u))+\n p=-\frac{1}{\mu_0}(\rota{B}) \times B,\\
          &\ro e_\theta(\pt \theta +u\cdot \n \theta)+\theta p_\theta \dig(u)-\dig(k \n \theta)={\mathbb T}+\frac{1}{\sigma \mu_0^2}|\rota B|^2,\\
          &\pt B +u\cdot\n B+B\dig{u}-B\cdot\n u-\frac{1}{\sigma \mu_0}\Delta B=-\n (\frac{1}{\sigma\mu_0})\times \rota{B}.
     \end{cases}
 \end{align}
Here $u(t,x) \in \mathbb{R}^3$ denotes the velocity field, $\ro=\ro(t,x)\in \mathbb{R}_+$ the density, $p = p(t,x) \in \mathbb{R}$ the pressure,  $e = e(t, x) \in \mathbb{R}$ the internal energy by unit mass, $\theta = \theta (t, x) \in \mathbb{R}_+$ the absolute temperature and $ B(t,x)\in  \mathbb{R}^3 $ the magnetic induction. We denoted by 
 $D(u) \defn \frac{1}{2} (\n u + {}^T\n u)$  the deformation tensor and ${\mathbb T}$ is the viscous dissipation function, given  by
\begin{align}
\label{def:vaphi:Dx}
  {\mathbb T} \defn 
    \frac{\mu}{2}\sum_{i,j=1}^3(\partial_{x_j} u^i+ \partial_{x_i} u^j)^2+ \lambda (\dig u)^2.
\end{align}
The real numbers $\lambda$ and $\mu$ are the viscosity coefficients depending on the thermodynamic quantities $\ro$ and $\theta$; $\sigma=\sigma(\ro,\theta)$ is called the coefficient of electrical conductivity and $\mu_0$ the magnetic permeability.  
\smallbreak
In order to reduce  \eqref{MHD} to a closed system of $8$ equations with the $8$ unknowns $(\ro,u,\theta,B),$ we make the following (physically relevant):
\medbreak\paragraph{\textbf{Assumption G}}
\begin{enumerate}
    \item The thermodynamic quantities $p$ and $e$ are smooth functions 
    of $\ro>0$ and $\theta>0,$  satisfying 
    \begin{align}
        \label{positive:Pr:eth}
        p_\ro(\ro,\theta) =\frac{\partial p}{\partial\ro}>0 \esp{and} e_\theta(\ro,\theta)= \frac{\partial e}{\partial\theta}>0.
    \end{align}
    \item The coefficients $\mu$, $\lambda$ and  $k$ are smooth functions of $(\ro,\theta)$
that satisfy 
    \begin{align}
        \label{LAme:cnd}
       \mu>0,\esp{ } \nu\defn 2\mu+\lambda>0 \esp{and} k>0.
    \end{align}
    \item The coefficient $\sigma$ is a smooth function of $(\ro,\theta)$ satisfying
    \begin{align}
        \label{positi:sig}
      \sigma>0 .
    \end{align}
\end{enumerate}
System \eqref{MHD} is supplemented with initial data $\rho_0,$ $u_0,$ $\theta_0$ and $B_0$
such that $\dig{B_0}=0.$ Then, for smooth enough solutions, applying the operator $\dig$ on \eqref{MHD}$_4$ and using  the identities: 
  \begin{align*}
    &  -(u\cdot \n) B-B\dig{u}+ (B\cdot\n ) u=\rota{(u\times B)}-u\dig{B},\\
      & \frac{1}{\sigma \mu_0}\Delta B-\nabla (\frac{1}{\sigma \mu_0})\times \rota{B}=-\rota{( \frac{1}{\sigma \mu_0} \rota{B})}+\frac{1}{\sigma \mu_0}\nabla \dig{B},
  \end{align*}
  we discover that $\dig{B}$ is solution of the following transport-diffusion equation
  \begin{align*}
      (\dig{B})_t+\dig{(u\dig{B})}=\dig{\biggl((\frac{1}{\sigma\mu_0})\n \dig{B}\biggr)}\cdotp
  \end{align*}
Hence, MHD system has the following property:
\smallbreak\textbf{(P)} A solution of \eqref{MHD} 
   such that $\dig{B_0}=0$ satisfies $\dig{B}=0$ as long as it remains smooth.
\medbreak
In terms of $U\defn(\ro,u,\theta,B),$ \eqref{MHD} may be seen as a symmetrizable hyperbolic partially diffusive system on the phase space
$\mathcal{U}\defn \{(\ro,u,\theta, B)\in 
    \mathbb{R}^{8}/ \ro>0,\; \theta>0\}.$ Indeed, we have:
\begin{align}
    \label{SHPDS:NS}
    S^0(U)\frac{d}{dt} U+\sum_{\alpha=1}^3 S^\alpha(U)\pal U-\sum_{\alpha,\beta=1}^3 \pal\left( Y^{\alpha\beta}\pbe V\right)=f(U)
\end{align}
where
\begin{align*}
    S^0(U)\defn 
    \begin{pmatrix}
\frac{p_\ro}{\ro} & 0 & 0&0\\
0 & \ro {I_3} & 0&0\\
0&0 & \frac{\ro e_\theta}{\theta}&0\\
0&0&0& \frac{1}{\mu_0} {I_3}
\end{pmatrix},\;\;\;
 f(U)\defn  \begin{pmatrix}
0\\
0 \\
\frac{1}{\theta}({\mathbb T}+\frac{1}{\sigma\mu_0}|\rota{B}|^2)-k\n \theta\cdot \n (\frac{1}{\theta})\\
-\n(\frac{1}{\sigma\mu_0^2})\rota{B}-\left( \n(\frac{1}{\sigma\mu_0^2})\cdot \n \right)B
\end{pmatrix},
\end{align*}

\begin{align*}
\sum_{\alpha=1}^3 S^\alpha(U)\xi_\alpha \defn
 \begin{pmatrix}
\frac{p_\ro}{\ro}u\cdot \xi & p_\ro \xi & 0& 0\\
p_\ro \;^T\xi & \ro(u\cdot \xi) {I_3} & p_\theta\; ^T\xi& \frac{1}{\mu_0}(^T\xi B-(B\cdot \xi) I_3)\\
0& p_\theta \xi & \frac{\ro e_\theta}{\theta}(u\cdot \xi) & 0\\
0 & \frac{1}{\mu_0}(^T\! B\xi -(B\cdot\xi)I_3) & 0 & \frac{1}{\mu_0}(u\cdot \xi)I_3,
\end{pmatrix}
\end{align*}
\begin{align*}
  \esp{and}\!   Y^{\alpha\beta}\defn
    \begin{pmatrix}
0& 0 \\
0 &  Z^{\alpha\beta}
\end{pmatrix}
\!\esp{with}\!
\sum_{\alpha,\beta=1}^3 Z^{\alpha\beta}\xi_\alpha\xi_\beta\defn 
\begin{pmatrix}
 \mu |\xi|^2 {I_3}+(\mu\!+\!\lambda)\xi\otimes \xi  & 0& 0\\
0 & \frac{k}{\theta} |\xi|^2& 0\\
0&0& \frac{1}{\mu_0^2\sigma}|\xi|^2I_3
\end{pmatrix}
\end{align*}
where the relation $B\times \rota{B}=\frac{1}{2}\n(|B|^2)-(B\cdot \n )B$ has been used.
\medbreak

It is clear that $S^0(U)$ is a positive definite matrix for all $U\in \mathcal{U}$, and that the matrices $S^\alpha(U)$ are real symmetric. Furthermore a simple calculation reveals that
for all  $A=(X_1,Y,X_2)$ with $X_1\in \mathbb{C}^3,$ $X_2\in \mathbb{C}^3$
and $Y\in\mathbb R,$ and all $\xi \in \mathbb{R}^3$ with $|\xi|=1,$ we have 
\begin{align}
    \label{sca:Yal:be:V}
     \sumab \left\langle Z^{\alpha\beta}\xi_\alpha\xi_\beta A, A\right\rangle\ge \min(\mu,\nu)|X_1|^2+\frac{k}{\theta}Y^2+  \frac{1}{\mu_0^2\sigma} |X_2|^2
\end{align}
where $\left\langle \cdot ,\cdot \right\rangle$ denotes the standard Hermitian product in $\mathbb{C}^{7}$. As for the right-hand side~$f$ of \eqref{SHPDS:NS}, it can be regarded as a lower order (quadratic) term satisfying \textbf{(D4)}.
\smallbreak
We claim that Condition (SK) is satisfied at any $U\in\U.$
Indeed, let $\phi=(a,A)$ with $a\in \mathbb{C}$ and consider $\omega\in \mathbb{S}^{2}.$ As $S^0(U)$ is invertible, if $ (S^0)^{-1}(U)(\sum_{\alpha,\beta=1}^3  Y^{\alpha\beta}(U)\omega_\alpha\omega_\beta) \phi=0_{\mathbb{C}^8}  $, then 
$$ \Big(\sum_{\alpha,\beta=1}^3  Z^{\alpha\beta}(U)\omega_\alpha\omega_\beta\Big) A=0_{\mathbb{C}^7}, $$
 which, owing to \eqref{sca:Yal:be:V}, \eqref{LAme:cnd} and \eqref{positi:sig}  implies that $ A =0_{\mathbb{C}^7}$. Now, if there exists $\lambda\in \mathbb{C}$ such that
 $ (S^0)^{-1}(U) (\sum_{\alpha=1}^3 S^\alpha(U)\omega_\alpha )\phi=\lambda\phi $, then  $ a \frac{p_\ro}{\ro}{}^T\!\omega =0_{\mathbb{C}^2} $. Since $ {p_\ro}>0 $ (Assumption \textbf{G}). Then, owing to $\omega\neq 0,$ we have $a=0$ and thus $\phi=0_{\mathbb{C}^8}.$
\medbreak
Consequently, we have the following result which is a direct application of Theorem~\ref{thm:glob:cri}:

\begin{thm}[Global existence for (MHD) system]
    \label{thm:glo:MHD}
   Let Assumption \textbf{G} be in force and fix some positive
   real numbers $\overline\theta$ and $\overline B.$ Assume that the initial data $(\ro_0,u_0,\theta_0,B_0)$ are such that    
   $\dig{B_0}=0.$ 
   There exist  positive constants $c$ and $C$ depending only on 
  the coefficients of the system such that if
\begin{align}
    \label{small:cnd:whole:space:MHD}
  \Vert \ro_0-\overline{\ro} \Vert_{ \BH{\frac{1}{2}}{2}{1}\cap\BH{\frac{5}{2}}{2}{1} }+  \Vert (u_0,\theta_0-\overline{\theta}, B_0-\overline{B}) \Vert_{ \BH{\frac{1}{2}}{2}{1}\cap\BH{\frac{3}{2}}{2}{1} }\le c
\end{align}
then  System \eqref{MHD} supplemented with initial data $(\ro_0,u_0,\theta_0,B_0)$ admits a unique global-in-time solution $(\ro,u,\theta, B)$
such that  $V=(V^1,V^2)$ with  $V^1\defn\ro-\overline{\ro}$ and $V^2\defn(u,\theta-\overline{\theta},B-\overline{B})$ belongs to the space $E$ of Theorem \ref{thm:glob:cri}, 
and  Inequality \eqref{ine:lya:V:LF:HF:BH} holds. 

If, in addition, $ (\ro_0-\overline{\ro}, u_0,\theta_0-\overline{\theta}, B_0-\overline{B}  )$ belongs to $\BH{-\sigma_1}{2}{\infty}$ for some $-1/2<\sigma_1\le 3/2$ then the solution $ V$ satisfies
\eqref{est:prog:NBH(V):r=inf} and the decay estimates of Theorem \ref{thm:decay:sucri}.
\end{thm}
If we put $\theta=0$ and $B=0$ in \eqref{MHD}
and keep only the first two equations, we get  the so-called barotropic compressible Navier-Stokes system:
\begin{align}
     \label{CNS:ro}
     \begin{cases}
         &  \pt \ro+ u\cdot\n \ro+ \ro\dig(u)=0\\
          &   \ro \pt u + \ro u\cdot \n u- \dig( 2\mu(\ro) D u+ \lambda(\ro) I_d\dig(u))+\n p(\ro)=0.
     \end{cases}
 \end{align}
Without difficulties, we can see that, if 
\begin{align}
    \label{Lame:cnd:mu:nu}
   \mu(\ro)>0,\;\;\; \nu(\ro)>0  \esp{and} p_\ro (\ro)>0 \;\; \text{ for all } \ro>0,
\end{align}
then the system \eqref{CNS:ro} satisfies both Assumption \textbf{E} and the (SK) condition. Then, Theorem \ref{thm:glo:L2:cri} can be applied and we recover the following result first proved by the second author in \cite{Dan00}: 
\begin{thm}
    \label{thm:glo:CNS}
   Let $d\ge 2$ and \eqref{Lame:cnd:mu:nu} be assumed.  Fix some constant and positive reference
   density $\bar\rho.$ 
   There exist  positive constants $c$ and $C$ depending only on 
   $\mu,$ $\lambda,$ $\bar\rho$ and $p$ such that
   if
\begin{align}
    \label{small:cnd:whole:space:NSC}
  \Vert \ro_0-\bar\ro \Vert_{ \BH{\cd-1}{2}{1}\cap\BH{\cd}{2}{1} }+  \Vert u_0 \Vert_{ \BH{\cd-1}{2}{1}}\le c
\end{align}
then System \eqref{CNS:ro} supplemented with initial data $(\ro_0,u_0)$ admits a unique global-in-time solution $(\ro,u)$ with $\ro>0$.  Moreover $(\ro, u)$  belongs to the class $\mathcal{E}$, defined in  Theorem \ref{thm:glo:L2:cri},   and satisfies the estimate \eqref{lya:cri:thm}.
\end{thm}

\medbreak
\paragraph{\textbf{Acknowledgments}}\textit{The first author has received funding from the European Union's Horizon 2020 research and innovation programme under the Marie Sk\l{}odowska-Curie grant agreement \textnumero~945332.}


  \appendix
\section{Some inequalities}
The goal of this part is twofold: first, we give an outline of the proof of  Lemma \ref{lem:derI} and
establish a basic linear algebra result which ensures that
if the pair $(N_\omega,M_\omega)$ defined in \eqref{def:matbf:N:M} satisfies Condition (SK), 
then so does $(\mathbf N_\omega,\mathbf M_\omega).$
Then, we recall a result about a differential inequality that has been used repeatedly in the paper. 
\medbreak
Let us start with the justification of Lemma \ref{lem:derI}.
Let $V$ satisfy Equation \eqref{lin:eq:gene}.
Then, performing the change of variable 
 \begin{align}\label{eq:rescaling}
     \tau\defn \frac{t\ro^b}{\kappa}\esp{and} r\defn\kappa\ro^{a-b}\esp{with}  \ro\defn|\xi|,
 \end{align}
 we discover that $v(\tau)\defn \widehat V(t,\xi) $ is solution of
 \begin{align*}
     v'+(r\mathcal{N}_\omega+\mathcal{M}_\omega)v=0\esp{with}   \mathcal{M}_\omega=S^{-1}\mathcal{B}_\omega,\; \mathcal{N}_\omega=S^{-1}\mathcal{A}_\omega\!\!\esp{and}\!\!\mathcal{A}_\omega, \mathcal{B}_\omega \!\!\esp{defined in}\!\!  \eqref{def:A:B:N:M_omega}.
 \end{align*}
 According to  \cite{BaratDan22M,Danc22},
 we can find arbitrarily small positive parameters $\ee_1,\cdots, \ee_{n-1}$ such that 
 \begin{align*}
      \mathcal{J}'_\omega\!+\!\frac{r}{2}  \sum^{n-1}_{k=1} \!\ee_k|\mathcal{M}_\omega \mathcal{N}_\omega^k v|^2\!\le \ee_0\max(r,r^{-1})|\mathcal{M}_\omega v|^2 \!\!\esp{with}\!\!   \mathcal{J}_\omega\defn   \sumk \!\ee_k   \mathcal{R}e (\mathcal{M}_\omega \mathcal{N}_\omega^{k-1} v \cdot \mathcal{M}_\omega \mathcal{N}_\omega^{k} v ).  
 \end{align*}
Taking the real part of the Hermitian product in $\C^n$ of the following equation
$$  Sv'+(r\mathcal{A}_\omega+\mathcal{B}_\omega)v=0$$
with $v$ and using the properties \eqref{null:A:sky} and \eqref{posi:ope:B:}, we readily get
$$\frac12\frac d{dt}\bigl(Sv\cdot v)+|\mathcal B_\omega v|^2\leq0,$$
whence 
$$\frac12\frac d{dt}\Bigl( Sv\cdot v+2\min(r,r^{-1})\mathcal J_\omega\Bigr)+|\mathcal B_\omega v|^2
+\frac{\min(1,r^2)}2\sum^{n-1}_{k=1} \!\ee_k|\mathcal{M}_\omega \mathcal{N}_\omega^k v|^2\leq
\ee_0|\mathcal{M}_\omega v|^2.$$
Because $|\mathcal M_\omega v|\simeq |\mathcal B_\omega v|,$ using the rescaling \eqref{eq:rescaling}
completes the proof of Lemma \ref{lem:derI}.
\medbreak

We also have the following result, which forms the foundation of Subsection \ref{subsec:V1:hf}.
\begin{lem}    \label{lem:SK:SK'} Let the nonnegative integers $n_1$ and $n_2$ satisfy $n=n_1+n_2.$
Let  $M_{22}\in\cM_{n_2}(\C),$ $N_{11}\in\cM_{n_1}(\C),$ $N_{22}\in\cM_{n_2}(\C),$
$N_{12}\in\cM_{n_1,n_2}(\C)$ and $N_{21}\in\cM_{n_2,n_1}(\C)$ such that ${}^T\!\overline{N_{21}}=N_{12}.$ Let us set
$$ M\defn \begin{pmatrix}
0 & 0\\0 &  M_{22}
\end{pmatrix}\esp{and}  N\defn \begin{pmatrix}
N_{11}& N_{12}\\N_{21} &  N_{22}
\end{pmatrix}\cdotp$$
Finally, let $Q\in\cM_{n_2}(\C)$ satisfy 
\begin{equation}\label{eq:ellip}
 \mathcal{R}e\bigl(Q\eta_2\cdot\eta_2)\neq0\esp{for all}\eta_2\in\C^{n_2}\setminus\{0\}.
\end{equation}
Then, if the pair $(N,M)$ satisfies Condition (SK), so does $(N_{11},N_{12}QN_{21}).$
The converse is true if, furthermore, $M_{22}$ is invertible.
    \end{lem}
\begin{proof} Let $\phi_1\in\C^{n_1}$ satisfy $N_{12}QN_{21}\phi_1=0$ and
$N_{11}\phi_1=\lambda\phi_1$ for some $\lambda\in\C.$
Then, we have
$$0=N_{12}QN_{21}\phi_1\cdot\phi_1=QN_{21}\phi_1\cdot N_{21}\phi_1,$$
whence $N_{21}\phi_1=0,$ due to \eqref{eq:ellip}.
This implies that $\phi\defn(\phi_1,0)\in\C^n$ satisfies $N\phi=\lambda\phi.$
Since obviously $\phi\in\ker M,$ and  as $(N,M)$ satisfies Condition (SK), one may conclude that $\phi=0.$
\smallbreak
Conversely, if $\phi=(\phi_1,\phi_2)\in \C^{n_1}\times\C^{n_2}$ 
satisfies $M\phi=0$ with $M_{22}$ invertible, then we must have $\phi_2=0.$
If we assume in addition that $N\phi=\lambda\phi$ for some $\lambda\in\C,$ then 
we have $N_{11}\phi_1=\lambda\phi_1$ and $N_{21}\phi_1=0,$ and thus $N_{12}QN_{21}\phi_1=0.$
As $(N_{11},N_{12}QN_{21})$ satisfies Condition (SK), we conclude that $\phi_1=0.$
\end{proof}
 \medbreak

The following classical result has been used
a number of times in this text (see the proof in e.g. the Appendix of \cite{DanADOloc}).
\begin{lem}
\label{lem_der_int}
    Let $X : [0, T]\to \mathbb{R}_+$ be a continuous function such that $X$ is differentiable.
    Assume that there exist a constant $B\ge 0$ and a measurable function $A :\mathbb{R}_+\longrightarrow [0, T] $  such that
\[ \frac{1}{2}\frac{d}{dt} X +BX\le AX^{\frac{1}{2}} \esp{a.e on}   [0, T].\]
Then, for all $t\in [0,T]$, we have
\[  X^{\frac{1}{2}}(t)+B\int^t_0 X^{\frac{1}{2}} \le X^{\frac{1}{2}}(0) + \int^t_0 A.\]
\end{lem}


\section{Some properties of Besov spaces}
\label{appendix:LP}
For the reader's convenience, we here recall some results on the Littlewood-Paley decomposition and Besov spaces, the source of which can be found in \cite[Chap. 2]{HajDanChe11}. 
\smallbreak
To define the Littlewood-Paley decomposition, we  fix some smooth radial non increasing function $\chi$ with $ \mathrm{Supp}\,\chi \subset B(0,\frac{4}{3})$  and $\chi \equiv 1$  on $ B(0,\frac{3}{4})$ , then set $\varphi(\xi)=\chi(\frac{\xi}{2})-\chi(\xi)$ so that
\begin{align*}
    \chi+\sum_{j\ge 0} \varphi(2^{-j}\cdot)=1 \;\text{on} \; \Rd \esp{and} \sum_{j\in \mathbb{Z}} \varphi(2^{-j}\cdot) =1  \;\text{on}\; \Rd\backslash\{0\}. 
\end{align*}
For $S$ a function and  $z$ a tempered distribution, we (formally) define   
the pseudo-differential operator $S(D)$ by:
\begin{align}
    \label{def:S(aD)}
    S(D) z\defn \mathcal{F}^{-1}\left( S(\cdot) \mathcal{F}z\right)\cdotp
\end{align}
Next, we introduce the following low frequency cut-off: 
\begin{align}
    \label{def:dot:S:j}
    \begin{split}
   & \Dot{S}_j =\chi(2^{-j}D) \;\; \text{for all}\; j\in \mathbb{Z}\esp{and} S_j\defn\Dot{S}_j \;\; \text{for all}\; j\ge 0,\;\; S_{j}=0 \;\; \text{for all}\; j\le  -1,
     \end{split}
\end{align}
and define the homogeneous dyadic block $\DDj$ and nonhomogeneous dyadic block $\Dj$ as 
\begin{align}
    \label{def:DDj:Dj}
    \begin{split}
    &\DDj \defn \varphi(2^{-j}D) \;\; \text{for all}\; j\in \mathbb{Z} \\
    & \Dj=\DDj  \;\; \text{for all}\; j\ge 0,\;\; \Delta_{-1}= \Dot{S}_0 \esp{and} \Dj=0  \;\; \text{for all}\; j< -1.
    \end{split}
\end{align}
We also consider the subset  $\mathcal{S}'_h$  of tempered distributions $z$ such that 
\begin{align}
    \label{cnd:on S'h}
    \lim_{j \rightarrow -\infty } \Dot{S}_j z=0.
\end{align}
We introduce the homogeneous Besov semi-norms (resp. nonhomogeneous Besov norms):
\begin{align}
    \label{def:NBH:NB}
    \NBH{s}{p}{r}{z} \defn \Vert 2^{js} \Vert\DDj  z\Vert_{L^p}\Vert_{l^r} \quad (\text{ resp. }   \NB{s}{p}{r}{z} \defn \Vert 2^{js} \normep{\DDj z}\Vert_{l^r}).
\end{align}
Then, for any $s\in \mathbb{R}$ and $(p,r)\in [1,\infty]$ we define the homogeneous Besov spaces  $\BH{s}{p}{r}$ (resp. nonhomogeneous Besov spaces $\B{s}{p}{r}$) to be the subset of those $z$ in $\mathcal{S}'_h$ (resp. the subset of those $z$ in the tempered distribution space $\mathcal{S}'$) such that $\NBH{s}{p}{r}{z} $ (resp. $\NB{s}{p}{r}{z}$) is finite.
Any $z\in \mathcal{S}'$  can be decomposed in terms of high and low frequencies parts, 
as follows:
\begin{align}
    \label{def:LF:HF:}
    z^l\defn \sum_{j\le 0} \DDj z \esp{and} z^h \defn \sum_{j> 0} \DDj z.
\end{align}
We constantly used the following Besov semi-norms for low and high frequencies:
\begin{align}
    \label{def:NBH:LF:HF}
    \NBH{s}{2}{1}{z}^l\defn   \sum_{j\le 0} 2^{js}\normede{\DDj z} \esp{and}  \NBH{s}{2}{1}{z}^h\defn   \sum_{j> 0} 2^{js}\normede{\DDj z}.
\end{align}

Even though most of the functions considered here have range in the set of
vectors or matrices, we keep the same notation for Besov spaces pertaining to this case. 
\smallbreak

 In order to bound the commutator terms, we used the following results: 
\begin{propo}
	 \label{propo_commutator-BH}
   The following inequalities hold true:
	   \begin{align}
       \label{comm:est:a:b}
        &\normede{[a,\DDj] b} \le Cc_j2^{-j\sigma} \NB{\cd}{2}{1}{\n  a}\NB{\sigma-1}{2}{1}{b} \!\esp{with}\! \sum_j c_j =1\!\esp{if}\! -d/2 < \sigma \le1+d/2, \\  
       \label{comm:est:a:b:inf}
        &\sup_{j\in{\mathbb Z}}\normede{[a,\DDj] b} \le C2^{-j\sigma} \Vert\n a\Vert_{\Dot{B}^{\cd}_{2,\infty}\cap L^\infty} \Vert b\Vert_{\Dot{B}^{\sigma-1}_{2,\infty}}\quad\hbox{if}
        \ -d/2\le \sigma< 1+d/2.
   \end{align}

	 \end{propo}
  The following product laws in Besov spaces have been used repeatedly.
	 \begin{propo}
	 \label{propo_produc_BH}
	     Let $(s,r)\in ]0,\infty[\times[1,\infty]$. Then $\B{s}{2}{r}\cap L^\infty$ is an algebra and we have 
	     \begin{equation}
	         \label{product_propo1}
	         \NBH{s}{2}{r}{ab} \le C(\normeinf{a}\NBH{s}{2}{r}{b} + \normeinf{b}\NBH{s}{2}{r}{a} ).
	     \end{equation}
	     If, furthermore $-d/2< s \le d/2$, then the following inequalities hold:
	      \begin{equation}
	         \label{product_propo2}
	         \NBH{s}{2}{1}{ab} \le C\NBH{\frac{d}{2}}{2}{1}{a}\NBH{s}{2}{1}{b} 
	     \end{equation}
      and if $  -d/2\le s < d/2 $
       \begin{equation}
	         \label{product:propo:uniq}
	         \NBH{s}{2}{\infty}{ab} \le C\NBH{\frac{d}{2}}{2}{1}{a}\NBH{s}{2}{\infty}{b}.
	     \end{equation}
	 \end{propo}
   Inequality \eqref{product_propo1} is often combined with the embedding $ \BH{\cd}{2}{1}\hookrightarrow L^\infty$. Then for $ s > 0 $ and $1\le r\le \infty$, Inequality \eqref{product_propo1} becomes 
      \begin{equation}
	         \label{product_propo4}
	         \NBH{s}{2}{r}{ab} \le C\NBH{\frac{d}{2}}{2}{1}{a}\NBH{s}{2}{r}{b}+\NBH{s}{2}{r}{a}\NBH{\frac{d}{2}}{2}{1}{b}.
	     \end{equation}

  Among the results necessary to prove  Theorems \ref{thm:glob:cri}, \ref{thm:glo:L2:cri} and \ref{thm:decay:sucri}, we have the following   one.
	 \begin{propo}
	 \label{propo_compo_BH}
 	Let $f$ be a function in $C^\infty(\mathbb{R})$. Then, the following inequalities hold:
    \begin{itemize}
     \item  If $-\cd< s \le \frac{d}{2}$, then
       \begin{align}
          \label{comp:uv:propo:r=1}
         \NBH{s}{2}{1}{f\circ u- f\circ v } \le C(f', \normeinf{u,v})(1+\NBH{\frac{d}{2}}{2}{1}{u}+ \NBH{\frac{d}{2}}{2}{1}{u})\NB{s}{2}{1}{ u-v}.
      \end{align}
 \item If $-\cd\le s < \frac{d}{2},$ then 
      \eqref{comp:uv:propo:r=1}  remains valid for $r=\infty$, that is,
         \begin{align}
          \label{comp:uv:propo:inf}
         \NBH{s}{2}{\infty}{f\circ u- f\circ v } \le C(f', \normeinf{u,v})(1+\NBH{\frac{d}{2}}{2}{1}{u}+ \NBH{\frac{d}{2}}{2}{1}{u})\NBH{s}{2}{\infty}{ u-v}.
      \end{align}
 \item   If  $f(0)=0$ and $s>0$, then
       \begin{align}
           \label{compo:propo:base}
          \Vert  f\circ u\Vert_{\Dot{B}_{2,r}^{s}} \le C(f', \normeinf{u})\ \Vert u \Vert_{\Dot{B}_{2,r}^{s}}.
       \end{align}   \end{itemize}
    	 \end{propo}
  \begin{proof}
      The proof of \eqref{comp:uv:propo:r=1} and \eqref{compo:propo:base} can be found in \cite[pages 94 and 104] {HajDanChe11} while  \eqref{comp:uv:propo:inf} and \eqref{comp:uv:propo:r=1} can be obtained by adapting the proof of first inequality of \cite[page 449]{HajDanChe11}.
  \end{proof}
    Let us finally state some other composition results:
  \begin{propo}
	 \label{propo_Compo_u_v_BH}
  Let $0\le n_1\le n$ be two integers. 	     Let $f: (X,Y)\in \mathbb{R}^{n_1}\times\mathbb{R}^{n-n_1} \mapsto f(X,Y) $ be a smooth function on $\mathbb{R}^n$ vanishing at $0$. Assume that $f$ is linear with respect to $Y$.
  
   Then,  the following inequalities hold true:
  \begin{align}
      \label{comp:u_v:BH:ine:prop:1}
      \NBH{s}{2}{1}{f(u,v)}\le   C(f', \normeinf{u})( \NBH{s}{2}{1}{ v}(1+\NBH{\cd}{2}{1}{u})+\NBH{s}{2}{1}{ u} )\esp{for any} 0<s\le \cd\cdotp
  \end{align}
  Furthermore if $-\cd< s\le \cd$ we have
     \begin{multline}
      \label{comp:u_v_1:2:BH:ine:prop:1}
    \NBH{s}{2}{1}{f(u_1,v_1)-f(u_2, v_2)}\le  C  \NBH{s}{2}{1}{ v_2-v_1}(1+\NBH{\cd}{2}{1}{u_2})\\
      +C(1+\NBH{\cd}{2}{1}{u_1}+\NBH{\cd}{2}{1}{u_2})\Bigl( \NBH{\cd}{2}{1}{u_2-u_1}\NBH{s}{2}{1}{ v_1}+ \NBH{s}{2}{1}{ u_1-u_2}\Bigr)\cdotp
  \end{multline}
    Finally,  if   $ -\cd\le  s< \frac{d}{2}$ then we have
  \begin{multline}
       \label{comp:u_v_1:2:BH:ine:prop:inf}
      \NBH{s}{2}{\infty}{f(u_1,v_1)-f(u_2, v_2)}\le  C  \NBH{s}{2}{\infty}{ v_2-v_1}(1+\NBH{\cd}{2}{1}{u_2})\\
      +C(1+\NBH{\cd}{2}{1}{u_1}+\NBH{\cd}{2}{1}{u_2}) \Bigl(\NBH{s}{2}{\infty}{u_2-u_1}\NBH{\cd}{2}{1}{ v_1}+ \NBH{s}{2}{\infty}{ u_1-u_2}\Bigr)\cdotp
         \end{multline}
  In the above two inequalities, $C= C(f', \normeinf{u_1,u_2}) $.
	 \end{propo}
	\begin{proof}
	  The reader can refer to \cite{DanADOloc} for the proof in the nonhomogeneous Besov spaces case, 
   the adaptation  to our framework being straightforward.
	\end{proof}
 
\bibliographystyle{plain} 
\bibliography{boblio}

\begin{thebibliography}{10}

\bibitem{DanADOloc}
J.-P. Adogbo and R.~Danchin.
\newblock Local well-posedness in the critical regularity setting for hyperbolic systems with partial diffusion.
\newblock {\em arXiv:2307.05981}, 2024.

\bibitem{HajDanChe11}
H.~Bahouri, J.-Y. Chemin, and R.~Danchin.
\newblock {\em Fourier Analysis and Nonlinear Partial Differential Equations}, volume 343.
\newblock Springer, 2011.

\bibitem{Beau11}
K.~Beauchard and E.~Zuazua.
\newblock Large time asymptotics for partially dissipative hyperbolic system.
\newblock {\em Arch. Rational Mech. Anal}, 199:177--227, 2011.

\bibitem{BurCrinTan23}
C.~Burtea, T.~Crin-Barat, and J.~Tan.
\newblock Pressure-relaxation limit for a one-velocity {B}aer-{N}unziato model to a {K}apila model.
\newblock {\em Math. Models Methods Appl. Sci.}, 33(4):687--753, 2023.

\bibitem{coron2007control}
J.-M. Coron.
\newblock {\em Control and nonlinearity}.
\newblock Number 136 in Mathematical Surveys and Monographs. American Mathematical Soc., 2007.

\bibitem{BaratDan22M}
T.~Crin-Barat and R.~Danchin.
\newblock Partially dissipative hyperbolic systems in the critical regularity setting: the multi-dimensional case.
\newblock {\em J. Math. Pures Appl. (9)}, 165:1--41, 2022.

\bibitem{BaratDan22D1}
T.~Crin-Barat and R.~Danchin.
\newblock Partially dissipative one-dimensional hyperbolic systems in the critical regularity setting, and applications.
\newblock {\em Pure and Applied Analysis}, 4(1):85--125, 2022.

\bibitem{BaratDan23}
T.~Crin-Barat and R.~Danchin.
\newblock Global existence for partially dissipative hyperbolic systems in the l p framework, and relaxation limit.
\newblock {\em Mathematische Annalen}, 386(3):2159--2206, 2023.

\bibitem{Dan00}
R.~Danchin.
\newblock Global existence in critical spaces for compressible {N}avier-{S}tokes equations.
\newblock {\em Invent. Math.}, 141(3):579--614, 2000.

\bibitem{Dan01glob}
R.~Danchin.
\newblock Global existence in critical spaces for compressible viscous and heat conductive gases.
\newblock {\em Archive for Rational Mechanics and Analysis}, 160:1--39, 2001.

\bibitem{Danc22}
R.~Danchin.
\newblock Partially dissipative systems in the critical regularity setting, and strong relaxation limit.
\newblock {\em EMS Surv. Math. Sci.}, 9(1):135--192, 2022.

\bibitem{DanMucha24}
R.~Danchin and P.~B. Mucha.
\newblock The compressible euler system with nonlocal pressure: global existence and relaxation.
\newblock {\em Calculus of Variations and Partial Differential Equations}, 63(6):148, 2024.

\bibitem{GioMat13}
V.~Giovangigli and L.~Matuszewski.
\newblock Structure of entropies in dissipative multicomponent fluids.
\newblock {\em Kinet. Relat. Models}, 6(2):373--406, 2013.

\bibitem{Haspot11}
B.~Haspot.
\newblock Existence of global strong solutions in critical spaces for barotropic viscous fluids.
\newblock {\em Archive for Rational Mechanics and Analysis}, 202(2):427--460, 2011.

\bibitem{Kawashima83}
S.~Kawashima.
\newblock {\em Systems of a hyperbolic parabolic type with applications to the equations of magnetohydrodynamics.}
\newblock PhD thesis, Kyoto University, 1983.

\bibitem{KawaSui88}
S.~Kawashima and Y.~Shizuta.
\newblock On the normal form of the symmetric hyperbolic-parabolic systems associated with the conservation laws.
\newblock {\em Tohoku Math. J. (2)}, 40(3):449--464, 1988.

\bibitem{lemarie2023parabolic}
V.~Lemari{\'e}.
\newblock Parabolic-elliptic keller-segel's system.
\newblock {\em Tunisian Journal of Mathematics}, To appear.

\bibitem{LiQin12}
T.~Li and T.~Qin.
\newblock Physics and partial differential equations.
\newblock {\em Higher Education Press}, 1:543--575, 2021.

\bibitem{Madja84}
A.~Majda.
\newblock Compressible fluid flow and systems of conservation laws in several space variable.
\newblock {\em Springer}, 1984.

\bibitem{MatsmuraNishida80}
A.~Matsumura and T.~Nishida.
\newblock The initial value problem for the equations of motion of viscous and heat-conductive gases.
\newblock {\em J. Math. Kyoto Univ.}, 20:67--104, 1980.

\bibitem{QuWang18}
P.~Qu and Y.~Wang.
\newblock Global classical solutions to partially dissipative hyperbolic systems violating the kawashima condition.
\newblock {\em Journal de Math{\'e}matiques Pures et Appliqu{\'e}es}, 109:93--146, 2018.

\bibitem{BianHanou07}
B.~Hanouzet S.~Bianchini and R.~Natalini.
\newblock Asymptotic behavior of smooth solutions for partially dissipative hyperbolic systems with a convex entropy.
\newblock {\em Comm. Pure and Appl. Math.}, 60:1559--1622, 2007.

\bibitem{Serre97}
D.~Serre.
\newblock {\em Syst\`emes de lois de conservation. {I}}.
\newblock Fondations. Diderot Editeur, Paris, 1996.
\newblock Hyperbolicit\'{e}, entropies, ondes de choc.

\bibitem{Serre08cours}
D.~Serre.
\newblock System of conservation laws with dissipation.
\newblock lecture at the SISSA, 2008.

\bibitem{Serr10}
D.~Serre.
\newblock Local existence for viscous system of conservation laws: {$H^s$}-data with {$s>1+d/2$}.
\newblock In {\em Nonlinear partial differential equations and hyperbolic wave phenomena}, volume 526 of {\em Contemp. Math.}, pages 339--358. Amer. Math. Soc., Providence, RI, 2010.

\bibitem{Serre09}
D.~Serre.
\newblock The structure of dissipative viscous system of conservation laws.
\newblock {\em Phys. D}, 239(15):1381--1386, 2010.

\bibitem{SK85}
S.~Shizuta and S.~Kawashima.
\newblock Systems of equations of hyperbolic-parabolic type with applications to the discrete boltzmann equation.
\newblock {\em Hokkaido Math J.}, 14:249--275, 1985.

\bibitem{XinXU21}
Z.~Xin and J.~Xu.
\newblock Optimal decay for the compressible {N}avier-{S}tokes equations without additional smallness assumptions.
\newblock {\em Journal of Differential Equations}, 274:543--575, 2021.

\bibitem{XuKawa15}
J.~Xu and S.~Kawashima.
\newblock The optimal decay estimates on the framework of {B}esov spaces for generally dissipative systems.
\newblock {\em Arch. Rational Mech. Anal}, 218:275--315, 2015.

\bibitem{Zua05}
E.~Zuazua.
\newblock Propagation, observation, and control of waves approximated by finite difference methods.
\newblock {\em SIAM Rev.}, 47(2):197--243, 2005.

\end{thebibliography}
\end{document}